\documentclass[12pt]{amsart}
\usepackage{amscd,verbatim}
\usepackage{amssymb}
\usepackage[all]{xy}
\usepackage{color}

\newcommand{\A}{\mathbf{A}}

\newcommand{\G}{\mathbf{G}}

\newcommand{\N}{\mathbb{N}}
\renewcommand{\P}{\mathbf{P}}
\newcommand{\Q}{\mathbb{Q}}
\newcommand{\Z}{\mathbb{Z}}

\newcommand{\sC}{\mathcal{C}}

\newcommand{\sI}{\mathcal{I}}

\newcommand{\sL}{\mathcal{L}}
\newcommand{\sO}{\mathcal{O}}

\newcommand{\bZ}{\mathbb{Z}}

\newcommand{\Cb}{{\overline{C}}}
\newcommand{\Xb}{{\overline{X}}}

\newcommand{\CuXY}{\sC_{(\Xb,Y)}}

\newcommand{\Cor}{\operatorname{\mathbf{Cor}}}

\newcommand{\HI}{\operatorname{\mathbf{HI}}}

\newcommand{\Rec}{{\operatorname{\mathbf{Rec}}}}

\newcommand{\PST}{{\operatorname{\mathbf{PST}}}}

\newcommand{\Ker}{\operatorname{Ker}}

\newcommand{\Coker}{\operatorname{Coker}}
\newcommand{\Tr}{\operatorname{Tr}}
\newcommand{\Div}{\operatorname{Div}}
\newcommand{\Pic}{\operatorname{Pic}}

\newcommand{\Spec}{\operatorname{Spec}}

\newcommand{\Sm}{\operatorname{\mathbf{Sm}}}
\newcommand{\Sch}{\operatorname{\mathbf{Sch}}}
\newcommand{\Ab}{\operatorname{\mathbf{Ab}}}

\newcommand{\by}[1]{\overset{#1}{\longrightarrow}}

\newcommand{\iso}{\by{\sim}}

\newcommand{\tr}{{\operatorname{tr}}}

\newcommand{\op}{{\operatorname{op}}}
\newcommand{\red}{{\operatorname{red}}}

\newcommand{\Zar}{{\operatorname{Zar}}}
\newcommand{\Nis}{{\operatorname{Nis}}}

\newcommand{\inj}{\hookrightarrow}

\newcommand{\surj}{\rightarrow\!\!\!\!\!\rightarrow}
\newcommand{\Surj}{\relbar\joinrel\surj}

\newcommand{\Jac}{{\operatorname{Jac}}}
\newcommand{\id}{{\operatorname{id}}}

\newcommand{\codim}{{\operatorname{codim}}}
\newcommand{\ch}{{\operatorname{ch}}}
\newcommand{\Sym}{{\operatorname{Sym}}}

\renewcommand{\lim}{\operatornamewithlimits{\varprojlim}}
\newcommand{\colim}{\operatornamewithlimits{\varinjlim}}

\renewcommand{\phi}{\varphi}
\renewcommand{\epsilon}{\varepsilon}
\renewcommand{\div}{\operatorname{div}}

\newcommand{\Zb}{{\overline{Z}}}

\def\indlim#1{\underset{{\underset{#1}{\longrightarrow}}}{\mathrm{lim}}\; }
\def\rmapo#1{\overset{#1}{\longrightarrow}}

\def\bZ{\mathbb{Z}}
\def\Ztr{\bZ_\tr}

\def\rC#1{rel\sC(#1)}
\def\pb{\overline{p}}
\def\Xd#1{X_{(#1)}}

\def\Ud#1{U_{(#1)}}

\newcommand{\til}{\widetilde}

\newcommand{\Vb}{\overline{V}}

\newcounter{spec}
\newenvironment{thlist}{\begin{list}{\rm{(\roman{spec})}}%
{\usecounter{spec}\labelwidth=20pt\itemindent=0pt\labelsep=10pt}}%
{\end{list}}%

\newtheorem{lemma}{Lemma}[subsection]

\newtheorem{thm}[lemma]{Theorem}
\newtheorem{theorem}{Theorem}
\newtheorem{prop}[lemma]{Proposition}
\newtheorem{proposition}[lemma]{Proposition}
\newtheorem{cor}[lemma]{Corollary}
\newtheorem{corollary}[lemma]{Corollary}
\theoremstyle{definition}
\newtheorem{defn}[lemma]{Definition}

\newtheorem{definition}[lemma]{Definition}

\newtheorem{para}[lemma]{}
\theoremstyle{remark}
\newtheorem{qn}{Question}
\newtheorem{rk}[lemma]{Remark}
\newtheorem{remark}[lemma]{Remark}

\newtheorem{claim}[lemma]{Claim}

\newtheorem{conj}{Conjecture}

\begin{document}
\title{Reciprocity sheaves}
\author{Bruno Kahn}
\address{IMJ-PRG\\Case 247\\
4 place Jussieu\\
75252 Paris Cedex 05\\
France}
\email{bruno.kahn@imj-prg.fr}
\author{Shuji Saito}
\address{Interactive Research Center of Science\\
Graduate School of Science and Engineering\\
Tokyo Institute of Technology\\
2-12-1 Okayama, Meguro\\ Tokyo 152-8551\\ Japan}
\email{sshuji@msb.biglobe.ne.jp}
\author[Takao Yamazaki]{Takao Yamazaki\\With two appendices by Kay R\"ulling}
\address{Institute of Mathematics\\ Tohoku University\\ Aoba\\ Sendai 980-8578\\ Japan}
\email{ytakao@math.tohoku.ac.jp}
\address{Freie Universit\"at Berlin, Arnimallee 7, 14195 Berlin, Germany}
\email{kay.ruelling@fu-berlin.de}
\date{\today}
\thanks{The first author acknowledges the support of Agence Nationale de la Recherche (ANR) under reference ANR-12-BL01-0005. 
The second author is supported by JSPS KAKENHI Grant (22340003).
The third author is supported by JSPS KAKENHI Grant (22684001, 24654001).}

\begin{abstract}
We start developing a notion of reciprocity sheaves,
generalizing Voevodsky's 
homotopy invariant presheaves with transfers which were used in the construction of 
his triangulated categories of motives. 
We hope reciprocity sheaves will eventually lead to 
the definition of larger triangulated categories of motivic nature, 
encompassing non homotopy invariant phenomena.
\end{abstract}



\maketitle


\tableofcontents

\section*{Introduction}\label{intro}

In this paper, we start developing a notion of reciprocity sheaves modelled on Voevodsky's theory of presheaves with transfers. Reciprocity is a weaker condition than homotopy invariance, used by Voevodsky as the main building block for constructing his triangulated categories of motives in \cite{voetri}; as in \cite[p. 195]{voetri}, we hope reciprocity sheaves will eventually lead to the definition of  larger triangulated categories of motivic nature, encompassing non homotopy invariant phenomena  which emerged in \cite{MR, iv-ru, KR, KS1, KS2}.


In the whole paper we fix a base field $k$. Our reciprocity sheaves form a full abelian subcategory $\Rec$ of $\PST$,
the abelian category of presheaves with transfers on the category $\Sm$ of smooth schemes over $k$ (see \ref{PST} for the definition).
It contains the subcategory $\HI$ of objects $F\in \PST$ which are homotopy invariant 
(i.e. $F(X)\simeq F(\A^1_X)$ for $X\in \Sm$). It also contains the object represented by a smooth  commutative algebraic group $G$ over $k$ (recall such an object is in $\HI$
if and only if $G^0$ is a semi-abelian variety). Typical examples of objects of $\Rec$ not contained in $\HI$ are
the additive group $\G_a$ and the modules of absolute K\"ahler differentials $\Omega^i$. 

As predecessors of reciprocity sheaves, there were notions of reciprocity functors studied 
in \cite{MR} and \cite{iv-ru}.
All definitions are inspired by the following theorem of Rosenlicht-Serre \cite[Ch. III]{gacl}.

%
%
%

\begin{theorem}\label{thm.RS}
Assume $k$ is algebraically closed. Let $G$ be a smooth connected commutative algebraic group over  $k$ and 
$\alpha: C \to G$ be a morphism where $C$ is a smooth irreducible curve over $k$.
Let $C\hookrightarrow \Cb$ be its smooth compactification. Then there is an effective divisor $D$ on $\Cb$
supported in $\Cb-C$ such that $\alpha$ has modulus $D$
in the sense of Rosenlicht-Serre, 
which means
\[
\underset{x\in C(k)}{\sum}\; v_x(f) \alpha(x) =0
~\text{in}~G(k)\quad\text{ for any } f\in G(\Cb,D),
\]
where $v_x:k(C)^\times \to \bZ$ is the normalized valuation at $x$,
and
\begin{equation}\label{eq.GCD}
G(\Cb,D):= \bigcap_{x\in D} \Ker\big(\sO_{\Cb,x}^\times \to \sO_{D,x}^\times\big).
\end{equation}
\end{theorem}


A distinguished feature of our reciprocity sheaves is their relation to relative Chow groups of zero cycles with moduli
studied in \cite{KS1}:
Let $X$ be a smooth variety over $k$ and choose a compactification
$X\hookrightarrow \Xb$ with $\Xb$ integral and proper over $k$.
For a (not necessarily reduced) closed subscheme $Y\subset \Xb$ with $X=\Xb-Y$, 
the relative Chow group $CH_0(\Xb,Y)$ of zero cycles is the quotient of the group of zero-cycles $Z_0(X)$ on $X$
by ``rational equivalence with modulus $Y$''. More precisely one defines {in \cite{KS1}}
\[
CH_0(\Xb,Y) = \Coker\Big(\Phi(\Xb,Y)(\Spec k)  \rmapo{\partial} Z_0(X)\Big),
\]
where 
\begin{equation}\label{eq1.intro}
\Phi(\Xb,Y)(\Spec k) = \underset{\phi:\Cb\to \Xb}{\bigoplus} \; G(\Cb,\phi^*Y),
\end{equation}
the direct sum being over all finite morphisms $\phi:\Cb\to \Xb$ 
where $\Cb$ is a normal proper curve over $k$ such that $\phi(\Cb)\not\subset Y$,  
$G(\Cb,\phi^*Y)$ is defined as \eqref{eq.GCD} with $\phi^*Y=\Cb\times_{\Xb} Y$, and $\partial$
is induced by the divisor map on $\Cb$ and the pushforward of zero-cycles by $\phi$.

Now the key idea to define reciprocity sheaves is to enhance the abelian group $CH_0(\Xb,Y)$ 
into an object $h(\Xb,Y)$ in $\PST$. It is the cokernel of a map in $\PST$:
\[
 \Phi(\Xb,Y) \to \Ztr(X),
\]
where $\Ztr(X)$ is the object of $\PST$ represented by $X$ and $\Phi(\Xb,Y)\in \PST$ is defined by an analogue of \eqref{eq1.intro}. We have 
\begin{equation}\label{CHh}
h(\Xb,Y)(\Spec k)= CH_0(\Xb,Y).
\end{equation}

To give more details, we introduce a terminology: a pair $(\Xb,Y)$ is called a \emph{modulus pair} if 
$\Xb$ is integral and proper over $k$, $Y\subset \Xb$ is a (possibly non-reduced) closed subscheme and $X:=\Xb-Y$ is smooth and quasi-affine. 
Then, 
for a modulus pair $(\Xb,Y)$ and 
for a section $a\in F(X)$ of a presheaf with transfers $F$, 
we define the notion of $a$ having modulus $Y$,
or $Y$ being a modulus for $a$ (see Def. \ref{def.modulus}),
as a generalization of the modulus 
in the sense of Rosenlicht-Serre (see Thm. \ref{thm.RS}). 

Then $F\in \PST$ is defined to have reciprocity 
(or to be a reciprocity presheaf) 
if for any $a\in F(X)$ with $X$ quasi-affine smooth over $k$ and for any dense open immersion $X\hookrightarrow \Xb$ with $\Xb$ integral proper over $k$, 
there exists a closed subscheme $Y\subset \Xb$ such that $X=\Xb-Y$ and that $Y$ is a modulus for $a$.
The first main result is the following.


\begin{theorem}[see Theorem \protect{\ref{thm:main-rep}}]\label{thm:main-rep.intro}
Let $(\Xb, Y)$ be a modulus pair with $X=\Xb-Y$.
Then the functor
\[ \PST \to \Ab, \quad
F \mapsto \{ a \in F(X) ~|~ a ~
\text{has modulus $Y$} \}\] is represented by an object $h(\Xb, Y)\in \PST$.
If moreover $Y$ is a Cartier divisor on $\Xb$, 
then $h(\Xb, Y)$ has reciprocity.
\end{theorem}

It will become clear from its construction that
$h(\Xb, Y)$ satisfies Formula \eqref{CHh}.
When $\Xb$ is smooth of dimension $1$, $h(\Xb,Y)$ recovers Rosenlicht's generalized Jacobian of \cite[Ch. V]{gacl}, see Proposition \ref{p13.1}.

It is easy to see that reciprocity sheaves form a full abelian subcategory $\Rec$ of $\PST$.
Theorems \ref{thm:HI.intro}, \ref{thm:Algr.intro}  and \ref{kay} show that $\Rec$ contains reasonably many interesting objects of $\PST$.

\begin{theorem}[see Theorem \protect{\ref{thm:HI.intro2}}]\label{thm:HI.intro}
If $F\in \PST$ is homotopy invariant, $F$ has reciprocity.
\end{theorem}

Recall from \cite{spsz} and \cite{bar-kahn} that a presheaf represented by a commutative algebraic group has the structure of a presheaf with transfers.

\begin{theorem}[see Theorem \protect{\ref{thm:alg-gp}}]\label{thm:Algr.intro}
If $F\in \PST$ is represented by a smooth commutative algebraic group, $F$ has reciprocity.
\end{theorem}

In the appendices, Kay R\"ulling proves:

\begin{theorem}[see Theorems \protect{\ref{thm-Omega-RF}} and \protect{\ref{thm-Witt-RF}}] \label{kay}
a) The presheaf with transfers $X\mapsto H^0(X,\Omega^i_X)$ has reciprocity, where $\Omega^i_X$ denotes the sheaf of absolute K\"ahler differentials. If $k$ is perfect, the same is true with $\Omega^i_X$ replaced by relative differentials $\Omega^i_{X/k}$.\\
b) If $k$ is perfect of positive characteristic, the de Rham-Witt presheaves $X\mapsto H^0(X,W_n \Omega^i_X)$ have structures of presheaves with transfers and have reciprocity.
\end{theorem}

An open problem is the following:

\begin{qn}
Is $\Rec$ closed under extensions in $\PST$?
\end{qn}

The next results extend part of Voevodsky's main theorems for homotopy invariant presheaves with transfers to
reciprocity sheaves: 
cf. \cite[Prop. 11.1]{mvw} 
for \ref{thm:inj.intro} and \cite[Th. 22.1, 2.2, 22.15]{mvw} for \ref{thm:ZarNis.intro}.

\begin{theorem}[see Theorem \protect{\ref{thm:inj} and Corollary \ref{cor:inj}}]\label{thm:inj.intro} Let $F\in \PST$ be a reciprocity presheaf.
\begin{enumerate}\item 
Let $X$ be a smooth semi-local $k$-scheme,
and $V \subset X$ an open dense subset.
Then the map $F(X) \to F(V)$ is injective.
\item For an open dense immersion $U \subset X$ in $\Sm$, $F_{\Zar}(X) \to F_{\Zar}(U)$ is injective, where $F_{\Zar}$ is the Zariski sheafification of $F$ as a presheaf.
\item If $F(E)=0$ for any field $E$, then $F_{\Zar}=0$.
\end{enumerate}
\end{theorem}

We note that in this theorem, (2) and (3) are easy consequences of (1) (cf.  \cite[Lemma 22.8 and Cor. 11.2]{mvw}).

\begin{theorem}[see Theorems \protect{\ref{thm:zar-trans}, \ref{thm:zar-rec} and \ref{thm:zar-nis-rec}}] \label{thm:ZarNis.intro} Let $F\in \PST$ be a reciprocity presheaf.
\begin{enumerate}
\item 
$F_{\Zar}$ has a unique structure of presheaf with transfers
such that $F \to F_{\Zar}$ is a morphism in $\PST$.
\item
If $k$ is perfect, $F_{\Zar}$ has reciprocity.
\item 
We have 
$F_{\Zar}=F_{\Nis}$, where $F_{\Nis}$ is the Nisnevich sheafification of $F$ as a presheaf.
\end{enumerate}
\end{theorem}

Theorems \ref{thm:inj.intro} and \ref{thm:ZarNis.intro} may be viewed as the degree $0$ part of the following conjecture:

\begin{conj}\label{conj1}
Suppose $k$ is perfect.
Let $F$ be a reciprocity presheaf which is a sheaf for the Nisnevich topology.
\begin{enumerate}
\item
(Gersten's conjecture.) For any smooth semi-local $k$-scheme $X$ essentially of finite type, the Cousin resolution
\begin{multline*}
0\to F(X)\to \bigoplus_{x\in X^{(0)}}F(x)\to \bigoplus_{x\in X^{(1)}} H^1_{x}(X, F)\to \\
\dots  \to \bigoplus_{x\in X^{(n)}} H^n_{x}(X, F)\to \dots
\end{multline*}
is universally exact in the sense of \cite{grayson}.
\item
The presheaves $H^i_\Zar(-,F)$ and $H^i_\Nis(-,F)$ coincide, and have reciprocity.
\end{enumerate}
\end{conj}

Note that the conjecture is known if $F$ is homotopy invariant 
(\cite[Th. 4.37, 5.6 and 5.7]{voepre}).

As remarked before, there are objects of $\Rec$ which are not homotopy invariant.
But we show the following result in \S \ref{P1inva}.

\begin{theorem}[see Theorem \protect{\ref{thm:P1invariance-pretheory}}] \label{thm:P1invariance}
Let $F$ be a reciprocity presheaf which is separated for the Zariski topology.
Then $F$ is $\P^1$-invariant, namely for any $X\in \Sm$, the projection $p_X:\P^1_X \to X$
induces $p_X^*: F(X) \iso F(\P^1_X)$.
\end{theorem}

This is related to the approach to Gersten's conjecture in \cite[\S 5.4]{cthk} (where the perfectness of $k$ is not required).
Note that $\P^1$-in\-var\-iance is weaker than reciprocity
(Remark \ref{rem:rec-vs-p1inv}).

\subsection*{Acknowledgements} We thank Rin Sugiyama for pointing out a mistake in our initial proof of \eqref{eq.PhiPST}.

\newpage

\numberwithin{equation}{section}

\section{Notation}\label{s1}


\setcounter{subsection}{1}

In the whole paper we fix a base field $k$. 
Let $\Sm$ be the category of separated smooth schemes of finite type over $k$.
For $S\in \Sm$, let $\Sch/S$ be the category of schemes of finite type over $S$. We put $\Sch := \Sch/\Spec k$.

\begin{para}\label{G(C,D)}
For an integral scheme $\Cb$ and a closed subscheme $D\subset \Cb$, we put
\begin{equation}\label{eq:def-of-gcd}
\begin{aligned}
G(\Cb,D) & = 
 \underset{x\in D}{\bigcap} \Ker\big(\sO_{\Cb,x}^\times \to \sO_{D,x}^\times\big)\\
&= \indlim {D\subset U\subset \Cb} \Gamma(U,\Ker(\sO_U^* \to \sO_{D}^*)), \\
\end{aligned}
\end{equation}
where $U$ ranges over open subsets of $\Cb$ containing $D$.
\end{para}

\begin{para}\label{c(X/S)}
Let $S\in \Sm$.
For $X\in \Sch/S$, $c(X/S)$ denotes the free abelian group on the set of  closed integral subschemes of $X$ which are finite over $S$ and surjective over a component of $S$: 
this group is denoted by $C_0(X/S)$ in \cite[\S 3]{sus-voe2} and by $c_{equi}(X/S,0)$ in \cite[\S 2]{voepre}. 

For any morphism $f:T\to S$ in $\Sm$, there is a homomorphism 
\[
f^*:c(X/S)\to c(X\times_S T/T)
\]
called pull-back of cycles (see \cite[\S 3.5]{SV1} and \cite[p. 90]{voepre}, where $f^*$ is denoted by $cycl(f)$; see also \cite[Ex. 1A.12]{mvw}).

For a morphism $f : X \to Y$ in $\Sch/S$, a push-forward 
\[
f_* : c(X/S) \to c(Y/S)
\]
is defined as follows (see \cite[Paragraph before Prop. 2.1]{voepre}):
let $Z \subset X$ be a closed integral subscheme which belongs to $c(X/S)$.
Since it is finite over $S$, its image $f(Z)$ in $Y$ defines an element of $c(Y/S)$,
and $f|_Z : Z \to f(Z)$ is finite and surjective.
We then define $f_*(Z) = \deg(f|_Z) f(Z)$.
\end{para}

\begin{para}\label{PST}
Recall the category $\Cor$ of finite correspondences and the category of presheaves with transfers $\PST$  \cite[\S 2.1, \S 3.1]{voetri}: the objects of $\Cor$ are those of $\Sm$ and for $X,Y\in \Sm$, the group of morphisms 
$\Cor(X,Y)$ is $c(X\times Y/X)$. The category $\PST$ is the category of contravariant functors $\Cor \to \Ab$.

Let $\widetilde{\Sm}$ be the category of
$k$-schemes $X$ which are written as limits 
$X= \lim_{i \in I} X_i$ over a filtered set $I$ 
where $X_i \in \Sm$  
and all transition maps are open immersions.
We frequently allow $F\in \PST$ to take values 
on objects of $\widetilde{\Sm}$ by
$F(X) := \colim_{i \in I} F(X_i)$.

\end{para}

\section{Reciprocity sheaves and representability}\label{representability}


In this section we introduce reciprocity sheaves, our fundamental objects of study in this paper, and prove Theorem \ref{thm:main-rep.intro}.

\subsection{Definition of reciprocity sheaves}\label{s2.1}

\begin{defn}\label{def.moduluspair}
A modulus pair is a pair $(\Xb,Y)$, 
where $\Xb\in \Sch/k$ is integral and proper over $k$ 
and $Y\subset X$ is 
a closed subscheme such that $X=\Xb-|Y|$ is quasi-affine and smooth over $k$,
where $|Y|$ is the support of $Y$.
\end{defn}

Let $(\Xb, Y)$ be a modulus pair with $X=\Xb-|Y|$.
For $S\in \Sm$ we consider commutative diagrams
(which we denote by $(\phi : \Cb \to \Xb \times S)$)
\begin{equation}\label{eq:diag-modulus-cond}
\xymatrix{
&\Xb\\
\Cb \ar[r]^{\phi}\ar[ru]^{\gamma_\phi} \ar[rd]_{p_{\phi}} &
\Xb \times S \ar[d]\ar[u]
\\
& S
}
\end{equation}
where
\begin{enumerate}
\item[(A)]
$\Cb \in \Sch$ is integral normal and $\phi$ is a finite morphism.
\item[(B)]
For some generic point $\eta$ of $S$,  $\dim \Cb \times_S \eta=1$.
\item[(C)]  The image of $\gamma_{\phi}$ is not contained in $Y$.
\end{enumerate}

These conditions imply that $p_{\phi}$ is proper and surjective over a connected component of $S$. 

Let  $G(\Cb,\gamma_{\phi}^*Y)$ be as in \eqref{eq:def-of-gcd} for $D =\gamma_{\phi}^*Y=\phi^*(Y\times S)$.
We will see in \S \ref{sect:const-div-g} that the divisor map on $\Cb$ induces
\begin{equation}\label{eq:div-modulus}
\div_{\Cb}: G(\Cb,\gamma_{\phi}^*Y) \to c(C/S),
\end{equation}
where $C=\Cb-|\gamma_{\phi}^*Y|$ and $c(C/S)$ is defined in \ref{c(X/S)}.

\begin{defn}\label{def.modulus}
Let $F$ be a presheaf with transfers, $(\Xb, Y)$ a modulus pair with $X=\Xb-|Y|$, and $a \in F(X)$.
We say $Y$ is a modulus for $a$ (or $a$ has modulus $Y$) if,
for any diagram \eqref{eq:diag-modulus-cond}, we have
\begin{equation}\label{eq.modulusdef}
(\phi_* \div_{\Cb}(f))^*(a) = 0 ~\text{in}~ F(S) \quad\text{ for any } f\in G(\Cb,\gamma_{\phi}^*Y), 
\end{equation}
where $\phi_* : c(C/S) \to c(X\times S/S)=\Cor(S, X)$ (see \ref{c(X/S)}).
\end{defn}

\begin{definition}\label{def.reciprocity}
We say $F \in \PST$ has {\it reciprocity} 
(or $F$ is a {\it reciprocity presheaf})
if for any quasi-affine $X\in \Sm$, 
any $a \in F(X)$, and any open immersion $X\hookrightarrow \Xb$ 
with $\Xb$ integral proper over $k$, 
$a$ has modulus $Y$ for 
some closed subscheme $Y\subset \Xb$ such that $X=\Xb-|Y|$.
Let $\Rec$ be the full subcategory of $\PST$
consisting of reciprocity presheaves.
\end{definition}

\begin{remark}
It is evident that $\Rec$ is closed in $\PST$ under
taking sub and quotient objects,
so that $\Rec$ is an abelian subcategory of $\PST$.
\end{remark}

We now state the main result of this section.

\begin{thm}\label{thm:main-rep}
Let $(\Xb, Y)$ a modulus pair with $X=\Xb-|Y|$.
\begin{enumerate}
\item
The functor
\[ \PST \to \Ab, \quad
F \mapsto  \{ a \in F(X) ~|~ a ~
\text{has modulus $Y$} \}
\]
is represented by a presheaf with transfers
$h(\Xb, Y)$.
\item
Suppose that $Y$ is a Cartier divisor on $\Xb$.
Then $h(\Xb, Y)$ has reciprocity.
\end{enumerate}
\end{thm}

This theorem is proven as follows.
The object $h(\Xb, Y)$ is constructed 
as the cokernel of a map
$\tau: \Phi(\Xb, Y) \to \Z_{\tr}(X)$ in $\PST$,
which we describe in \S \ref{sect:reformulation}
(see Proposition \ref{prop:rep-sheaf-is-a-pst}).
After proving auxiliary lemmas in \S \ref{sect:aux-lemmas},
we construct $\Phi(\Xb, Y)$ and $\tau$
in \S \ref{sect:pf-of-prop-rep-sheaf-is-a-pst}
and in \S \ref{sect:const-div-g} respectively.
Then (1) becomes obvious from its construction.
The proof of (2) occupies
\S \ref{sect:adm-corr}--\ref{s2.9}.
We introduce a notion of {\it admissible correspondences}
in \S \ref{sect:adm-corr}.
Using a preliminary result proven in \S \ref{sect:inv-prop-g},
we show in \S \ref{secti:functoriality-Phi} a functoriality of $\Phi(\Xb, Y)$
with respect to admissible correspondences,
from which (2) is deduced in \S \ref{s2.9}.

\begin{remark}\label{rem:universality-hxy}
Let $(\Xb, Y)$ and $X$ be as in Theorem \ref{thm:main-rep} (1).
By definition, 
there is a surjection $\pi : \Z_{\tr}(X) \to h(\Xb, Y)$
with the following universal property:
let $F \in \PST$ and $a \in F(X)$,
which we regard as a morphism $a : \Z_{\tr}(X) \to F$
by Yoneda's lemma.
Then $a$ factors through $\pi$ if and only if
$Y$ is a modulus for $a$.
\end{remark}

\begin{cor}\label{coro:reduction-to-cartier}
Let $F \in \PST$. Then $F$ has reciprocity if for any quasi-affine $X \in \Sm$ and any $a \in F(X)$,
there is a modulus pair $(\Xb, Y)$ with $X=\Xb-|Y|$
such that $Y$ is a Cartier divisor on $\Xb$ and $Y$ is a modulus for $a$.
\end{cor}
\begin{proof}
Let $X \in \Sm$ be quasi-affine
and $a \in F(X)$.
Let $(\Xb, Y)$ be a modulus pair given by the hypothesis.
Take any open immersion $X\hookrightarrow \Xb'$ with $\Xb'$ integral proper over $k$.
We need to find a closed subscheme $Y' \subset \Xb'$ such that $|Y'|=\Xb'-X$ and that $Y'$ is a modulus for $a$.
Let $a : \Z_\tr(X) \to F$ be the map in $\PST$ corresponding to $a\in F(X)$.
By the assumption it factors as
\[ \Ztr(X) \overset{\pi}{\to} h(\Xb, Y) \overset{\theta}{\to} F. \]
Noting that $Y$ is a Cartier divisor on $\Xb$, $h(\Xb, Y)$ has reciprocity by Theorem \ref{thm:main-rep}(2).
Hence there exists a closed subscheme $Y' \subset \Xb'$ such that $|Y'|=\Xb'-X$ and that $Y'$ is a modulus 
for $\pi \in h(\Xb, Y)(X)$. This implies that $\pi$ factors through
$\Ztr(X)\to h(\Xb',Y')$ and hence so does $a$, which means $Y'$ is a modulus for $a$.
\end{proof}

\begin{remark}\label{rem:resol}
For any given $X\in\Sm$, we can use Nagata's compactification theorem, then blowup and normalization to find a modulus pair $(\Xb,Y)$ 
such that $\Xb$ is normal and $Y$ is a Cartier divisor.
If $\ch(k)=0$, then $\Xb$ can be taken to be smooth and $\Xb-X$ to be the support of
a normal crossing divisor on $\Xb$.
\end{remark}

%

\subsection{Reformulation of Theorem \ref{thm:main-rep} (1)}
\label{sect:reformulation}
\begin{definition}\label{def:relative-g-sheaf}
Let $(\Xb, Y)$ be a modulus pair.
For $S\in \Sm$, let $\sC_{(\Xb, Y)}(S)$ be the  class of all morphisms $\phi : \Cb \to \Xb \times S$
satisfying Conditions (A), (B), (C) of \S \ref{s2.1}, and put 
(see \eqref{eq:diag-modulus-cond})
\[  \Phi(\Xb, Y)(S) = 
\bigoplus_{(\phi : \Cb \to \Xb \times S) \in \sC_{(\Xb, Y)}(S)} 
G(\Cb, \gamma_{\phi}^* Y) 
\]

\end{definition}

For $( \Cb \by{\phi} \Xb \times S) \in \sC_{(\Xb, Y)}(S)$
and $f \in G(\Cb, \gamma_{\phi}^*Y)$,
we prove that $\phi_* \div_{\Cb}(f)$ belongs to $\Cor(S, X)$
in \S \ref{sect:const-div-g} below,
thereby obtaining a map
\begin{equation}\label{eq:div-g-cor}
 G(\Cb, \gamma_{\phi}^*Y) \to \Cor(S, X).
\end{equation}

Collecting these maps, we get
\begin{equation}\label{eq:map-g-to-cor}
 \tau_S: \Phi(\Xb, Y)(S) \to \Cor(S, X).
\end{equation}
Theorem \ref{thm:main-rep} (1) follows immediately from the following proposition.

\begin{proposition}\label{prop:rep-sheaf-is-a-pst}
The assignment $S \mapsto \Phi(\Xb, Y)(S)$ gives a presheaf with transfers, and the $\tau_S$ define a morphism 
\[
\tau: \Phi(\Xb, Y) \to \Z_{\tr}(X)\;\;\text{in } \PST.
\]
Moreover, 
\[ h(\Xb, Y) := \Coker(\Phi(\Xb, Y) \to \Z_{\tr}(X)) \in \PST \]
represents the functor in Theorem \ref{thm:main-rep} (1).
\end{proposition}

\begin{remark}
Formula \eqref{CHh} is obvious from the definition.
\end{remark}

The proof of this proposition 
is given in \S\S \ref{sect:pf-of-prop-rep-sheaf-is-a-pst}--\ref{sect:const-div-g}.
Before this, we need some preparation which is the object of the next subsection.
Until the end of \S \ref{sect:const-div-g},
we fix a modulus pair $(\Xb,Y)$ and put $X=\Xb-Y$.

\subsection{Auxiliary lemmas}\label{sect:aux-lemmas}
Let $(\phi : \Cb \to \Xb \times S) \in \sC_{(\Xb, Y)}(S)$ with $S$ connected.
Put $C := \Cb - |\gamma_{\phi}^* Y|$.
We consider a condition
for a Cartier divisor $\alpha$ on $\Cb$:
\begin{equation}\label{eq:cartier-div-in-c}
\text{the support $|\alpha|$ of $\alpha$ is contained in $C$.}
\end{equation}

\begin{lemma}\label{lem:alpha}
\begin{enumerate}
\item
If $\alpha$ satisfies \eqref{eq:cartier-div-in-c},
then any irreducible component $D$ of $|\alpha|$
is finite and surjective over $S$.
\item
For any $g \in G(\Cb, \gamma_{\phi}^* Y)$,
$\alpha = \div(g)$ satisfies \eqref{eq:cartier-div-in-c}.
\end{enumerate}
\end{lemma}
\begin{proof}
(1)
Since $X$ is quasi-affine, $C$ is quasi-affine over $S$, 
and hence $D$ is quasi-affine over $S$. 
Since $D$ is closed in $\Cb$, it is proper over $S$ hence finite by Lemma \ref{lqf} below.
Since $D$ is pure of codimension one in $\Cb$, $\dim(D)=\dim(S)$ and $D\to S$ must be surjective, hence (1).
(2) is obvious from definition.
\end{proof}

\begin{lemma}\label{lqf} Let $X,S\in \Sch$. Then a morphism $f:X\to S$ which is quasi-affine and proper is finite.
\end{lemma}

\begin{proof} Factor $f$ as $X\by{j} X'\by{f'} S$, where $j$ is a dense open immersion and $f'$ is affine. Since $f$ is proper, $j$ is proper, hence is an isomorphism. Therefore $f$ is affine and the lemma is well-known.
\end{proof}

Let $Z \in \Sch$ be integral with generic point $\eta$,
and let $f : Z \to S$ be a morphism in $\Sch$.
We denote by $\tilde{\Lambda}(\phi, Z)$
the set of all irreducible components of $\Cb \times_S Z$,
and by $\beta : \Cb \times_S Z \to \Cb$ the projection map.
Define
\begin{equation}\label{eq:def-Lambda}
 \Lambda(\phi, Z) := \{ T \in \tilde{\Lambda}(\phi, Z)
~|~ \dim T_{\eta} =1,~~ T \not\subset |\beta^*(\gamma_{\phi}^*Y)| \},
\end{equation}
where $T_\eta := T \times_Z \eta$.

\begin{lemma}\label{acHI.lem1}
Let $T \in \tilde{\Lambda}(\phi, Z)$
and let $\alpha$ be a Cartier divisor on $\Cb$
satisfying \eqref{eq:cartier-div-in-c}.
(Note that $\beta^*(\alpha)$ is well-defined by Lemma \ref{lem:alpha} (1).)
\begin{enumerate}
\item
Let $s \in Z$ and let $\Sigma$ be 
an irreducible component of $T_s := T \times_Z s$.
\begin{enumerate}
\item
We have
$\dim \Sigma \geq 1$
and $\Sigma \not\subset |\beta^*(\alpha)|$.
\item
If $\Sigma \cap |\beta^*(\alpha)| \not= \emptyset$,
then we have $\dim \Sigma = 1$.
\end{enumerate}
\item
If $T \cap |\beta^*(\alpha)| \not= \emptyset$,
then $T \in \Lambda(\phi, Z)$.
\end{enumerate}
\end{lemma}

\begin{proof}
(1a) 
The first assertion follows from  Chevalley's theorem applied 
at $f(s)$ \cite[13.1.1]{ega4}. 
As all components of $\beta^*(\alpha) \times_S s$ are finite over $s$ 
by Lemma \ref{lem:alpha} (1),
the second assertion follows from the first.

(1b)
On the one hand we have $\dim \Sigma \geq 1$ by (1a).
On the other hand
$\Sigma \cap |\beta^*(\alpha)|$ is 
of codimension $\leq 1$ in $\Sigma$ since 
it is the support of a Cartier divisor on $\Sigma$,
and hence we get $\dim \Sigma \leq 1$ noting $\dim \Sigma \cap
|\beta^*(\alpha)|=0$ by Lemma \ref{lem:alpha} (1). 
This proves (1b).

(2)
The assumption  $T \cap |\beta^*(\alpha)| \not= \emptyset$
and \eqref{eq:cartier-div-in-c}
imply $T \not\subset |\beta^*(\gamma_{\phi}^*Y)|$.
It remains to show $\dim T_\eta = 1$.
Let $W\subset Z$ be the closure of the image of $T\to Z$ 
and $\xi\in W$ be the generic point. 
We separate two cases.

i) Assume $W=Z$ (so that $\xi=\eta$). Then 
$T_\eta \cap |\beta^*(\alpha)| \not= \emptyset$ 
by Lemma \ref{lem:alpha} (1).
Thus we may apply (1b) with $s=\eta$
to conclude $\dim T_{\eta}=1$.

ii) Assume $W \not= Z$.
We show
$T \cap |\beta^*(\alpha)|=\emptyset$.
By Lemma \ref{oesterle} below, we have
$\dim T \geq  \dim Z +1$, 
hence $\dim T_\xi = \dim T -\dim W >1$ since $W\not=Z$.
By Chevalley's theorem this implies that 
for any $s \in W$ and
for any irreducible component $\Sigma$ of $T_s$, 
we have $\dim\Sigma>1$.  
Hence by (1b) we conclude 
$\Sigma \cap |\beta^*(\alpha)| = \emptyset$.
This proves $T \cap |\beta^*(\alpha)| = \emptyset$.
\end{proof}

\begin{lemma} \label{oesterle} Let $S,Z,C\in \Sch$ and let $p:C\to S$, $f:Z\to S$ be two morphisms. Then, for any irreducible component $T$ of $C\times_S Z$, one has
\[\dim T\ge \dim C +\dim Z -\dim S.\]
\end{lemma}

\begin{proof}
\footnote{We thank J. \OE sterl\'e for his help in this proof.} 
Viewing $C\times_S Z$ as a the inverse image in $C\times Z$ of the diagonal $\Delta_S\subset S\times S$ via the projection $C\times Z\to S\times S$, we reduce to show
\[\codim_{C\times Z} T  \le \codim_{S\times S} \Delta_S\]
which easily follows from \cite[Ch. VIII, p. 34, proof of Cor. 4]{bbki}. 
\end{proof}


\subsection{$\Phi(\Xb, Y)$ is a presheaf with transfers.}
\label{sect:pf-of-prop-rep-sheaf-is-a-pst}
For $S, S' \in \Sm$ and $Z\in \Cor(S',S)$, 
we define a homomorphism 
\begin{equation}\label{eq.Z*a}
Z^*:\Phi(\Xb,Y)(S) \to \Phi(\Xb,Y)(S')
\end{equation}
(see Definition \ref{def:relative-g-sheaf}).
We may assume $S, S', Z$ are integral;
the definition is then extended linearly to the general case.

Take $(\phi:\Cb \to \Xb\times S)\in \CuXY(S)$
and $T \in \Lambda(\phi, Z)$.
Let $T^N \to T$ be the normalization and 
consider the composite map
\begin{equation}\label{eq:phi-t}
\phi_T : T^N \to \Cb \times_S Z \rmapo{\phi\times_S Z} \Xb\times Z \to \Xb \times S'.
\end{equation}
Noting that $\phi$ and $Z\to S'$ are finite, so is $\phi_T$ and we have
\begin{equation}\label{eq.phi'}
(\phi_T:T^N \to \Xb\times S') \in \CuXY(S').
\end{equation}
We have a commutative diagram
\[
\xymatrix{
& \Xb& \\
\Cb   \ar[ru]^{\gamma_{\phi}} & \ar[l]_{h_T} T^N \ar[u]_{\gamma_{\phi_T}}  \\
}\]
We define a map
\begin{equation}\label{eq:def-z-t-star}
Z_T^*: G(\Cb,\gamma_{\phi}^*Y) 
\rmapo{h_T^*} 
G(T^N,\gamma_{\phi_T}^*Y) 
\hookrightarrow \Phi(\Xb,Y)(S'),
\end{equation}
where the second comes from \eqref{eq.phi'}. 
To explain $h_T^*$, take $g\in  G(\Cb,\gamma_{\phi}^*Y)$
and put $\Sigma=|\div_{\Cb}(g)|$. 
Note $g\in \Gamma(\Cb\setminus \Sigma,\sO^\times_{\Cb})$ 
since $\Cb$ is normal.
By Lemma \ref{acHI.lem1} (1a), 
$T^N - h_T^{-1}(\Sigma)$ is a dense open subset of $T^N$
and we get 
$h_T^*(g) \in \Gamma(T^N - h_T^{-1}(\Sigma), \sO^\times_{T^N})$.
As $\gamma_{\phi_T}^*(Y) = h_T^*(\gamma_{\phi}(Y))$,
we find
$h_T^*(g) \in G(T^N, \gamma_{\phi_T}^*(Y))$.

Letting $m_T$ be the multiplicity of $T$ in $\Cb \times_S Z$, we then define
\[
Z^*= \underset{T \in \Lambda(\phi, Z)}{\sum} \; m_T \cdot Z_T^*\;:\;
G(\Cb,\gamma_{\phi}^*Y) \to \Phi(\Xb,Y)(S')
\]
which induces the desired map \eqref{eq.Z*a}. 


Let $S, S', S'' \in \Sm$ and $Z \in \Cor(S', S)$, $Z' \in \Cor(S'', S')$.
We need to show
\begin{equation}\label{eq.PhiPST}
{Z'}^* Z^* = (ZZ')^* : 
\Phi(\Xb,Y)(S) \to \Phi(\Xb,Y)(S'')
\end{equation}
so that $\Phi(\Xb,Y)$ is an object of $\PST$.
To verify \eqref{eq.PhiPST}, one may suppose $S, S', Z$ and $Z'$ are integral. Let $W_1, \dots, W_r$ be the irreducible components of 
$Z\times_{S'} Z'\subset S\times S''$ 
and $n_j$ the multiplicity of $W_j$.
Then we have $ZZ' = \sum_j n_j W_j$ 
(see \cite[Proof of Lemma 4.1.15]{deglise}).
Consider the following two subclasses of $\CuXY(S'')$
(where $\Lambda(-,-)$ is as in \eqref{eq:def-Lambda} and $\phi_{T'}$ is as in \eqref{eq:phi-t}):
\begin{align*}
&\Xi_1=\{ \phi_{T'} 
~|~ T' \in \Lambda(\phi_T, Z')~ \text{for some}~ T \in \Lambda(\phi, Z) \},
\\
&\Xi_2=\underset{1\leq j\leq r}{\bigcup}\; 
\{ \phi_{T'} 
~|~ T' \in \Lambda(\phi, W_j) \}.
\end{align*}
Note that in the definition of $\Xi_2$, the sum is disjoint since 
the morphism $T'\to W_j$ for $T' \in \Lambda(\phi, W_j)$ is surjective. 
For $T \in \Lambda(\phi, Z)$ and
$T' \in \Lambda(\phi_T, Z')$
(so that $\phi_{T'} \in \Xi_1$),
we have a commutative diagram
\begin{equation}\label{eq:Xidiagram}
\xymatrix{
T' \ar[r]^{\hskip -20pt\hookrightarrow} \ar[d] & T^N\times_{S'} Z' \ar[d] \\
\tilde{T'}  \ar[r]^{\hskip -20pt\hookrightarrow} & T\times_{S'} Z'\ar[r]^{\hskip -20pt\hookrightarrow} 
& \Cb\times_S Z\times_{S'} Z'&\Cb\times_S W_j \ar[l]_{\hskip 20pt\hookleftarrow}\\
}\end{equation}
where $\tilde{T'}$ is the image of $T'$ in $T\times_{S'} Z'$. 
It suffices to prove the following:
\begin{enumerate}
\item[(i)]
There is an inclusion $\Xi_2\hookrightarrow \Xi_1$.
\item[(ii)]
For 
$\phi_{T'} \in  \Xi_1$, 
if there is a Cartier divisor $\alpha$ on $\Cb$ such that $|\alpha|\subset C$ 
and that $\tilde{T'}\cap (|\alpha|\times_S Z\times_{S'} Z')\not=\emptyset$, 
then 
$\phi_{T'} \in \Xi_2$.
\item[(iii)]
Let 
$T \in \Lambda(\phi, Z),~ T' \in \Lambda(\phi_T, Z'),~
 T'' \in \Lambda(\phi, W_j), ~1 \leq j \leq r$
and suppose
$\phi_{T'} \in \Xi_1$ corresponds to $\phi_{T''} \in \Xi_2$ by (i).
Then we have $m_T m_{T'} = n_j m_{T''}$.
\end{enumerate}
We show (i). Take $\phi_{T''}$ from $\Xi_2$.
The canonical map ${T''}^N \to \Cb$ must factor through the normalization ${T}^N$ of 
some irreducible component $T$ of $\Cb \times_{S} Z$.
We claim $T \in \Lambda(\phi, Z)$
(see \eqref{eq:def-Lambda}),
which yields a desired correspondence.
It is obvious $T \not\subset |\beta^*(\gamma_{\phi}^*Y)|$.
If $\dim T_{\eta} \not= 1$,
Lemma \ref{acHI.lem1} (1a)
shows that $\dim T_{\eta} > 1$ 
and hence we get $\dim {T''}_{\eta} > 1$ 
by Chevalley's theorem \cite[13.1.1]{ega4},
which contradicts $T'' \in \Lambda(\phi, W_j)$.
Thus we get $\dim T_{\eta} = 1$, and (i) is proved.

Next we show (ii). In view of \eqref{eq:Xidiagram}, there is $j\in [1,r]$ and an irrreducible component of $T''$
of $\Cb\times_S W_j$ such that the map $T' \to \Cb\times_S Z\times_{S'} Z'$ factors through $T''$. 
By the assumption we have $T''\cap (|\alpha|\times_S W_j)\not=\emptyset$. 
By Lemma \ref{acHI.lem1} (2), this implies $T''\in \Lambda(\phi, W_j)$.

Finally (iii) can be seen by a computation:
\[
\begin{split}
m_T m_{T'} &=
l(\mathcal{O}_{T \times_{S'} Z', \eta_{T'}})
l(\mathcal{O}_{\Cb \times_{S} Z, \eta_{T}})
\\
&=
l(\mathcal{O}_{Z \times_{S'} Z', \eta_{W_j}})
l(\mathcal{O}_{T \times_{Z} W_j, \eta_{T''}})
l(\mathcal{O}_{\Cb \times_{S} Z, \eta_{T}})
\\
&=
l(\mathcal{O}_{Z \times_{S'} Z', \eta_{W_j}})
l(\mathcal{O}_{\Cb \times_{S} W_j, \eta_{T''}})
= n_j m_{T''},
\end{split}
\]
where we denote by $l(R)$ 
the length of an Artin local ring $R$,
and by $\eta_V$ the generic point of an integral scheme $V$.


\begin{remark}
Here is an example 
where
the inclusion $\Xi_2 \hookrightarrow \Xi_1$ in (i) is not surjective.
\footnote{
This example is communicated to us by R. Sugiyama.
}
Let $\Xb = \P^1 \times \P^1$ and 
$Y:=\infty \times \P^1 \cup \P^1 \times \infty$.
Let $B$ be the blow-up of $S:=\A^2$ at $0$,
regarded as a closed subscheme of $\P^1 \times S$.
Denote by $i : B \hookrightarrow \P^1 \times S$ the inclusion map.
Put $\Cb := \P^1 \times B$
and define $\phi = \id_{\P^1} \times i : \Cb \to \Xb \times S$.
Then $\phi$ defines an element of $\CuXY(S)$.
Set $S=\A^2 \supset Z=S'=0 \times \A^1
\supset Z'=S''=0 \times 0$.
Let $E \subset B$ be the exceptional curve,
$L \subset B$ the strict transform of $S'$,
and $p := L \cap E$.
Then $\Xi_2 = \emptyset$,
because
$\Cb \times_S S'' = \P^1 \times E$ does not belong to 
$\Lambda(\phi, Z')$ \eqref{eq:def-Lambda}.
On the other hand,
we have $\Xi_1 = \{ \phi_{\P^1 \times p} \}$ 
because 
$\Lambda(\phi, Z)=\{ \P^1 \times L \}$
and
$\Lambda(\phi_{\P^1 \times L}, Z')=\{ \P^1 \times p \}$.
\end{remark}

\subsection{Construction of $\tau$}
\label{sect:const-div-g}
Let $S \in \Sm$,
 $(\phi : \Cb \to \Xb \times S) \in \sC_{(\Xb, Y)}(S)$
and put 
$C = \phi^{-1}(X \times S)$.
Lemma \ref{lem:alpha} shows that
$\div_{\Cb}(g)$ satisfies \eqref{eq:cartier-div-in-c}
for any $g\in G(\Cb, \gamma_\phi^*(Y))$.
Hence the divisor map on $\Cb$ induces \eqref{eq:div-modulus}.
We obtain the map $\tau_S$ in \eqref{eq:map-g-to-cor} by composing \eqref{eq:div-modulus} with 
\[ \phi_* : c(C/S) \to \Cor(S, X)=c(X \times S/S). \]

To show $\tau$ is a morphism in $\PST$,
we need to show that the following diagram commutes
for any $Z \in \Cor(S', S), ~S, S' \in \Sm$:
\begin{equation}\label{eq.PhiCor}
\xymatrix{
\Phi(\Xb,Y)(S) \ar[r]^{\tau_S} \ar[d]^{Z^*} & \Cor(S,X)\ar[d]^{Z^*}\\
\Phi(\Xb,Y)(S') \ar[r]^{\tau_{S'}} &  \Cor(S',X).}
\end{equation}
As $\Cor(S',X) \to \Cor(S' \times k',X \times k')$ is injective 
for any extension $k'/k$, we may suppose $k$ is perfect.
Also note that $\Cor(S',X) \to \Cor(V,X)$ is injective for any 
open dense immersion $V \to S'$.
After such a base change, 
we may assume $Z$ is regular and hence $Z \in \Sm$
by the assumption that $k$ is perfect.
By \eqref{eq.PhiPST} we may assume $Z$ is either 
(i) the transpose of a finite surjective morphism $S=Z \to S'$
or (ii) the graph of a morphism $f : S'=Z \to S$.
In the case of (i), we have $c(C/S) = c(C/S')$ 
and the statement becomes trivial.
We consider the case (ii).
By shrinking $Z$ further,
$f:Z \to S$ can be written as the composite of 
a flat map and 
regular immersions of codimension one in $S$. 
Again by \eqref{eq.PhiPST}, 
we may assume $f$ is one such morphism.
Here we present a proof 
of the commutativity of \eqref{eq.PhiCor}
assuming $f$ is a regular immersion of codimension one in $S$.
We omit the proof for the case $f$ is flat,
as it can be shown by a similar (and much easier) way.

Let $(\phi:\Cb \to \Xb\times S) \in \sC_{(\Xb, Y)}(S)$.
Then the desired assertion will follow from the commutativity of the diagram
%
%
%
%
%
%
%
%
\[
\xymatrix{
G(\Cb,\gamma_{\phi}^*Y) \ar[r]^{\div_{\Cb}}\ar[d]^{\sum h_T^*} & 
c(C/S) \ar[r]^{\phi_*} \ar[d]^{f^*} & 
c(X\times S/S)\ar[d]^{f^*} 
\\
\displaystyle\bigoplus_{T \in \Lambda(\phi, Z)} 
G(T^N,\gamma_{\phi_T}^*Y)\ar[dr]_{\bigoplus \div_{T^N}} &
c(C_Z/Z) \ar[r]^{\phi_*} & 
c(X\times Z/Z) 
\\
&\displaystyle\bigoplus_{T \in \Lambda(\phi, Z)} c(T^{0}/Z) 
\ar[u]_{\sum m_T (i_T)_*} 
 }
\]
where 
$C = \Cb - |\gamma_\phi^* Y|$, $T^{0} := T^N - |\gamma_{\phi_T}^* Y|$,
$C_Z=C\times_S Z$, and $i_T:T^0\to C_Z$ is the natural map.
(See \eqref{eq:def-z-t-star} for $h_T^*$
and \S \ref{c(X/S)} for $f^*$).
The commutativity of the right squares is obvious. 
To prove that of the left pentagon,
we use the following facts:
\begin{enumerate}
\item
For a Cartier divisor $\iota: D \hookrightarrow C$, we have an identity :  
\[
D\cdot [C_Z] = C_Z\cdot [D]\;\;\text{ as cycles on $D\cap C_Z$},
\]
where $[C_Z]$ (resp. $[D]$) is the cycle on $C$ associated to the Cartier divisor $C_Z$ (resp. $D$)
and $D\cdot -$ (resp. $C_Z\cdot -$) is the intersection product (see \cite[Thm.2.4]{fulton}). 
\item
If $D=\div_C(g)$ for $g\in k(C)^\times$ and $W\subset C$ is integral closed such that $W\not\subset |D|$, we have
\[ 
D\cdot [W] = \div_W(g_{|W}) = \pi_*\div_{W^N}(g_{|W^N}),
\]
where $\pi:W^N \to W$ is the normalization of $W$.
\end{enumerate}
Take $g \in G(\Cb, \gamma_{\phi}^*Y)$.
We have $\div_{\Cb}(g)=\div_C(g)$
and
\[
f^* \div_{\Cb}(g) = 
C_Z \cdot [\div_{\Cb}(g)]
\overset{(1)}{=} 
\div_{\Cb}(g) \cdot [C_Z].
\]
For each $T \in \Lambda(\phi, Z)$, we have
\[ \div_{\Cb}(g) \cdot [T^0] 
\overset{(2)}{=} \div_{T^0}(g|_{T^0})
= \div_{T^N}(h_T^*(g)).
\]
Noting $C_Z = \sum_{T \in \Lambda(\phi, Z)} m_T T^0$
as a divisor on $C$,
we get the desired commutativity.

The last claim of Proposition \ref{prop:rep-sheaf-is-a-pst} (representability) is now tautologically true.
This completes the proof of Proposition \ref{prop:rep-sheaf-is-a-pst}.
\qed

\subsection{Admissible correspondences}\label{sect:adm-corr}
In order to prove Theorem \ref{thm:main-rep} (2),
we need to show a functoriality of 
$\Phi(\Xb, Y)$ with respect to modulus pairs $(\Xb, Y)$: this will be done in Proposition \ref{prop:admissible-induce-mor}. In this subsection, we introduce a notion of admissible correspondences, and prove their existence in suitable cases.

For two closed subschemes $Y_1, Y_2$ in a scheme $\Xb$, we write $Y_1 \geq Y_2$
if the inclusion $Y_2 \to \Xb$ factors through $Y_1 \to \Xb$
(equivalently, if for all $x \in \Xb$ one has
$\sI_{Y_1, x} \subset \sI_{Y_2, x}$,
where $\sI_{Y_i}$ is the ideal sheaf of $Y_i$).

\begin{definition}\label{def:modulus-pair}
Let $M_i = (\Xb_i, Y_i) ~(i=1, 2)$ be modulus pairs
and put $X_i = \Xb_i \setminus |Y_i|$.
Let $Z \in \Cor(X_1, X_2)$ be an integral finite correspondence.
We write $\bar Z^N$ for the normalization
of the closure of $Z$ in $\Xb_1 \times \Xb_2$
and $p_i : \bar Z^N \to \Xb_i$
for the canonical morphisms for $i=1, 2$.
We say $Z$ is \emph{admissible} for $(M_1, M_2)$
if $p_1^* Y_1 \geq p_2^* Y_2$.
An element of $\Cor(X_1, X_2)$ is called admissible
if all of its irreducible components are admissible.
\end{definition}

\begin{lemma}\label{lem:admissible-induce-mor}
Let $Z \in \Cor(X_1, X_2)$, where $X_1, X_2 \in \Sm$ are quasi-affine.
For $i=1, 2$, let $\Xb_i \in \Sch$ be a normal proper $k$-scheme
which contains $X_i$ as an open dense subset.
Let $Y_2 \subset \Xb_2$ be a closed subscheme
supported on $\Xb_2 - X_2$.
Then there exists a closed subscheme
$Y_1 \subset \Xb_1$ supported on $\Xb_1 - X_1$
such that $Z$ is admissible
for $((\Xb_1, Y_1), (\Xb_2, Y_2))$. 
\end{lemma}
\begin{proof}
We may assume $Z\subset X_1\times X_2$ integral and finite surjective over $X_1$.
Let $\Zb\subset \Xb_1\times\Xb_2$ be the closure of $Z$ and $p_i:\Zb\to \Xb_i ~(i=1, 2)$  be 
the natural maps. We remark 
\begin{equation}\label{lem:admissible-induce-mor.eq}
p_1^{-1}(\Xb_1-X_1)_\red= \Zb-Z \;\;\text{ and }\;\;p_2^{-1}(Y_2)_{\red}\subset \Zb-Z.
\end{equation}
The first assertion follows from Lemma \ref{l1.1} below noting that $Z\to X_1$ is finite.
The second assertion is obvious since $p_2$ induces $Z\to X_2$. 
Hence one can find a closed subscheme $Y_1 \subset \Xb_1$ supported on $\Xb_1-X_1$ such that $p_1^*Y_1\geq p_2^*Y_2$ on $\Zb$,\footnote{Note that if $F\inj W$ is a closed immersion in  $\Sch$ defined by the ideal sheaf $\sI$, there exists $n>0$ such that $\sI\supseteq \sI_0^n$ where $\sI_0$ is the ideal sheaf of $F_\red$.
} and therefore on $\Zb^N$.  
\end{proof}

\begin{lemma}\label{l1.1} Let
\[\begin{CD}
X@>j>> \bar X\\
@V{p}VV @V{\bar p}VV\\
Y@>j'>> \bar Y
\end{CD}\]
be a commutative diagram of schemes, where $j,j'$ are  dense open immersions and $p$ is proper. Then $\bar p^{-1}(Y) = X$.
\end{lemma}
\begin{proof} Let $Z=\bar p^{-1}(Y)$. The dense open immersion $X\inj Z$ is proper hence an isomorphism.
\end{proof}

\subsection{An invariance property for $G(\Cb,D)$}\label{sect:inv-prop-g}

\begin{prop}\label{lem:pullback}
Let  $f:\Cb' \to \Cb$ be a proper surjective morphism of $k$-varieties and $D \subset \Cb$ be an effective
 Cartier divisor.
Assume $f_*\sO_{\Cb'}=\sO_{\Cb'}$.
Then, for $D'=\Cb' \times_\Cb D$, we have   
\[ 
f^*: G(\Cb, D) \iso G(\Cb', D').
\]
\end{prop}

For the proof, we need a lemma.

\begin{lemma} \label{l3.2.1} Let $f:\Cb' \to \Cb$ be a proper surjective morphism of schemes.
Let $D \subset \Cb$ be a closed subscheme. Then the system
\[\{f^{-1}(U)\mid U\supset D, U \text{ open}\}\]
is cofinal among the open neighbourhoods of $f^{-1}(D)$ in $\Cb'$.
\end{lemma}

\begin{proof} Let $U'$ be an open neighbourhood of $f^{-1}(D)$, and $Z=\Cb'-U'$. Then $f(Z)$ is closed and $f^{-1}(D)\subset f^{-1}(U)\subseteq U'$ with $U=\Cb-f(Z)$.
\end{proof}


\begin{proof}[Proof of Proposition \ref{lem:pullback}] 
By Lemma \ref{l3.2.1}, it suffices to show 
\[\sO_{U|D}^*\iso f_* \sO_{U'|D'}^*\]
for $U$ running through the open neighbourhoods of $D$ in $\Cb$ and $U'=f^{-1}(U)$. In the commutative diagram of exact sequences
\[\begin{CD}
0@>>> \sO_{U|D}^* @>>> \sO_{U}^* @>>> \sO_D^*\\
&&@VaVV @VbVV @VcVV\\
0@>>> f_*\sO_{U'|D'}^* @>>> f_*\sO_{U'}^* @>>> f_*\sO_{D'}^*
\end{CD}\]
$b$ is an isomorphism by the assumption $f_*\sO_{\Cb'}=\sO_{\Cb}$.
Hence it suffices to show that $c$ is injective. For this, it suffices to show that $c'$ is injective in the diagram 
\[\begin{CD}
0@>>> \sI_D @>>> \sO_{U} @>>> \sO_D@>>> 0\\
&&@Va'VV @Vb'V{\simeq}V @Vc'VV\\
0@>>> f_*\sI_{D'} @>>> f_*\sO_{U'} @>>> f_*\sO_{D'}
\end{CD}\]
which will follow from the surjectivity of $a'$.
To see this, note
\[ f^*\sI_D=\sI_D\otimes_{\sO_U} \sO_{U'} \iso \sI_D \sO_{U'}=\sI_{D'}\]
by the assumption that $D$ is a Cartier divisor. Thus we get
\[f_*\sO_{U'}\otimes\sI_D\simeq f_*(f^*\sI_D) \iso f_*\sI_{D'}\]
by the projection formula, 
hence the claim since 
$\sO_{U}\iso f_*\sO_{U'}$.
\end{proof}

\subsection{Functoriality of $\Phi(\Xb, Y)$}\label{secti:functoriality-Phi}

\begin{proposition}\label{prop:admissible-induce-mor}
Let $M_i = (\Xb_i, Y_i) ~(i=1, 2)$ be modulus pairs
with $X_i = \Xb_i \setminus |Y_i|$,
and let $Z \in \Cor(X_1, X_2)$ be admissible for $(M_1, M_2)$. We assume that $Y_2$ is a Cartier divisor.
Then there is a morphism
\[
Z_* : \Phi(\Xb_1, Y_1) \to \Phi(\Xb_2, Y_2)\;\;\text{in }\PST
\]
which fits into a commutative diagram
\begin{equation}\label{hXYrec.eq1}
\xymatrix{
\Phi(\Xb_1, Y_1) \ar[r]^{\tau} \ar[d]^{Z_*} 
& \Z_{\tr}(X_1) \ar[d]^{Z_*}
\\
\Phi(\Xb_2, Y_2) \ar[r]^{\tau} & \Z_\tr(X_2)
}
\end{equation}
so that $Z_*$ induces a map
$h(\Xb_1,Y_1) \to h(\Xb_2,Y_2)$.
\end{proposition}

\begin{proof}
We may assume $Z\subset X_1\times X_2$ integral and finite surjective over $X_1$.
Let $\Zb\subset \Xb_1\times\Xb_2$ be the closure of $Z$ and $p_i:\Zb\to \Xb_i ~(i=1, 2)$  be 
the natural maps. 
Take $S\in \Sm$ connected and 
\[
(\phi:\Cb\to \Xb_1\times S)\in \sC_{(\Xb_1, Y_1)}(S)\quad\text{ with }\;C=\phi^{-1}(X_1\times S).
\] 
Let $C_{Z,i}$ ($i\in I$) be the irreducible components of $C\times_{X_1} Z$ and
$\Cb_{Z,i}$ be the closure of $C_{Z,i}$ in $\Cb\times_{\Xb_1} \Zb$ and $\Cb_{Z,i}^N$ be its normalization. 
We have the composite map
\[
\psi_i':  \Cb_{Z,i}^N \to \Cb\times_{\Xb_1} \Zb \rmapo{\phi\times_{\Xb_1} \Zb} \Zb\times S \to \Xb_2\times S.
\]

By construction, $\psi_i'$ is proper. Let 
\begin{equation}\label{eq.Stein}
\Cb_{Z,i}^N \rmapo{p_{2,i}} \Cb_{X_2,i} \rmapo{\psi_i} \Xb_2\times S
\end{equation}
be its Stein factorization. Note that $\Cb_{X_2,i}$ is normal and $\psi_i$ is finite.
Let $J\subset I$ be the subset of those $i\in I$ such that $\dim(\Cb_{X_2,i}\times_S \eta)=1$,
where $\eta$ is the generic point of $S$.
For $i\in J$ we have
\[
(\psi_i: \Cb_{X_2,i} \to \Xb_2\times S)\in \sC_{(\Xb_2, Y_2)}(S).
\]

Indeed it suffices to check that the image of $\psi'_i$ is not contained in $Y_2\times S$.
The normalization $C_{Z,i}^N$ of $C_{Z,i}$ is dense open in $\Cb_{Z,i}^N$ and $\psi_i'$ induces a morphism
\[
C_{Z,i}^N \to C\times_{X_1} Z \rmapo{\phi\times_{\Xb_1} \Zb} Z\times S \to X_2\times S,
\]
which implies the desired assertion.
We obtain a commutative diagram
\[
\xymatrix{
\Cb  \ar[d]_{\gamma_{\phi}} 
& \ar[l]_{p_{1,i}}\Cb^N_{Z,i} \ar[r]^{p_{2,i}} 
\ar[d]^{\gamma_{Z,i}}
& \Cb_{X_2, i}  \ar[d]^{\gamma_{X_2, i}} \\
\Xb_1  & \ar[l]_{p_1} \Zb \ar[r]^{p_2} & \Xb_2.\\
}\]

From this we get a composition 
\begin{multline*}
Z^{i}_*: G(\Cb,\gamma_{\phi}^* Y_1) 
\to G(\Cb^N_{Z,i},p_{1,i}^*\gamma_{\phi}^*Y_1)=G(\Cb^N_{Z,i},\gamma_{Z,i}^*p_1^*Y_1)  \\
\inj
G(\Cb^N_{Z,i},\gamma_{Z,i}^*p_2^*Y_2) 
\simeq G(\Cb_{X_2, i},\gamma_{X_2, i}^* Y_2) 
\subset \Phi(\Xb_2,Y_2)(S),
\end{multline*}
where the first map is the pullback by $p_{1,i}$,  the middle inclusion comes from the assumption $p_1^*Y_1 \geq p_2^*Y_2$, 
and the isomorphism is the inverse of the pullback along $p_{2,i}$,
 which is an isomorphism by Proposition \ref{lem:pullback} 
(here we use the assumption that $Y_2$ is a Cartier divisor; also, $p_{2,i}$ satisfies the assumption of loc. cit. by the construction \eqref{eq.Stein}).
Letting $m_i$ be the multiplicity of $C_{Z,i}$ in $C\times_{X_1} Z$, we then define
\[
Z_*=\underset{i\in J}{\sum}\; m_i Z^{i}_* :
G(\Cb, \gamma_{\phi}^* Y_1)\to \Phi(\Xb_2,Y_2)(S)
\]
which induces $Z_* : \Phi(\Xb_1,X_1)(S) \to \Phi(\Xb_2,Y_2)(S)$.
 It is easy to check {(essentially as a special case of \eqref{eq.PhiPST})} that this induces a map $Z_* : \Phi(\Xb_1,Y_1) \to \Phi(\Xb_2,Y_2)$ in $\PST$.

To prove the commutativity of \eqref{hXYrec.eq1}, we may assume $S=\Spec K$ for a field $K$ by the same reason as for the commutation of \eqref{eq.PhiCor}.
Then $\Cb$ and $\Cb^N_{Z,i}$ are regular of dimension one so that $p_{1,i}$ is flat.
We are reduced to showing the commutativity of the diagram
\[
\xymatrix{
G(\Cb, \gamma_{\phi}^* Y_1) 
\ar[rr]^{\div_{\Cb}}
\ar[d]^{\oplus m_i p_{2, i *} p_{1, i}^*} &&
c(C/K)
\ar[r]^{\phi_*}
\ar[d]^{\oplus m_i p_{2, i *} p_{1, i}^*} &
c(X_1  \times K/K) 
\ar[d]^{Z_*}
\\
\displaystyle\bigoplus_{i \in J} G(\Cb_{X_2, i}, \gamma_{X_2, i}^* Y_2)
\ar[rr]^(0.6){\oplus \div_{\Cb_{X_2, i}}} && 
\displaystyle\bigoplus_{i \in J} c(C_{X_2, i}/K)
\ar[r]^{\oplus \psi_{i *}} & 
c(X_2 \times K/K),
} \]
where $p_{1,i}^*$ is the flat pullback of cycles.
Each square of the diagram is easily seen to be commutative.
This completes the proof of Proposition \ref{prop:admissible-induce-mor}.
\end{proof}

\subsection{Proof of Theorem \ref{thm:main-rep} (2)}\label{s2.9}

Let $(\Xb,Y)$ be as in the theorem.
Assume we are given a connected quasi-affine $V\in \Sm$ and $a\in h(\Xb,Y)(V)$.
Let $\Vb$ be a proper normal scheme over $k$ containing $V$ as a dense open subset.
Let $\tilde{a}\in \Ztr(X)(V)=\Cor(V,X)$ be a lift of $a$ under the canonical surjection
$\Ztr(X)(V)\to h(\Xb,Y)(V)$. We have a commutative diagram
\[
\xymatrix{
\Ztr(V) \ar[r]^{\tilde{a}_*} \ar[rd]_{a_*} & \Ztr(X) \ar[d]\\
&h(\Xb,Y)\\
}\]
where $a_*$ and $\tilde{a}_*$ are respectively induced by $a$ and $\tilde{a}$ via the Yoneda embedding.
By Lemma \ref{lem:admissible-induce-mor} and Proposition \ref{prop:admissible-induce-mor}, there is a closed subscheme $W\subset \Vb$ supported on $\Vb-V$ such that $a_*$ factors through $\Ztr(V) \to h(\Vb,W)$, which implies that $a$ has modulus $W$.\qed


\section{Homotopy invariance implies reciprocity}

\subsection{Introduction} In this section we prove Theorem \ref{thm:HI.intro}.
Actually, we prove the following stronger result:

\begin{thm}\label{thm:HI.intro2}
Let $F \in \PST$. 
We consider the following condition:
\begin{enumerate}
\item[$(\diamondsuit)$]
If $X\in \Sm$  and $a \in F(X)$, then $Y$ is a modulus for $a$ for any  modulus pair  $(\Xb,Y)$ with $X=\Xb-Y$. (Equivalently:  for any  modulus pair  $(\Xb,Y)$ with $Y$ reduced, $Y$ is a modulus for $a$.)
\end{enumerate}
Then we have the following.
\begin{enumerate}
\item
If $F$ is homotopy invariant,
then $(\diamondsuit)$ holds.
\item
 If $F$ is separated for the Zariski topology and satisfies $(\diamondsuit)$,
then $F$ is homotopy invariant.
\end{enumerate}
\end{thm}

For the proof, we need some preliminary results
on relative Picard groups and relative Suslin homology,
which will occupy \S \ref{ssect:rel-pic} and \S \ref{sect:rel-sus-hom}.
The proofs of (1) and (2) will be given in
\S \ref{sect:pf-hi-one} and \S \ref{sect:pf-hi-two} respectively.

\subsection{Relative Picard group}\label{ssect:rel-pic}
We recall [SV, \S 2-3] with some modifications.
Let $\Xb$ be an integral scheme and $Y$ a closed subscheme of $\Xb$.
We denote by $\Pic(\Xb, Y)$ the group of all isomorphism classes of pairs $(L, \sigma)$ of an invertible sheaf $L$
and an isomorphism $\sigma : L|_Y \cong \sO_Y$.
Note that $\Pic(\Xb)=\Pic(\Xb, \emptyset)$. 
We have an exact sequence
\[ \Gamma(\Xb, \sO_X^\times) \to 
   \Gamma(Y, \sO_Y^\times) \to 
 \Pic(\Xb, Y) \to \Pic(X) \to \Pic(Y)
\]
and an isomorphism
\[ \Pic(\Xb, Y) \cong H^1_\Zar(\Xb, \G_{\Xb, Y}) \]
where $\G_{\Xb, Y} := \ker(\sO_\Xb^\times \to \sO_Y^\times)$.
An element $(L, \sigma)$ of $\Pic(\Xb, Y)$ is called liftable if there exist a pair 
$(U, \tilde{\sigma})$ consisting of an open subset $U \subset \Xb$ and an isomorphism $\tilde{\sigma} : L|_U \cong \sO_U$ 
such that $Y \subset U$ and $\tilde{\sigma}|_Y = \sigma$. We define 
$\til{\Pic}(\Xb, Y)$ to be the subgroup of $\Pic(\Xb, Y)$ consisting of liftable elements.
Let $\Div(\Xb, Y)$ be the group of Cartier divisors on $\Xb$
whose support does not intersect with $Y$.


\begin{lemma}\label{lem:rel-pic-sv}
Let $\Xb$ and $Y$ be as above.
\begin{enumerate}
\item
We have an exact sequence
\[ 0 \to \Gamma(\Xb, \G_{\Xb, Y}) \to  G(\Xb, Y) \rmapo{\div_{\Xb}} \Div(\Xb, Y) \to \til{\Pic}(\Xb, Y) \to 0.
\]
\item
If $\Xb$ is normal and $Y$ is reduced, 
 the pullback by the projection $p:X\times\A^1 \to X$
induces isomorphisms
\begin{equation}\label{eq1.lem:rel-pic-sv}
p^*:\Pic(\Xb, Y)\simeq \Pic(\Xb\times\A^1, Y\times\A^1),
\end{equation}
\begin{equation}\label{eq2.lem:rel-pic-sv}
p^*: \til{\Pic}(\Xb, Y)\simeq \til{\Pic}(\Xb\times\A^1, Y\times\A^1).
\end{equation}
\item If $Y$ has an affine open neighbourhood in $\Xb$, then $\til{\Pic}(\Xb, Y)=\Pic(\Xb, Y)$.
\end{enumerate}
\end{lemma}
\begin{proof}
The assertions are shown in \cite[2.3, 2.5]{sus-voe2} except for \eqref{eq2.lem:rel-pic-sv}.
which we deduce from \eqref{eq1.lem:rel-pic-sv}. Let
\[
i^*: \Pic(\Xb\times\A^1, Y\times\A^1)\to \Pic(\Xb,Y)
\]
be the pullback along the zero section $i:\Spec k \to \A^1$. By definition 
\[i^*(\til{\Pic}(\Xb\times\A^1, Y\times\A^1))\subset \til{\Pic}(\Xb,Y)\]
and $i^*p^*$ is the identity. Hence \eqref{eq2.lem:rel-pic-sv} follows from the fact
$\Ker(i^*)=0$ which follows from \eqref{eq1.lem:rel-pic-sv}.
\end{proof}

\subsection{Relative Suslin homology}\label{sect:rel-sus-hom}
We take a connected $S \in \Sm$ and $\Xb \in \Sch/S$. Let $Y \subset \Xb$ be a closed subscheme.
Suppose $\Xb$ is normal and set $X :=\Xb - Y$. Let $c(X/S)$ be as \ref{c(X/S)}.
We define for each $n \in \Z_{\geq 0}$
\[ C_n(X/S) := c(X \times \Delta^n/S \times \Delta^n) \]
where $\Delta^n = \Spec k[t_0, \dots, t_n]/(\sum t_i -1)$ is the standard cosimplicial scheme over $k$.
The complex $C_*(X/S)$ of abelian groups thus obtained
is called the \emph{relative Suslin complex}.
Its  homology group is denoted by
$H^S_*(X/S)$ and called the \emph{relative Suslin homology}.

Now suppose that the generic fiber of $X \to S$ is one-dimensional and that $X$ is quasi-affine over $S$.
By definition the components of $Z \in C_n(X/S)$ are closed in $\Xb \times \Delta^n$ so that $Z$ is a Weil divisor on 
$\Xb \times \Delta^n$. Let $\til{C}_n(X/S)\subset C_n(X/S)$ be the subgroup of all
$Z \in C_n(X/S)$ which are Cartier divisors on $\Xb \times \Delta^n$.  We have
\begin{equation}\label{eq:til-c-div}
  \til{C}_n(X/S) = \Div(\Xb \times \Delta^n, Y \times \Delta^n).
\end{equation}
As a pull-back of a Cartier divisor is Cartier, $\til{C}_*(X/S)$ is a subcomplex of $C_*(X/S)$.
Its homology groups are denoted by $\til{H}^S_*(X/S)$.

\begin{thm}
Assume  $Y$ reduced.
Then we have
\[
\til{H}_n^S(X/S) \cong
\left\{
\begin{array}{ll}
\til{\Pic}(\Xb, Y) &(n=0),
\\
0 &(n>0).
\end{array}
\right.
\]
\end{thm}
\begin{proof}
By Lemma \ref{lem:rel-pic-sv} (1)
and \eqref{eq:til-c-div}, we have
an exact sequence of complexes of abelian groups
\[
\begin{split}
0 &\to 
\Gamma(\Xb \times \Delta^*, \G_{\Xb \times \Delta^*, Y \times \Delta^*})
\to
G(\Xb \times \Delta^*, Y \times \Delta^*)
\\
&\to
\til{C}_n(X/S)
\to
\til\Pic(\Xb \times \Delta^*, Y \times \Delta^*)
\to 0.
\end{split}
\]
The assumption that $X$  is quasi-affine over $S$ is preserved under any base change {\cite[(5.1.10) (iii)]{ega2}}, hence remains true after passing to the geometric fibers. This shows that $Y$ meets every component of these fibers, which implies as in \cite[Proof of Th. 3.1]{sus-voe2} that the first term vanishes:
$\Gamma(\Xb \times \Delta^*, \G_{\Xb \times \Delta^*, Y \times \Delta^*})=0$. 
It is also proved in loc. cit. that
$G(\Xb \times \Delta^*, Y \times \Delta^*)$
is acyclic.
Now the theorem follows from
Lemma \ref{lem:rel-pic-sv} (2).
\end{proof}
\begin{corollary}\label{cor:vanishing-of-g-in-h}
Under the same assumption as in the previous theorem,
the following composition is zero:
\[ G(\Xb, Y) \to \til{C}_0(X/S) \to \til{H}_0^S(X/S).
\]
\end{corollary}

\subsection{Proof of Theorem \ref{thm:HI.intro2} (1)}\label{sect:pf-hi-one}

Suppose $F\in \PST$ is homotopy invariant.
Let $(\Xb,Y)$ and $a \in F(X)$ be as in $(\diamondsuit)$.
We show $Y$ is a modulus for $a$.
Let $S \in \Sm$ and $(\phi : \Cb \to \Xb_S) \in \sC_{(\Xb, Y)}(S)$.
Put $Y_{\Cb}=\phi^*(Y_S), ~C=\Cb - Y_{\Cb}$.
We need to show that the composition
\[(*) \qquad 
G(\Cb, Y_{\Cb}) \to c(C/S) \rmapo{\phi_*} c(X_S/S)=\Ztr(X)(S) \overset{a}{\to} F(S) \]
is zero. 
By 
\cite[Lemma 7.5]{mvw} 
the homotopy invariance of $F$ implies that the last map factors as
\[ c(X_S/S) \to H_0^S(X_S/S) \to F(S). \]
Therefore, the map $(*)$ factors as
\[ 
\begin{split}
G(\Cb, &Y_{\Cb}) \subset G(\Cb, (Y_{\Cb})_\red) 
\to \til{C}_0(C/S) \to \til{H}_0^S(C_S/S)
\\
&\to H_0^S(C_S/S)
\to H_0^S(X_S/S)
\to F(S),
\end{split}
\]
which is zero by Corollary \ref{cor:vanishing-of-g-in-h}.
\qed

\subsection{Proof of Theorem \ref{thm:HI.intro2} (2)}\label{sect:pf-hi-two}
This proof is adapted from \cite[Prop. 3.11]{voepre}.
 Suppose $F \in \PST$ is separated for Zariski topology and satisfies $(\diamondsuit)$,
and take $S \in \Sm$. 
Let $p : S \times \A^1 \to S$ be the projection and $i : \Spec k \inj \A^1$ be the $0$-section.
We must show $p^* : F(S) \to F(S \times \A^1)$
is an isomorphism, which will follow from the injectivity of 
$(\id_S \times i)^*: F(S \times \A^1) \to F(S)$ (see proof of Lemma \ref{lem:rel-pic-sv}).
Since $F \in \PST$ is separated for Zariski topology, we may assume $S$ is affine.
Consider the commutative diagram
\[
\xymatrix{
S \times \A^1 
\ar[r]^(0.35){\id_S \times \delta}
&
(S \times \A^1) \times \A^1
\ar[r]^(0.65){p \times \id_{\A^1}}
&
S \times \A^1
\\
& S \times \A^1
\ar[r]^{p}
\ar[u]^{\id_{S \times \A^1} \times i}
& S
\ar[u]^{\id_{S} \times i}
}
\]
Since 
$(p \times \id_{\A^1})(\id_S \times \delta)=\id_{S \times \A^1}$,
it suffices to show
\[
(*) \quad
 (\id_S \times \delta)^* = (\id_{S \times \A^1} \times i)^*
 : F((S \times \A^1) \times \A^1) \to F(S \times \A^1).
\]
Take a proper integral variety $\overline{S}$
which contains $S$ as an open dense subscheme.
Put $X:=(S \times \A^1) \times \A^1$,
$\Xb := (\overline{S} \times \P^1) \times \P^1$
and $Y := (\Xb-X)_{\red}$
so that $(\Xb, Y)$ is a modulus pair  (here we used the assumption that $S$ is affine).
Let $\Cb :=(S \times \A^1) \times \P^1$
and let $\phi : \Cb \to \Xb \times (S \times \A^1)$
be the morphism defined by the natural inclusion $\Cb \inj \Xb$
and the projection $\Cb \to S \times \A^1$.
Then 
$\phi$ is an element of $\sC_{(\Xb, Y)}(S \times \A^1)$
(cf. \ref{def:relative-g-sheaf}).
Since $Y$ is a modulus for 
any $a \in F(X)$ by $(\diamondsuit)$,
we get
$(\phi_* \div_{\Cb}(f))^*(a)=0$ for
any $f \in G(\Cb, Y|_{\Cb})$
where
$Y|_{\Cb} = 
\phi^* (Y \times (S \times \A^1)) = (S \times \A^1) \times \infty$.
If we write $x, y$ for the coordinates of $\P^1 \times \P^1$,
then the function $f=1-\frac{x}{y}$
belongs to $G(\Cb, Y|_{\Cb})$,
and
$\phi_* \div_{\Cb}(f) \in \Cor(S \times \A^1,(S \times \A^1) \times \A^1)$
agrees with the difference of
the graphs of 
$\id_S \times \delta$ and $\id_{S \times \A^1} \times i$.
This proves $(*)$ and hence completes the proof.
\qed

\begin{remark}\label{r3.6}
Let $(\Xb, Y)$ be a modulus pair with $Y$ reduced.
Put $X=\Xb-|Y|$.
Let $h_0(X)$ be the presheaf with transfers
introduced in \cite[p.207]{voetri},
which is characterized as the maximal homotopy invariant
quotient of $\Z_{\tr}(X)$.
By Theorem \ref{thm:HI.intro},
we get a (surjective) map $h(\Xb, Y) \to h_0(X)$.
It is an interesting problem
to know when this is an isomorphism: 
{it is true (after Zariski sheafification)
if $\dim X=1$ by Theorem \ref{p13.1} below.}
\end{remark}

%
%

\section{Algebraic groups have reciprocity}

\subsection{Introduction} In this section we show the following result by adapting an argument of [GACL].

\begin{thm}\label{thm:alg-gp}
Any smooth commutative group scheme $G$ locally of finite type over $k$,
regarded as a presheaf with transfers, has reciprocity.
\end{thm}

The proof goes as follows.
In \S \ref{sect:first-reduction},
we reduce the proof to the case
$k$ is algebraically closed and $G$ connected.
We then prove Theorem \ref{thm:alg-gp}
for characteristic zero in \S \ref{sect:ch-zero},
and 
for characteristic $p>0$ in \S \ref{sect:ch-p}.

In what follows we identify $G$ with the object of $\PST$ represented by $G$ 
(cf.  \cite[Proof of Lemma 3.2]{spsz} and \cite[Lemma 1.3.2]{bar-kahn}).
Before going to the proof of \ref{thm:alg-gp}, 
we make a simple remark which will be used frequently 
in this paper.

\begin{remark}\label{rem:reduction-to-curve}
Assume  $k$ is perfect.
Let $(\Xb, Y)$ be a modulus pair with $X=\Xb-Y$, and $F \in \PST$, 
and $a \in F(X)$.
Assume that 
$F$ has global injectivity,
which means that
for any dense open immersion $j:U\hookrightarrow X$ in $\Sm$,
$j^*:F(X)\to F(U)$ is injective.
Then, in order to show that $Y$ is a modulus for $a$, it suffices to verify the following condition
(cf. \eqref{eq.modulusdef}):

Let $K=k(S)$ be the function field of a connected $S\in \Sm$, $\Cb$ a normal integral proper curve over $K$ and 
$\phi : \Cb \to \Xb \times K$ a finite morphism such that $\phi(\Cb) \not\subset Y \times K$. 
Put $C=\phi^{-1}(X\times K)$ and let $\psi:C\to X $ be the induced map.
Since $C$ is regular and $k$ is perfect,
we have $C \in \widetilde{\Sm}$, 
whence $\psi^* : F(X) \to F(C)$ (cf. \ref{PST}).
Then
\begin{equation*}
(\div_{\Cb}(g))^*(\psi^*(a)) = 0 ~\text{in}~ F(K) \quad\text{ for any } g\in G(\Cb,D), 
\end{equation*}
where $D = \phi^*(Y\times K)$ and 
$\div_{\Cb}(g)\in c(C/K)$ is viewed as an element of $\Cor(K,C)$ by the map $C\to C\times K$
induced by the identity on $C$ and the projection $C\to \Spec K$.

\end{remark}

\subsection{Reduction to algebraically closed and connected cases}
\label{sect:first-reduction}
Let $\bar k$ be an algebraic closure of $k$.
We assume $G_{\bar k}$ has reciprocity (over $\bar k$),
and prove $G$ has reciprocity (over $k$).
Take $X\in \Sm$ connected quasi-affine, $a\in G(X)$ and $\Xb$ a compactification of $X$:  we must find a closed subset $Y\subset \Xb$ with support $\Xb-X$ which is a modulus for $a$. By hypothesis, there is such a $Y\subset \Xb_{\bar k}$ for the image $a_{\bar k}$ of $a$ in $G(X_{\bar k})$ (more accurately, this is true component by component in case $X$ is not geometrically connected). We may choose $Y$ such that $\sI_Y=\sI_{(\Xb_{\bar k}-X_{\bar k})_\red}^n$ for some $n>0$, which implies that $Y$ is defined over $k$. We claim that $Y$ is the desired modulus for $a$. Indeed,  let $S\in \Sm$ and $\Cb\in \sC_{(\Xb,Y)}(S)$:  with the notation of \ref{def.modulus}, we obviously have  for $f\in G(\Cb,\gamma_{\phi}^*Y)$
\[\left((\phi_* \div_{\Cb}(f))^*(a)\right)_{\bar k}=  ((\phi_{\bar k})_* \div_{\Cb_{\bar k}}(f_{\bar k}))^*(a_{\bar k})=0\in G(S_{\bar k})\]
hence the claim since $G(S)\to G(S_{\bar k})$ is injective.

Now assume $k$ is algebraically closed. 
We reduce the proof of Thm. \ref{thm:alg-gp} to the case $G$ connected. Indeed, $G$ then splits as a direct product $G^0\times D$, where $G^0$ is the connected component of $0$ and $D$ is discrete; in particular, the choice of such a splitting yields  a translation isomorphism $\tau_\delta:G^\delta\iso G^0$ for any $\delta\in D$. By construction of the action of finite correspondences, a finite correspondence of degree $d$ in $\Cor(S,X)$ maps $G^\delta(X)$ to $G^{d\delta}(S)$ for any $\delta\in D$; in particular, $(\phi_* \div_{\Cb}(f))^*(a)$ maps $G^\delta(X)$  to $G^0(S)$. Since the action of $(\phi_* \div_{\Cb}(f))^*(a)$ clearly commutes with $\tau_\delta$, reciprocity for $G^0$ implies reciprocity for $G$.

\emph{From now on until the end of this section, 
we assume $k$ algebraically closed and $G$ connected.}

\subsection{Proof of Thm. \ref{thm:alg-gp} in characteristic zero}\label{sect:ch-zero}
We take an integral quasi-affine $X \in \Sm$ and $a \in G(X)=\text{Mor}_k(X,G)$.
Let $\Xb$ be a proper smooth integral variety
that contains $X$ as an open dense subset
and $W:=(\Xb-X)_{\red}$ be a normal crossing divisor
(Remark \ref{rem:resol}).
In view of Corollary \ref{coro:reduction-to-cartier},
it suffices to find $n \in \Z_{>0}$
such that $(\Xb, nW)$ is a modulus for $a$.

Let $\{ \omega_1, \dots, \omega_d \}$ be a basis of
the space of invariant differential forms on $G$. Noting that 
$
\indlim n \Gamma\big(\Xb, \Omega_{\Xb/k}^1\otimes_{\sO_{\Xb}}\sO_{\Xb}(nW)\big)=
\Gamma\big(X, \Omega_{X/k}^1 \big)$, take $n>0$ such that 
\[
a^*\omega_i \in \Gamma\big(\Xb, \Omega_{\Xb/k}^1\otimes_{\sO_{\Xb}}\sO_{\Xb}(nW)\big)
\quad\text{ for all }i=1,\dots, d.
\]

We claim that $Y:=nW$ is a modulus for $a$.
For this, it suffices to verify the condition
in Remark \ref{rem:reduction-to-curve},
but it follows from the following result from [GACL].
\qed

\begin{proposition}[GACL, III, Prop. 10]
Let $G$ be a commutative algebraic group 
over a field $K$ of characteristic zero,
and let $\{ \omega_1, \dots, \omega_d \}$ be a basis of
the space of invariant differential forms on $G$.
Let $\Cb$ be a proper normal curve over $K$,
$C$ an open dense subscheme of $\Cb$,
and $a : C \to G$ a morphism.
Let $D$ be an effective divisor on $\Cb$
such that $|D|=\Cb-C$  
and that
\[
a^*\omega_i \in \Gamma\big(\Cb, \Omega_{\Cb/K}^1\otimes_{\sO_{\Cb}}\sO_{\Cb}(D)\big)
\quad\text{ for all }i=1,\dots, d.
\]
Then we have
$(\div_{\Cb}(g))^*(a) = 0\in G(K)$ for any $g\in G(\Cb,D)$, 
where $a$ is viewed as a section over $C$ of the object of $\PST$ represented by $G$.
\end{proposition}

\subsection{Proof of Thm. \ref{thm:alg-gp} in positive characteristic}\label{sect:ch-p}

In this case, the following lemma is proved in [GACL, III, \S 7, 9]:

\begin{lemma}\label{lem:alg-gp-decomp}
There exist 
a semi-abelian variety $G_1$,
a unipotent group $U$ and
a homomorphism
$\theta : G \to G_1 \times U$
with finite kernel.
\end{lemma}
Here we include a sketch of the proof taken from [GACL].
Let $A$ be the maximal abelian quotient of $G$,
and put $L=\ker(G \to A)$.
Write $L=L_u \times L_m$ where $L_u$ is a unipotent group
and $L_m$ is a torus.
Put $G_1 := G/L_u $, which is semi-abelian.
Take an $N \in \Z_{>0}$ such that $p^N L_u =0$
and put $U := (G/L_m)/p^N$, which is unipotent.
Define $\theta$ to be the product of
the canonical surjective maps
$G \to G_1$ and $G \to U$.
One checks $\ker(\theta)$ is finite.
\qed

\begin{proposition}[cf. [GACL, III, Prop. 14]
\label{prop:finite-kernel}
Let $\theta : G \to G'$ be 
a morphism of smooth connected commutative algebraic groups
such that $\Ker(\theta)$ is finite.
Let $(\Xb, Y)$ be a modulus pair
such that $Y$ is a Cartier divisor on $\Xb$.
Put $X=\Xb-Y$ and let $a \in G(X)$.
If $(\Xb, Y)$ is a modulus for $\theta(a) \in G'(X)$,
then $(\Xb, Y)$ is a modulus for $a \in G(X)$.
\end{proposition}
\begin{proof}
Since $G$ satisfies global injectivity, 
it suffices to check the condition in
Remark \ref{rem:reduction-to-curve}.
We use the notations therein.

Let $g \in G(\Cb, D)$ be non-constant and $g : \Cb \to \P^1_K$ be the corresponding morphism.
Let $U_g = \P^1_K - g(\Cb - C) \subset \P^1_K$ 
and $C_g = g^{-1}(U') \subset C$.
Then $C_g$ is finite over $U_g$ and we let $[g] \in \Cor(U_g, C)$ denote 
the finite correspondence given by $C_g$, 
which induces $[g]_* : G(C) \to G(U_g)$.
It is proved in [GACL, III, Prop. 9] that,
for any effective divisor $D$ on $\Cb$
such that $|D|=\Cb-C$ and $b \in G(C)$,
the following conditions are equivalent:
\begin{enumerate}
\item
$(\div_{\Cb}(g))^*(b) = 0\in G(K)$ for any $g\in G(\Cb,D)$.
\item
For any non-constant $g \in G(\Cb, D)$, $[g]_*(b)\in G(U_g)$, 
viewed as a morphism $U_g \to G$, is constant.
\end{enumerate}

Thus we are reduced to showing (2) for $D=\phi^*(Y\times K)$ and $b=\phi^*(a)$ with $\phi: C\to X\times K$.
Let us take a non-constant $g \in G(\Cb, D)$.
Since $\theta(a) \in G'(X)$ has modulus $(\Xb, Y)$ by assumption, (1) holds for $\theta(b)\in G'(C)$ instead of 
$b\in G(C)$. Hence the morphism $[g]_*(\theta(b)) : U_g \to G'$ is constant.
Note that $[g]_*(\theta(b))$ factors as
\[  U_g \overset{[g]_*(b)}{\longrightarrow} G
   \overset{\theta}{\longrightarrow}  G'.
\]
Since $\ker(\theta)$ is finite and $U_g$ is connected,
$[g]_*(b)$ must be constant too.
\end{proof}

\begin{proposition}\label{prop:unip}
If $G$ is unipotent, it has reciprocity.
\end{proposition}

\begin{proof}
We take integral quasi-affine $X \in \Sm$ and $a \in G(X)$.
We also take a modulus pair $(\Xb, W)$
such that $\Xb$ is normal,
that $W$ is an effective Cartier divisor on $\Xb$,
and that $X=\Xb - |W|$
(Remark \ref{rem:resol}).
In view of Corollary \ref{coro:reduction-to-cartier},
it suffices to find $n \in \Z_{>0}$
such that $(\Xb, nW)$ is a modulus for $a$.

We fix an embedding $G \to GL_r$.
Then our section $a \in G(X)$ can be represented as
a matrix $(a_{ij})_{i, j=1}^r \in GL_r(X)$,
where $a_{ij} \in \sO(X)$ are regular functions on $X$.
Take such $n>0$ that 
\[
a_{ij} \in \Gamma\big(\Xb, \sO_{\Xb}(nW)\big)
\quad\text{ for all }i,j=1,\dots, r.
\]
We claim that $Y=nW$ is a modulus for $a$.
For this, it suffices to verify the condition
in Remark \ref{rem:reduction-to-curve},
but it follows from the following result from [GACL].
\end{proof}

\begin{proposition}[GACL, III, Prop. 15]
Let $G$ be a connected unipotent commutative 
algebraic subgroup of $GL_r$.
Let $\Cb$ a proper normal curve over $K$,
$C$ an open dense subscheme of $\Cb$,
and $a : C \to G$ a morphism.
Write $a_{ij} \in \sO(C)$ for the $(i, j)$-entry of 
$a \in G(C) \subset GL_r(C)$.
Let $D$ be an effective divisor on $\Cb$
such that $|D|=\Cb-C$
and that 
\[
a_{ij} \in \Gamma\big(\Cb, \sO_{\Cb}(D)\big)
\quad\text{ for all }i,j=1,\dots, r.
\]
Then we have
$(\div_{\Cb}(g))^*(a) = 0\in G(K)$ for any $g\in G(\Cb,D)$, 
where $a$ is viewed as a section over $C$ of the object of $\PST$ represented by $G$.
\end{proposition}

\begin{proof}[Proof of Thm. \ref{thm:alg-gp} in positive characteristic]
Lemma \ref{lem:alg-gp-decomp} and Proposition \ref{prop:finite-kernel}
reduces us to the cases of
a semi-abelian variety and of a unipotent group.
We have proved the theorem for unipotent groups in 
Proposition \ref{prop:unip}.
A semi-abelian variety is homotopy invariant,
hence has reciprocity by Theorem \ref{thm:HI.intro}.
\end{proof}


\begin{remark}\label{rem:cheating}
Let $G$ be a commutative algebraic group over an arbitrary field $k$.
Let $\Cb$ be a smooth projective curve over $k$,
$D$ an effective divisor on $\Cb$, $C:=\Cb - D$
and $a \in G(C)$.
If $D$ is a modulus for $a$ in the sense of Rosenlicht-Serre
(cf. Thm. \ref{thm.RS}), 
then $D$ is a modulus for $a$ in our sense.
Indeed,
by Remark \ref{rem:reduction-to-curve} 
it suffices to show that,
for any field $K \in \widetilde{\Sm}$,
the image $a_K$ of $a$ in $G(C \times K)$
has modulus $D \times K$.
But this follows from the proof in this section.
\end{remark}

\section{Weak reciprocity}\label{weakrec}

In this section we introduce the notion of ``weak reciprocity".
It is weaker than reciprocity defined in \ref{def.reciprocity} but strong enough to imply the injectivity property 
stated in Theorem \ref{thm:inj.intro}.
An advantage of reciprocity in Def. \ref{def.reciprocity}
over weak reciprocity is that the analogue of 
Thm. \ref{thm:main-rep.intro} fails for weak reciprocity
(Remark \ref{rem:adv-of-rec-over-wrec}).
Weak reciprocity can be formulated for pretheories, a more general notion than presheaves with transfers (see \ref{def.pretheory}).
Theorem \ref{thm:inj.intro} follows from Theorem \ref{thm:inj}.

\subsection{Definition of weak reciprocity}

\begin{definition}\label{def.relcurve} 
For $S\in \Sm$, $\rC S$ is the class of the morphisms $p:X \to S$ in $\Sm$ 
which are quasi-affine and equidimensional of relative dimension $1$. 
We sometimes write $X/S$ for $p:X \to S$.
\end{definition}

\begin{para}\label{pretheory-pairing} 
For $S\in \Sm$ and $(p:X\to S)\in \rC S$, we have a map
\begin{equation}\label{eq.cCor}
\psi_{X/S}: c(X/S) \to \Cor(S,X)=c(X\times S/S)
\end{equation}
induced by $X \rmapo{id_X\times p} X\times S$.
For $F\in \PST$ we define a pairing
\begin{equation}\label{eq.prethpairing}
 \langle \hbox{ }, \hbox{ } \rangle_{X/S} : c(X/S) \times F(X) \to F(S),
\end{equation}
\[
\langle Z,a \rangle_{X/S} = \psi_{X/S}(Z)^*(a)\quad\text{ for } Z\in c(X/S),\; a\in F(X).
\]
It satisfies the following conditions:
\begin{thlist}
\item 
If $i:S\to X$ is a section of $X\to S$, then $\langle i(S),a\rangle_{X/S} = i^*a$ for $a\in F(X)$.
\item 
If $f:S'\to S$ is a morphism, then $f^*\langle Z,a\rangle_{X/S}=\langle f^*Z,g^*a\rangle_{X'/S'}$ for $Z\in c(X/S)$ and any $a\in F(X)$, where $X'=S'\times_S X$ and $g:X'\to X$ is the second projection (see \ref{c(X/S)} for $f^*Z$).
\item
For $(p_i:X_i\to S)\in \rC S$ with $i=1,2$ and for an $S$-morphism $f:X_1\to X_2$, we have
$\langle f_*(Z),\alpha \rangle_{X_2/S} = \langle Z, f^*(a) \rangle_{X_1/S}$
for $Z\in c(X/S)$ and $a \in F(X_2)$.
\end{thlist}
\end{para}

\begin{para}\label{def.pretheory} 
Following Voevodsky \cite[Def.3.1]{voepre}, a \emph{pretheory over $k$} is defined as a presheaf
$F: \Sm^\op \to \Ab$, commuting with coproducts and provided with bilinear pairings \eqref{eq.prethpairing} 
for all $S\in \Sm$ and $X/S\in \rC S$ subject to the conditions (i) and (ii). 
If it satisfies additionally (iii), $F$ is called of homological type.
Pretheories form an abelian category containing $\PST$ as a (non full) subcategory via \eqref{eq.cCor}.
\end{para}

\begin{defn}\label{def.goodcompactwithmodulus} 
Let $S\in \Sm$ and $X/S\in \rC S$. A \emph{good compactification} of $X/S$ is a dense open immersion
$j:X\hookrightarrow \Xb$ of $S$-schemes such that 
\begin{enumerate}
\item $\Xb$ is normal and $\pb:\Xb\to S$ is proper of equidimensional of dimension one.
\item $\Xb-X$ has an affine neighbourhood in $\Xb$.
\end{enumerate}
A \emph{good compactification} of $X/S$ with modulus is a pair $(j:X\hookrightarrow \Xb,Y)$ of 
a good compactification of $X/S$ and a closed subscheme $Y\subset \Xb$ with $X=\Xb-Y$.
We sometime write $(\Xb/S,Y)$ for $(j:X\hookrightarrow \Xb,Y)$ for simplicity.
Note that this is not a modulus pair in the sense of \ref{def.moduluspair} {unless $S$ is  proper  \cite[5.4.3 (ii)]{ega2}}. 
\end{defn}

\begin{para}
Let $(\Xb/S,Y)$ be as in \ref{def.goodcompactwithmodulus} with $X=\Xb-Y$.
Let $\Pic(\Xb,Y)$ be the relative Picard group (see \ref{ssect:rel-pic}).
As in the proof of \cite[Th. 3.1]{sus-voe2}, it follows from \cite[Lemma 2.3]{sus-voe2} that 
we have a short exact sequence
\begin{equation}\label{eq2.2}
0\to G(\Xb,Y)\by{\div_{\Xb}} c(X/S)\rmapo{\tau} \Pic(\Xb,Y)\to 0
\end{equation}
(compare \eqref{eq:til-c-div}). The map $\tau$ can be described as follows.
The components of $Z\in c(X/S)$ are closed and of codimension one in $\Xb$ so that 
$Z$ is a Cartier divisor on $\Xb$ whose support is disjoint from $Y$, and that  
there is a natural isomorphism $\sigma:\sO_{\Xb}(Z)_{|Y} \simeq \sO_Y$.
Then $\tau$ sends $Z$ to the class of the pair $(\sO_{\Xb}(Z),\sigma)$.
\end{para}


\begin{defn}\label{def.recpretheory} 
Let $F$ be a pretheory.
\begin{enumerate}
\item
Given $S\in \Sm$ and $X/S\in \rC S$ and a good compactification $(\Xb/S,Y)$ of $X/S$ with modulus and 
$a\in F(X)$, we say that $Y$ is a \emph{weak modulus} for $a\in F(X)$ (or $a$ has \emph{weak modulus} $Y$) if
\[
\langle\div_{\Xb}(g),a\rangle_{X/S}=0\quad \text{ for any } g\in G(\Xb,Y), 
\]
equivalently the morphism
\begin{equation*}
\langle -, a\rangle_{X/S} : c(X/S)\to F(S)
\end{equation*} 
factors through $\tau$ in \eqref{eq2.2}.
\item
We say that $F$ \emph{has weak reciprocity} if, for any affine $S\in \Sm$ and 
$X/S\in \rC S$ with a good compactification $X\hookrightarrow \Xb$, any section $a\in F(X)$ has a weak modulus.
\end{enumerate}
\end{defn}


\begin{lemma}\label{recimplyweakrec}
If $F\in \PST$ has reciprocity, it has weak reciprocity.
\end{lemma}
\begin{proof}
Let $S\in \Sm$ be affine and 
take $C/S\in \rC S$ with a good compactification $C\hookrightarrow \Cb$.
(Note that $C$ is quasi-affine 
because $S$ is affine and $X/S \in \rC S$.)
Take $a \in F(C)$. We must show that $a$ has a weak modulus.
Take an integral proper $k$-scheme $\Xb$ which contains $\Cb$ as a dense open subset.
By the assumption on $F$,  $a$ has a modulus $Y \subset \Xb$ in the sense of \ref{def.modulus} such that $|Y| = \Xb - C$.
Let $\phi$ be the composite map
\[
\Cb \by{\gamma_{\bar p}}  \Cb \times S \hookrightarrow \Xb \times S\]
with $\bar p:\Cb\to S$ the projection and $\gamma_{\bar p}$ its graph. Then $(\phi:\Cb\to \Xb\times S)\in\sC_{(\Xb, Y)}(S)$
and $\Cb \times_{\Xb \times S} (Y \times S)$ is a weak modulus for $a$ (cf. \ref{def:relative-g-sheaf}).
\end{proof}


\begin{remark}\label{rem:adv-of-rec-over-wrec}
The analogue of Theorem \ref{thm:main-rep} (1) still holds for weak reciprocity, but the analogue of (2) does not hold.
More precisely, for $(\Xb/S,Y)$ as in \ref{def.goodcompactwithmodulus} with $X=\Xb-Y$, the functor
\[ \PST \to \Ab, \quad
F \mapsto \{ a \in F(X) ~|~ a ~
\text{has weak modulus $Y$} \}
\]
is representable. But the representing object does not have weak reciprocity,
even if $S=\Spec k$.
\end{remark}

\subsection{Local symbols}\

In this subsection, we show the equivalence between weak reciprocity for $F\in \PST$ in the sense of Definition \ref{def.recpretheory} 
for curves over a function field and the existence of local symbols satisfying ``Weil reciprocity'', 
in analogy with \cite[Ch. III]{gacl}.

Let $E=k(S)$ be the function field of a connected $S\in \Sm$.
Let $X$ be a normal projective irreducible curve over $E$ with function field $K=E(X)$.
Let $\Xd 0$ be the set of closed points of $X$.
For each $x \in \Xd 0$, 
let $v_x : K^\times \to \Z$ be the normalized valuation at $x$.

We fix $F\in \PST$. Set for $a\in F(\sO_{X,x})$ : 
\[a(x)=\langle x,\tilde{a}\rangle_{U/E},\]
where $U$ is an open neighbourhood of $x$ such that $a$ is the image of some $\tilde{a}\in F(U)$.
Note that $a(x)$ is independent of choices of $U$ and $\tilde{a}$. In fact we have $a(x)={f_x}_*i_x^* a$,
where $i_x:x\to \Spec\sO_{X,x}$ is the closed immersion and $f_x$ is the finite morphism
$x\to \Spec E$. (By Proposition \ref{p8.1} and Lemma \ref{l6.7} below, this formula holds more generally for a pretheory which is $\P^1$-rigid in the sense of Definition \ref{d6.2}.)

\begin{proposition}\label{prop:loc-symbol-over-fields}
The following conditions are equivalent:
\begin{enumerate}
\item
For any dense open affine $U \subset X$ and $a \in F(U)$, there exists an effective divisor $Z\subset X$ such that
$U=X\setminus |Z|$ and that
\begin{equation}\label{clubsuit}
\langle \div_{X}(f), a \rangle_{U/E}=0\in F(E)\quad\text{for all } f\in G(X,Z).\tag{$\clubsuit$}
\end{equation}
\item
There exists a family
of bilinear pairings 
\[ \Big\{ (-, -)_x : K^\times \times F(K) \to F(E) 
\Big\}_{x \in \Xd 0} 
\]
which satisfies the following conditions:
\begin{enumerate}
\item
For any $x \in \Xd 0, ~a \in F(\sO_{X, x})$ and $g \in K^\times$,
we have $(g,a|_K)_x = v_x(g) a(x)$.
\item 
For any $x \in \Xd 0$ and $a \in F(K)$,
we have $(U_{x}^{(m)},a)_x=0$ for sufficiently large $m>0$, where 
\[U_{x}^{(m)} = \{ u \in K^\times \mid v_x(u-1) \geq m \}.\]
\item
$\sum_{x \in \Xd 0} (g,a)_x = 0$ 
for any $a \in F(K)$ and $g \in K^\times$.
\end{enumerate}
\end{enumerate}
Moreover, if (2) holds, \eqref{clubsuit} for $a\in F(U)$ is equivalent to the condition:
\begin{equation}\label{heartsuit}
(U_x^{(m_x)},a)_x=0\quad\text{ for all } x \in X - U,\tag{$\heartsuit$}
\end{equation}
where $m_x$ is the multiplicity of $x$ in $Z$.
\end{proposition}

\begin{proof}  Assume (1). By \eqref{eq2.2}, we have an isomorphism
\[\Coker\big(G(X,Z)\to c(U/E)\big)\simeq \Pic(X,Z).\]
Hence, passing to the limit over all $U\subset X$ and $Z$, 
the pairings $\langle-, -\rangle_{U/E}$ induces 
\[\langle-, -\rangle_{K/E} :  \underset{Z}{\lim} \Pic(X,Z) \times F(K) \to F(E).\]
By weak approximation, we have an isomorphism
\[
\Pic(X,Z)  \simeq 
\Coker \big(K^\times\to  \underset{x\in \Xd 0-|Z|}{\bigoplus} \Z \oplus
\underset{x\in |Z|}{\bigoplus} K^\times/U_{x}^{(m_x)}\big),
\]
where $m_x$ is the multiplicity of $x$ in $Z$.
Whence a natural map
\[
\pi_x: K^\times \to \underset{Z}{\lim} \Pic(X,Z)\quad\text{  for } x\in \Xd 0
\]
induced by the projections $K^\times \to K^\times/U_{x}^{(m_x)}$. Define 
\[
(-, -)_x :K^\times \times F(K) \to F(E)
\]
by $(f,a)_x=\langle \pi_x(f), a\rangle_{K/E}$ for $f\in K^\times$ and $a\in F(K)$.
It is easy to verify (a)-(c). Hence (1) $\Rightarrow$ (2).

Now suppose (2) and take $a \in F(U), ~U \subset X$
as in (1).
For each $x \in X \setminus U$, there exists an integer $m_x\geq 1$
such that $( U_{x}^{(m_x)},a)_x=0$ by (b).
Define $Z = \sum_{x \in X \setminus U} m_x x$.
For $g \in G(\bar{X}, Z)$, we get
\[
\langle \div_{X}(g),a\rangle_{U/E} 
\overset{(a)}{=} \sum_{u \in \Ud 0} (g,a)_u
\overset{(!)}{=}\sum_{x \in \Xd 0} (g,a)_x
\overset{(c)}{=} 0. \]
Here $(!)$ holds since $g \in G(\bar{X}, Z)$ implies $g\in U_{x}^{(m_x)}$ for $x\in |Z|$.
This proves (2) $\Rightarrow$ (1) and the implication \eqref{heartsuit} $\Rightarrow$ \eqref{clubsuit}.

It remains to show the implication \eqref{clubsuit} $\Rightarrow$ \eqref{heartsuit} .
Fix $x \in X - U$ and let $m_x$ be the multiplicity of $x$ in $Z$.
Take $f \in U_x^{(m_x)}$.
By (2)(b), for each $y \in X - U$
there exists $n_y\geq m_y$ such that $( U_y^{(n_y)},a)_y=0$.
By the approximation theorem, we find $g \in K^\times$ such that
$g/f \in U_x^{(n_x)}$ and $g\in U_y^{(n_y)}$ for all $y \in X - U$.
This implies $g\in G(X,Z)$, and hence $(\clubsuit)$ and (2)(a) imply
\[0=\langle\div_{X}(g),a\rangle_{U/E}
= \sum_{y \in  \Ud 0} v_y(g)a(y) .
\]
Then, using properties in (2), we compute
\begin{multline*}
 (f,a)_x = ( g,a)_x
\overset{(c)}{=} -\sum_{y \in \Xd 0 - \{ x \}} ( g,a)_y =\\
 -\sum_{y \in \Ud 0 } ( g,a)_y
\overset{(a)}{=} -\sum_{y \in \Ud 0} v_y(g)a(y)=0.
\end{multline*}
The proposition is proved.
\end{proof}


\section{$\P^1$-invariance}\label{P1inva}

In this section we discuss 
$\P^1$-invariance and $\P^1$-rigidity
for a pretheory with weak reciprocity, and draw a few consequences.

\subsection{$\P^1$-invariance and $\P^1$-rigidity}

\begin{thm}\label{thm:P1invariance-pretheory}
Let $F$ be a pretheory which is separated for the Zariski topology.
If $F$ has weak reciprocity, then it is $\P^1$-invariant, namely
$\pi^*:F(S)\iso F(\P^1_S)$  for any $S\in \Sm$ where $\pi:\P^1_S \to S$ is the projection.
\end{thm}

\begin{rk}\label{rem:rec-vs-p1inv}
$\P^1$-invariance is clearly stable under arbitrary products in $\PST$. On the other hand, one can easily see that the Zariski separated presheaf with transfers $\G_a^\N$ does not have (weak) reciprocity, by checking that the evaluation in the proof of Proposition \ref{prop:unip} is optimal.
\end{rk}

Theorem \ref{thm:P1invariance-pretheory} will be deduced from 
Proposition \ref{p8.1} below,
for which we need to introduce the notion of $\P^1$-rigidity.
For $t\in \P^1(k)$, let $i_t:S \to \P^1_S$ be the corresponding section of $\pi$ at $t$.

\begin{defn}\label{d6.2} A presheaf $F:\Sm^\op\to \Ab$ of abelian groups is \emph{$\P^1$-rigid} 
if $i_0^*=i_\infty^*:F(\P^1_S)\to F(S)$ for any $S\in \Sm$.
\end{defn}

\begin{prop}\label{p2.1.2} For a presheaf $F$ of abelian groups, $\P^1$-invariance implies $\P^1$-rigidity. The converse is true if $F$ is separated for the Zariski topology.
\end{prop}

\begin{proof} It is obvious that $\P^1$-invariance implies $\P^1$-rigidity. We prove the converse:
It suffices to show $\Ker(F(\P^1_S) \by{i_0^*}F(S)) = 0$ for any $S\in \Sm$. We apply $\P^1$-rigidity  to $X=\A^1_S$  to get
\begin{equation}\label{eq7.1} \tilde\imath_0^*=\tilde\imath_\infty^* : F(\P^1_X) \to F(X)
\end{equation}
where (for clarity) $\tilde\imath_t$ is relative to the base $X$. Consider the morphisms
\[\phi_1,\phi_2 : \P^1_X=\P^1\times \A^1_S \to \P^1_S\]
given by
\[\phi_1((u_0: u_1), t) )= (u_0 : u_1 + tu_0),\quad \phi_2((u_0: u_1), t)) = (u_1 + tu_0 : u_0).\]

Letting $p : \A^1_S \to S$ be the projection, we compute
\begin{alignat}{3}\label{eq7.2}
\phi_1\circ \tilde\imath_0 &= j_\infty : \A^1_S \to \P^1_S,&& \quad t \mapsto (1 : t)\notag\\
\phi_1 \circ \tilde\imath_1 &= i_1 \circ p : \A^1_S \to \P^1_S,&&\quad  t \mapsto 1\\
\phi_2\circ \tilde\imath_0 &= j_0 : \A^1_S \to \P^1_S,&&\quad t \mapsto (t : 1)\notag\\
\phi_2\circ \tilde\imath_1 &= i_0 \circ p : \A^1_S \to \P^1_S,&&\quad  t \mapsto 0\notag
\end{alignat}
Take $a \in \Ker(F(\P^1_S) \by{i_0^*} F(S))$. Then $a \in \Ker(F(\P^1_S) \by{i_\infty^*}F(S))$ by rigidity. By \eqref{eq7.1}
\[\tilde\imath_0^*\phi_\nu^*(a) =\tilde\imath_\infty^*\phi_\nu^*(a)\quad (\nu = 1; 2).\]

By \eqref{eq7.2} this implies
\[j_\infty^*  a = p^*i_\infty^*a = 0,\quad j_0^*a = p^*i_0^*a = 0.\]
Noting $F$ is separated for the Zariski topology, this implies $a = 0$ 
since $\P^1_S = j_0(\A^1_S) \cup j_1(\A^1_S)$ is an open covering of $\P^1_S$.
\end{proof}

Theorem \ref{thm:P1invariance-pretheory} now follows from

\begin{prop}\label{p8.1} A pretheory $F$ having weak reciprocity is $\P^1$-rigid.
\end{prop}

\begin{proof} Let $S\in \Sm$ and $a\in F(\P^1_S)$. Let $t$ be the standard coordinate on $\A^1_k=\P^1_k-\{\infty\}$.
Let $j_0: \A^1_S\hookrightarrow \P^1_S$ be the natural inclusion and $j_\infty$ be the composite of $j_0$ and 
the morphism $\P^1_S \to \P^1_S; t\to t^{-1}$.
These define two good compactifications of $\A^1_S$.
Since $F$ has weak reciprocity, there exists $n_0>0$ such that for all $ n\ge n_0$, we have
\[\langle \div_{\A^1_S}(1+t^{-n}),j_0^* a\rangle_{\A^1_S/S} = 
\langle \div_{\A^1_S}(1+t^{n}),j_\infty^*a\rangle_{\A^1_S/S}=0\in F(S)\]
Noting
\[\frac{1+t^n}{1+t^{-n}}=t^n,\]
we deduce
\[n(i_0^*a-i_\infty^*a)=0.\]
Replacing $n$ by $n+1$, we get the desired rigidity.
\end{proof}

\subsection{Application to $0$-cycles}

\begin{cor} Let $F\in \PST$ be such that the associated pretheory has weak reciprocity, and let $X$ be a smooth proper $k$-variety. Then the natural pairing $\langle ,\rangle_X:Z_0(X)\times F(X)\to F(k)$ factors through a pairing
\[CH_0(X)\otimes F(X)\to F(k).\]
\end{cor}

\begin{proof} Let $C\subset X$ be an irreducible curve and let $f\in k(C)^*$: we must show that the map $F(X)\to F(k)$ induced by pairing with $\div(f)$ is $0$. Let $\tilde C$ be the normalisation of $C$: we have a diagram
\[\begin{CD}
\tilde C@>i>> X\\
@V{f}VV\\
\P^1.
\end{CD}\]

Let $a\in F(X)$. We have
\[\langle \div(f),a\rangle_X = \langle \div(f),i^*a\rangle_{\tilde C} =\langle 0-\infty,f_*i^*a\rangle_{\P^1}=0,\]
where the last equality follows from Prop.~\ref{p8.1}.
\end{proof}

\subsection{Functoriality for pretheories with weak reciprocity}

We apply the above to show that weak reciprocity is sufficient to yield some important properties to a pretheory, which are automatic for presheaves with transfers. This generalises  \cite[Prop. 3.12, 3.14 and 3.15]{voepre} (the homotopy invariant case).

\begin{prop}\label{pr.pretheory} Let $F$ be a $\P^1$-rigid pretheory. 
Let $S\in \Sm$ and $X/S\in \rC S$ and $j:U\inj X$ be an open embedding. 
Then for any $(Z,a)\in c(U/S)\times F(X)$ one has
\[\langle j_*Z,a\rangle_{X/S} = \langle Z,j^*a\rangle_{U/S}.\]
\end{prop}


\begin{proof} 
This is proven
by adapting the proof of
\cite[Prop. 3.12]{voepre}.
Let $\Xi =X-U$ and $W = X\times \P^1\setminus \Xi\times \{0\}$, so that we have a commutative diagram
\[\begin{CD}
W@>q>> X\\
@VVV @VVV \\
S\times \P^1 @>\pi>> S
\end{CD}\]
where $q$ is induced by the projection $X\times \P^1\to X$. 
For the sections $i_0$ (at $0$) and $i_1$ (at $1$) of $\pi$, we have $i_0^*W=U$ and $i_1^*W=X$. 
Let $\tilde Z\in c(W/S\times \P^1)$ be the unique element whose image in $c(X\times \P^1/S\times \P^1)$ equals $\pi^*Z=Z\times \P^1$. Then we have $i_0^* \tilde Z = Z$ and $i_1^*\tilde Z = j_*Z$. 
Put $\phi= \langle \tilde Z,q^*a\rangle_{W/S\times \P^1}\in F(X\times \P^1)$.
By \ref{eq.prethpairing}(ii), we get
\[i_0^*\phi = \langle Z,j^*a\rangle_{U/S}, \quad i_1^*\phi =\langle j_*Z,a\rangle_{X/S}.\]
By $\P^1$-rigidity, we have $i_0^*\phi = i_1^*\phi$, hence the claim.
\end{proof}

Let $E$ be a finite separable extension of $k$ and $f:\Spec E \to \Spec k$ the projection.
Let $F$ be a pretheory. In \cite[p. 101]{voetri}, Voevodsky defines a trace map $f_*:F(E)\to F(k)$ by
\[f_* a = \langle \{0\}_{E},p^* a\rangle_{\A^1_{E}/k}\]
where $p:\A^1_{E}\to \Spec E$ is the structural map. 
Taking $T\in \Sm$, we apply the above construction to the pretheory $S \to F(T\times S)$ and get a map 
$f_*:F(T_E) \to F(T)$ with $T_E=T\times \Spec E$. By construction $f_*$ satisfies the obvious functoriality 
with respect to $T\in \Sm$. 

\begin{prop}\label{pr.pretheory2} 
If $F$ has weak reciprocity, one has the identity
\[ f_*f^*=\times [E_2:E_1]\;, \]
where $f^*:F(T) \to F(T_E)$ is the pullback by $f$.
%
\end{prop}

\

\begin{proof}
%
We adapt that of \cite[Cor. 3.15]{voetri}. We give details, especially as one line of that proof has been omitted from the published edition. We may assume $T=\Spec k$.
Let $x_0, ~x_\infty \in \P^1$ be the points at zero and infinity.
Since $E/k$ is separable, the choice of a primitive element $\alpha$ yields a point $x\in \P^1$ with residue field $E$. 
Let $P$ be the minimal polynomial of $\alpha$ 
and put $d:=[E:k]$.
For $n \in \Z_{>0}$
the function $t^{-dn}P^n$ belongs to $G(\P^1, nd x_{\infty})$
and has divisor $nx - nd x_0$.
Thus, by weak reciprocity, for any $a\in F(\P^1_k)$
there exists $n_0 \in \Z_{>0}$ such that 
$\langle nx - nd x_0, a\rangle_{\P^1/k}=0$
for all $n \geq n_0$.
Applying this to $n=n_0$ and $n_0+1$,
we get
\[\langle x - d x_0 ,a\rangle_{\P^1/k}=0. \]
Let $\pi:\P^1\to \Spec k$ be the structure map
and $i_y : y \to \P^1$ the natural embedding for $y =x, x_0$.
Taking $a=\pi^*a_0$ for $a_0\in F(k)$,
we get from the lemma below:
\[0=  f_*i_x^*\pi^* a_0-d i_{x_0}^* \pi^*a_0  = f_*f^* a_0-da_0
\]
as requested.
\end{proof}

\begin{lemma}\label{l6.7}
Let $C$ be a smooth curve over $k$ and let $x$ be a closed point of $C$ 
with separable residue field. Then for any $\P^1$-rigid pretheory $F$ and any $a\in F(C)$ one has
\[f_{x*} i_x^*a = \langle x, a\rangle_{C/k}  \]
where $f_x:x\to \Spec k$ is the structural map and $i_x:x\inj C$ is the closed immersion.
\end{lemma}
\begin{proof}

This is proven
by adapting the argument of \cite[Prop. 3.13]{voepre}
in the same manner as Proposition \ref{pr.pretheory}.
The details are left to the reader.
\end{proof}

\section{Injectivity}\label{inj}

\subsection{Statement of the results}
The purpose of this section is to prove Theorem \ref{thm:inj.intro} of the introduction. It  follows from the stronger

\begin{thm}\label{thm:inj} 
Theorem \ref{thm:inj.intro} is valid for any pretheory  $F$ which has weak reciprocity.
\end{thm}

Theorem \ref{thm:inj} is in turn a direct consequence of 
the following, whose proof will be given in \S \ref{sect:proof-of-inj}.
 
\begin{thm}\label{thm:key-factor}
Let $F$ be a pretheory with weak reciprocity. 
Let $X\in \Sm$, $V \subset X$ an open dense subset and
$x_1, \dots, x_n \in X$ a finite collection of points.
Then there exists an open neighbourhood $U$ of $\{x_1, \dots, x_n\}$ such that 
we have a injection
\[ \Ker(F(X) \to F(V)) \subset \Ker(F(X) \to F(U)). \]
\end{thm}

This strengthens \cite[Th. 11.3]{mvw} where the same assertion is proven assuming $F$ homotopy invariant. By the argument of \cite[Lemma 22.8]{mvw} and \cite[Cor. 11.2]{mvw},
the above theorem implies the following.

\begin{cor}\label{cor:inj} 
Let $F$ be as in \ref{thm:inj}.
Let $F_{\Zar}$ be the Zariski sheafification of $F$ as a presheaf.
\begin{enumerate}
\item For an open dense immersion $U \hookrightarrow X$ in $\Sm$, $F_{\Zar}(X) \to F_{\Zar}(U)$ is injective.
\item If $F(E)=0$ for any field $E$, then $F_{\Zar}=0$.\qed
\end{enumerate}
\end{cor}


\begin{remark}
We prove Theorem \ref{thm:key-factor}
following Voeveodsky's proof of \cite[Th. 11.3]{mvw}.
Under the additional assumption that $F$ is a Nisnevich sheaf,
one could alternatively deduce  Theorem \ref{thm:key-factor}
from $\P^1$-invariance (Theorem \ref{thm:P1invariance-pretheory})
by the method of \cite{cthk}.
\end{remark}

\subsection{Proof of Theorem \ref{thm:key-factor}}\label{sect:proof-of-inj}
We will use the notion of standard triples as in 
 \cite[Def. 4.1]{voepre}, \cite[Def. 11.5 and 11.11]{mvw}
and two results about them 
(Proposition \ref{prop:walker} and Lemma \ref{lem:exist-s-tri}).

\begin{definition}\label{def.standardtriple}
Let $S\in \Sm$ be connected.
\begin{enumerate}
\item
A triple $(\bar{p}: \bar{X} \to S, X_{\infty}, Z)$ with $X:=\bar{X} \setminus X_{\infty}$,
is called a {\it standard triple} if the following conditions are satisfied:
\begin{enumerate}
\item $Z,X_{\infty}\subset \Xb$ are closed reduced and $Z \cap X_{\infty} = \emptyset$,
\item $X/S\in \rC S$ with $X \to S$ smooth, 
and $\pb:\Xb\to S$ is its good compactification, 
\item $Z \cup X_{\infty}$ has an affine open neighborhood in $\Xb$.
\end{enumerate}
\item
A standard triple 
$(\bar{p}: \bar{X} \to S, X_{\infty}, Z)$ 
is called {\it split} over 
an open subset $U \subset \bar{X} \setminus X_{\infty}$
if $L|_{U \times_S Z}$ is trivial,
where $L$ is the line bundle on $U \times_S \bar{X}$
corresponding to the graph of the diagonal map.
\end{enumerate}
\end{definition}

\begin{remark}\label{rem:affine}
By \cite[Rk. 11.6]{mvw}, \ref{def.standardtriple} (1) implies the following :
\begin{enumerate}
\item
$S$ is affine, and both $Z$ and $X_{\infty}$ are finite over $S$,
\item
$\Xb$ is a good compactification of both $X$ and $X\setminus Z$.
\end{enumerate}
\end{remark}


\begin{proposition}[Mark Walker, see \protect{\cite[Thm. 11.17]{mvw}}]
\label{prop:walker}
Assume $k$ is infinite. 
Let $W\in \Sm$ be connected and quasi-projective over $k$ and 
$Y \subsetneq W$ a closed subset with points $y_1, \dots, y_n \in Y$.
Then there exist an affine open neighbourhood
$X \subset W$ of $y_1, \dots, y_n$ and
a standard triple $(\bar{X} \to S, X_{\infty}, Z)$
such that $(X, X \cap Y) \simeq (\bar{X} \setminus X_{\infty}, Z)$.
\end{proposition}


\begin{lemma}[\protect{\cite[Lemma 11.14]{mvw}}]
\label{lem:exist-s-tri}
Let $T=(\bar{X} \to S, X_{\infty}, Z)$ be a standard triple,
and $x_1, \dots, x_n \in X := \bar{X} \setminus X_{\infty}$.
Then there exists an open neighbourhood $U \subset X$
of $x_1, \dots, x_n$ such that $T$ is splits over $U$.
\end{lemma}

We need a lemma.
Its second assertion will be used in \S \ref{sec:pf-keytheorem}.

\begin{lemma}\label{lem:pic-push}
Let $(\Xb/S, X_{\infty}, Z)$ be a standard triple
and
$Y \subset \Xb$ a closed subscheme with $|Y|=X_{\infty}$.
We have a commutative diagram
\[
\xymatrix{
c(X-Z/S) \ar[r]^{\tau} \ar[d]_{j_*} 
& \Pic(\Xb, Y \sqcup Z) \ar[d]^{f}
\\
c(X /S) \ar[r]_{\tau}
& \Pic(\Xb, Y),
}
\]
where 
the left vertical map is the push-forward
along immersion $j : X-Z \inj X$
and the right vertical map
is defined by
$f(\sL, \sigma)=(\sL, \sigma|_{Y})$.
Moreover,
we have an exact sequence
\[
\sO^{\times}(Z) \to \Pic(\Xb, Y \sqcup Z) \overset{j_*}{\to} \Pic(\Xb). 
\]
\end{lemma}

\begin{proof}
The first part is proven in \cite[Lem. 3.1.5 (2)]{deglise}.
To show the second part,
we consider a commutative diagram
with exact rows
(see \S \ref{ssect:rel-pic})
\[ 
\begin{CD}
\sO^\times(\bar X) @>>> 
\sO^\times(Y \sqcup Z) @>>>
\Pic(\bar X, Y \sqcup Z) @>>>
\Pic(\bar X) 
\\
@|  @V{a}VV @VV{j_*}V @| 
\\
\sO^\times(\bar X) @>>> 
\sO^\times(Y) @>>>
\Pic(\bar X, Y) @>>>
\Pic(\bar X).
\end{CD}
\]
As the Chinese reminder theorem shows
$\ker(a) = \sO^{\times}(Z)$,
the lemma follows from the diagram.
\end{proof}



The following proposition is the key to the proof of Theorem \ref{thm:key-factor}. 

\begin{proposition}
\label{prop:factor}
Let $F$ be as in Theorem \ref{thm:key-factor}. 
Let $T=(\bar{X} \overset{p}{\to} S, X_{\infty}, Z)$ be a standard triple.
Let $U \subset X:=\bar{X} \setminus X_{\infty}$ be an affine open subset such that $T$ is split over $U$.
Then, for any $a \in F(X)$, there is a homomorphism $\phi_a : F(X \setminus Z) \to F(U)$
such that $\phi_a(a|_{X\setminus Z}) = a|_U$.
\end{proposition}

\begin{proof}
(Compare \cite[Prop. 11.15]{mvw}.)
For any $S$-scheme $W$, 
we write $W_U=W\times_S U$ and denote by  $pr_W$ the projection $W_U  \to W$.
We write $j : U \to X$ and $j':X \setminus Z \to X$
for the inclusion maps. Put $j'_U=j' \times_S 1_U:(X \setminus Z)_U\to X_U$.
We then obtain a commutative diagram
\[
\begin{CD}
F(X \setminus Z) @<{j'}^*<<  F(X)  @>j^*>> F(U) \\
@V{pr_{X-Z}^*}VV @V{pr_X^*}VV \Vert \\
F((X \setminus Z)_U) @<{j' _U}^*<< F(X_U) @>\gamma_j^*>> F(U). 
\end{CD}
\]
where $\gamma_j:U\to X_U$ is the graph of $j$.

Let $a \in F(X)$.
Since $X_U/U\in \rC U$ admits the good compactification $X_U\hookrightarrow \Xb_U$,
there exists a weak modulus $Y\subset \Xb_U$ for 
$pr_X^*(a) \in F(X_U)$ such that $|Y|=(X_{\infty})_U$.
We shall construct a homomorphism
\[
\tilde{\phi}_Y : F((X \setminus Z)_U) \to F(U)
\]
such that $\gamma_j^*(\tilde{a}) = \tilde{\phi}_Y{(j'_U}^* (\tilde{a}))$ for any $\tilde{a} \in F(X_U)$ having weak modulus $Y$.
The proposition will then follow by
setting $\phi_a := \tilde{\phi}_Y \circ pr_{X-Z}^*$.

Consider the standard triple 
$T_U=(\bar{X}_U \to U, (X_{\infty})_U, Z_U)$.
Let $\gamma_j(U) \in c(X_U/U)$ be the image of $U\in c(U/U)$ under ${\gamma_j}_*: c(U/U)\to c(X_U/U)$ and
$(L, \sigma) \in \Pic(\bar{X}_U, Y)$ be the image of $\gamma_j(U)$ under $\tau$ in \eqref{eq2.2}, 
where $L \in \Pic(\bar{X}_U)$ and $\sigma : L|_Y \simeq \sO_Y$ is a trivialisation. 
To say that $T$ is split over $U$ means that
there exists a  trivialization
$\tau : L|_{Z_U} \simeq \sO_{Z_U}.$
Since $Z_U$ is disjoint from $Y$,
$\sigma$ and $\tau$ define
a trivialization 
$\sigma \oplus \tau : L|_{Y \sqcup Z_U} \simeq \sO_{Y \sqcup Z_U}$.

By construction, the image of $(L, \sigma \oplus \tau) $ 
via the canonical map
\[
\Pic(\bar{X}_U, Y \sqcup Z_U) \to \Pic(\bar{X}_U, Y)\]
(see Lemma \ref{lem:pic-push}) agrees with $(L, \sigma)$.
Choose a lift  
$\tilde{\delta} $ of $(L, \sigma \oplus \tau)$ 
in $c((X\setminus Z)_U/U)$ via $\tau$ in \eqref{eq2.2}. Now we define
\[ \tilde{\phi}_Y := \langle \tilde{\delta},-\rangle_{(X \setminus Z)_U/U}
  : F((X \setminus Z)_U) \to F(U).
\]
If $ \tilde a\in F(X_U)$, by Proposition \ref{pr.pretheory}, we have
\[\tilde \phi_Y({j'_U}^*\tilde a) = \langle \tilde{\delta},{j'_U}^*\tilde a\rangle_{(X \setminus Z)_U/U} =  \langle {j'_U}_*\tilde{\delta},\tilde a\rangle_{X_U/U}
.\]
If, moreover, $\tilde a$ has weak modulus $Y$, then $\langle {j'_U}_*\tilde{\delta},\tilde a\rangle_{X_U/U}$ only depends on the image of ${j'_U}_*\tilde{\delta}$ in $\Pic(\bar{X}_U, Y)$, hence
\[\langle {j'_U}_*\tilde{\delta},\tilde a\rangle_{X_U/U}=\langle \gamma_j(U),\tilde a\rangle_{X_U/U}=\gamma_j^* \tilde a\]
by \ref{pretheory-pairing}(i). This completes the proof of the proposition.
\end{proof}

\begin{remark}[Comparison with \protect{\cite[Thm. 11.3]{mvw}}]
If $F$ is homotopy invariant,
then $Y_\red$ will be a weak modulus for all $\tilde{a} \in F(X_U)$.
Consequently, one can choose $\phi_a$ independently of $a$.
Hence there is a homomorphism
$\phi : F(X - Z) \to F(U)$ such that $\phi(a|_{X-Z})=a|_U$
for any $a \in F(X)$.
\end{remark}

We are now ready to prove Theorem \ref{thm:key-factor}.

\begin{proof}[Proof of Theorem \ref{thm:key-factor}]
First we assume $k$ infinite. 
In this case Theorem \ref{thm:key-factor} follows from

\begin{claim}
Assume $k$ is infinite, and keep the notation and assumptions of Theorem \ref{thm:key-factor}. 
Then there exists an open neighbourhood $U$ of $x_1, \dots, x_n$ satisfying the following property:
for any $a \in F(X)$, there is a homomorphism $\phi_a : F(V) \to F(U)$
such that $\phi_a(a|_V)=a|_U$.
\end{claim}

\begin{proof}
By replacing $V$ by $V \setminus \{ x_1, \dots, x_n \}$,
we may assume $x_1, \dots, x_n \in X \setminus V$.
By Proposition \ref{prop:walker},
we can replace $X$ by an open neighbourhood of 
$x_1, \dots, x_n$ in such a way that
there exists a standard triple
$T=(\bar{X} \to S, X_{\infty}, Z)$ with
$(X, X \setminus V) \simeq (\bar{X} \setminus X_{\infty}, Z)$.
By Lemma \ref{lem:exist-s-tri}, 
$T$ splits over some
open neighbourhood $U \subset X$ of $x_1, \dots, x_n$.
Now apply Proposition \ref{prop:factor}.
\end{proof}

Next we treat the case where $k$ is finite. 
Let $E$ be a finite separable extension of $k$ and $f:\Spec E \to \Spec k$ the projection.
For $T\in \Sm$ the transfer structure on $F$ induces a map $f_*:F(T_E) \to F(T)$ with $T_E=T\times \Spec E$.
It satisfies the standard functoriality with respect to $T\in \Sm$ together with the identity
$f_*f^*=\times [E:k]$, where $f^*:F(T) \to F(T_E)$ is the pullback by $f$.
Now a standard norm argument using
Proposition \ref{pr.pretheory2} reduces us to the case where $k$ is infinite.
\end{proof}


\section{MV-effaceablity}\label{sec:pf-keytheorem}
In this section, 
we define (Definition \ref{def:mv-eff})
a condition for a pretheory to be \emph{MV-effaceable}.
We then prove that
a pretheory having weak reciprocity is MV-effaceable
(Theorem \ref{thm:key-tech}).
This is a key technical result
for the proof (to be given in the next section)
of Theorem \ref{thm:ZarNis.intro} in the introduction.

\subsection{MV-effaceable pretheories}\label{mveff}

\begin{definition}
An {\it upper distinguished square} 
is a Cartesian diagram 
\begin{equation}\label{eq:upper-dist}
\begin{CD}
B @>i>> Y \\
@VfVV @VVfV \\
A @>i>> X
\end{CD}
\end{equation}
of objects in $\Sm$ such that
(i) $i$ is an open immersion, 
(ii) $f$ is etale,
and (iii) $f$ induces an isomorphism $Y \setminus B \simeq X \setminus A$.

We denote the square \eqref{eq:upper-dist}
by $Q(X, Y, A)$.
Note that this induces for a pretheory $F$ a complex
\[0 \to  
F(X) \overset{\left(\begin{smallmatrix}i^*\\ f^*\end{smallmatrix}\right)}{\longrightarrow}
F(A) \oplus F(Y) \overset{(-f^*, i^*)}{\longrightarrow}
F(B) \to 0. 
\]
\end{definition}

\begin{definition}
Let 
$T_X = (\bar{X} \to S, X_{\infty}, Z_X)$
and 
$T_Y = (\bar{Y} \to S, Y_{\infty}, Z_Y)$
be standard triples, with $X=\bar X-X_\infty$ and $Y=\bar Y-Y_\infty$.
A {\it covering morphism} $f : T_Y \to T_X$
is 
a finite morphism $f : \bar{Y} \to \bar{X}$ such that
(i) $f^{-1}(X_{\infty}) \subset Y_{\infty}$,
(ii) $f|_Y : Y \to X$ is etale, 
(iii) $Z_Y = f^{-1}(Z_X) \cap Y$, and
(iv) $f$ induces an isomorphism $Z_Y \simeq Z_X$.
\end{definition}

\begin{remark}
If $f: T_Y \to T_X$ is a covering morphism,
then the square $Q=Q(X, Y, X \setminus Z_X)$ is 
upper distinguished with
$Y-Z_Y = (X-Z_X) \times_X Y$.
\end{remark}

The following lemma is proved in \cite[Lemma 21.3]{mvw}.

\begin{lemma}
\label{lem:split-cov}
Let $f: T_Y \to T_X$ be a covering morphism.
If $T_X$ is split over $U \subset X := \bar{X} \setminus X_{\infty}$,
then $T_Y$ is split over $f^{-1}(U) \cap Y$.
\end{lemma}

\begin{definition}\label{def:mv-eff}
A pretheory $F$ 
is said to be {\it MV-effaceable} if
the following condition is satisfied:
let 
$T_X = (\bar{X} \to S, X_{\infty}, Z_X)$
and 
$T_Y = (\bar{Y} \to S, Y_{\infty}, Z_Y)$
be standard triples,
and $f: T_Y \to T_X$ a covering morphism
so that 
$Q(X, Y, X \setminus Z_X)$ is upper distinguished.
Put $A=X \setminus Z_X, ~B=Y \setminus Z_Y$.
Let $Q'=Q'(X', Y', A')$ be another upper distinguished square
such that $X'$ and $Y'$ are affine.
Put $B' = A' \times_{X'} Y'$.
Let 
\[
j=\begin{pmatrix}
j_B & j_Y \\
j_A & j_X
\end{pmatrix}
 : 
Q' =
\begin{pmatrix}
B' & \overset{i'}{\to} & Y' \\
{}_{f'}\downarrow &  & \downarrow_{f'} \\
A' & \overset{i'}{\to} & X'
\end{pmatrix}
\to
Q =
\begin{pmatrix}
B & \overset{i}{\to} & Y \\
{}_{f}\downarrow &  & \downarrow_{f} \\
A & \overset{i}{\to} & X
\end{pmatrix}
\]
be a morphism of squares.
We then get a morphism of complexes
\begin{equation}\label{eq:key-lem}
\begin{CD}
0 &\to&  
F(X) @>\text{$\left(\begin{smallmatrix}i^*\\ f^*\end{smallmatrix}\right)$}>>
F(A) \oplus F(Y) @>(-f^*, i^*)>>
F(B) &\to& 0 
\\
@. @Vj_X^*VV  @V\text{$\begin{pmatrix}j_A^*&0\\0&  j_Y^*\end{pmatrix}$}VV @Vj_B^*VV
\\
0 &\to& 
F(X') @>\text{$\left(\begin{smallmatrix}{i'}^*\\ {f'}^*\end{smallmatrix}\right)$}>>
F(A') \oplus F(Y') @>(-{f'}^*, {i'}^*)>>
F(B') &\to& 0.
\end{CD}
\end{equation}
The condition for $F$ to be MV-effaceable
is that \eqref{eq:key-lem} 
induces the zero-map on all cohomology groups
if $j_X : X' \to X$ is an open immersion,
and the triple $T_X$ is split over $X'$.
\end{definition}

In \cite[Th. 21.6]{mvw}, 
a homotopy invariant presheaf with transfers
is shown to be MV-effaceable.
The following generalizes this result.

\begin{thm}\label{thm:key-tech}
A pretheory having weak reciprocity is MV-effaceable.
\end{thm}

The proof of this theorem will be completed in 
\S \ref{ssect:pf-of-keytech-th}.
Before that,
we need to prepare a few lemmas in
\S \ref{ssect:funct-rel-pic-gp}.

\subsection{Functoriality of the relative Picard group}\label{ssect:funct-rel-pic-gp}

\begin{lemma}
\label{lem:pic-pull}
Let $S \in \Sm$ and $X/S \in \rC S$.
Let $(\Xb/S, Y)$ be its good compactification with modulus.
For any morphism $f: S' \to S$ in $\Sm$,
we have a commutative diagram
\[
\xymatrix{
c(X/S) \ar[r]^{\tau} \ar[d]_{f^*} 
& \Pic(\Xb, Y) \ar[d]^{f^*}
\\
c(X \times_S S'/S') \ar[r]^(0.4){\tau}
& \Pic(\Xb \times_S S', Y\times_S S'),
}
\]
where the right vertical map
is defined by
$f^*(\sL, \sigma)=(f^* \sL, f^* \sigma)$.
\end{lemma}
\begin{proof}
See \cite[Lem. 3.1.5 (1)]{deglise}.
\end{proof}

\begin{lemma}\label{lem:pic-push2}
Let $S \in \Sm$.
We take $X/S,~ Y/S \in \rC S$
and let $(\Xb/S, V), (\overline{Y}/S, W)$ 
be their good compactifications with moduli.
Let $f : \Xb \to \overline{Y}$ be a finite {surjective} $S$-morphism.
Suppose that $f^*W \leq V$.
Then there is a unique homomorphism
\[ f_* : \Pic(\Xb, V) \to \Pic(\overline{Y}, W) \]
which fits into a commutative diagram
\[
\xymatrix{
c(X/S) \ar[r]^{\tau} \ar[d]_{f_*} 
& \Pic(\Xb, V) \ar[d]^{f_*}
\\
c(Y/S) \ar[r]^{\tau}
& \Pic(\overline{Y}, W).
}
\]
\end{lemma}
\begin{proof}
We may suppose $X$ and $Y$ are irreducible.
Uniqueness is obvious from the surjectivity of $\tau$.
Thus it suffices to show 
$N_{k(X)/k(Y)}(G(\Xb, V))\allowbreak \subset G(\overline{Y}, W)$,
which follows from Lemma \ref{l3.2.2} below.
\end{proof}
\begin{lemma}\label{l3.2.2} 
Let $R \subset R'$ be an extension of domains, where $R$ is normal and $R'$ is finite over $R$,
and let $I \subset R$ be an ideal. 
Let $K$ and $K'$ be the fraction fields of $R$ and $R'$ respectively.
Write $(1+I)^*=(1+I)\cap R^*$ and $(1+IR')^*=(1+IR)\cap {R'}^*$. Then $N_{K'/K}(1+IR')^* \subset (1+I)^*$.
\end{lemma}

\begin{proof} The hypotheses imply that $N_{K'/K}({R'}^*)\subseteq R^*$ and  $\Tr_{K'/K}(R')\subseteq R$, hence clearly  $\Tr_{K'/K}(IR')\subseteq I$. So it suffices to show that $N_{K'/K}(1+a) \equiv 1 + \Tr_{K'/K}(a) \mod I^2$
for any $a \in IR'$. We have
\[N_{K'/K}(1+a) = P_a(1)\]
where $P_a(T)= T^n-\sigma_1(a)T^{n-1}+\dots \pm \sigma_n(a)$
is the characteristic polynomial of $a$, with $n=[K':K]$. But $\sigma_i:K'\to K$ is given by a homogeneous polynomial of degree $i$ with coefficients in $K$, and $\sigma_i(R')\subseteq R$ by integrality. In particular, if $a =\sum \lambda_\alpha\mu_\alpha$ with $\lambda_\alpha\in I$, $\mu_\alpha\in R'$, then $\sigma_i(a)$ is a homogeneous polynomial of degree $i$ in the $\lambda_\alpha$'s, with coefficients in $R$.
\end{proof}

\subsection{Proof of  Theorem \ref{thm:key-tech}}
\label{ssect:pf-of-keytech-th}
We follow the method of \cite[Th. 21.6]{mvw}.
We give ourselves all the data in Definition \ref{def:mv-eff},
and prove \eqref{eq:key-lem} induces the zero-map
on all cohomology groups.
We take a closed subscheme 
$V=V_{\infty} \sqcup V_0$ on $\bar{X}$
(resp. 
$W=W_{\infty} \sqcup W_0$ on $\bar{Y}$)
such that $|V_{\infty}|=X_{\infty}, ~|V_0|=Z_X$
(resp. $|W_{\infty}|=Y_{\infty}, ~|W_0|=Z_Y$).
We suppose $f^* V \leq W$
and $f$ induces $W_0 \simeq V_0$.

In the sequel, we use the notation
\[(-)_{X'} = -\times_S X', \text{ etc.}\]
throughout. The main part of the proof is in 
the following proposition:


\begin{proposition}\label{prop:key-tech}
There exist $\lambda_A \in c(A_{X'}/X'),
\lambda_B \in c(B_{Y'}/Y')$
and $\psi \in c(B_{A'}/A')$
satisfying the following conditions:
\begin{align*}
 i_*\lambda_A  - \tilde \jmath_X&\in\Ker(c(X_{X'}/X') \Surj 
\Pic(\bar{X}_{X'}, V_{\infty X'})). \tag{1}\\
i_*\lambda_B  - \tilde\jmath_Y&\in\Ker c(Y_{Y'}/Y') \Surj 
\Pic(\bar{Y}_{Y'}, W_{\infty Y'})).\tag{2}\\
{f'}^*\lambda_A-f_*\lambda_B&\in\Ker(c(A_{Y'}/Y') \Surj 
\Pic(\bar{X}_{Y'}, V_{Y'})).\tag{3}\\
f_*\psi  - {i'}^*\lambda_A + \tilde\jmath_A&\in\Ker( c(A_{A'}/A') \Surj
\Pic(\bar{X}_{A'}, V_{A'})).\tag{4}\\
{f'}^*\psi - {i'}^*\lambda_B + \tilde\jmath_B&\in\Ker(c(B_{B'}/B') \Surj
\Pic(\bar{Y}_{B'}, W_{B'})).\tag{5}\\
i_*\psi&\in \Ker(c(Y_{A'}/A') \Surj
\Pic(\bar{Y}_{A'}, W_{\infty A'})).\tag{6} 
\end{align*}
Here $\tilde\jmath_X:X'\to X_{X'}$ is the graph of $j_X$, and similarly for $\tilde\jmath_Y,\tilde\jmath_A,\tilde\jmath_B$.
\end{proposition}

Given $\lambda_A \in c(A_{X'}/X'),
\lambda_B \in c(B_{Y'}/Y')$
and $\psi \in c(B_{A'}/A')$,
we define maps
\begin{equation}\label{eq:homotopy}
s_1 : F(A) \oplus F(Y) \to F(X'),
\quad
s_2: F(B) \to F(A') \oplus F(Y')
\end{equation}
by
$s_1(a, b)=\langle \lambda_A,a_{X'}\rangle_{A_{X'}/X'},
s_2(a)=(\langle\psi,a_{A'}\rangle_{B_{A'}/A'},\langle\lambda_B,a_{Y'}\rangle_{B_{Y'}/Y'})$.

\begin{remark}\label{r10.1}
When $F$ is homotopy invariant,
Theorem \ref{thm:key-tech} is proved
in \cite[Th. 21.6]{mvw} as follows:
Let us take $V$ and $W$ to be reduced.
By applying Proposition \ref{prop:key-tech},
one obtains $\lambda_A, \lambda_B$ and $\psi$.
Then the maps $s_1, s_2$ given by \eqref{eq:homotopy}
define a chain homotopy from 
\eqref{eq:key-lem} to zero.
(By homotopy invariance, 
the map $\lambda_A^* : F(A) \to F(X')$
depends only on the class of $\Pic(\bar{X}_{X'}, V_{X'})$,
for instance. Hence the conditions in Proposition \ref{prop:key-tech}
contains sufficient information to prove this assertion.)

When $F$ is not homotopy invariant (but has reciprocity),
we need to choose $V$ and $W$ depending on
elements of $F(X), F(A), F(Y)$ and $F(B)$.
Thus we cannot construct a globally defined chain homotopy.
However, the following corollary can be deduced:
\end{remark}

\begin{cor} 
Let $s_1, s_2$ be as in \eqref{eq:homotopy}.
\begin{enumerate}
\item[(a)]
Let $a \in F(X)$.
If the image of $a$ in $F(X_{X'})$
has modulus $V_{\infty X'}$,
then we have 
\[j_X^*a = s_1\left(\begin{smallmatrix}i^*\\ f^*\end{smallmatrix}\right) a.\]
\item[(b)]
Let $(a,b) \in F(A)\times F(Y)$.
Suppose the image of $a$ in $F(A_{X'})$ (resp.  in $F(A_{Y'})$) has modulus $V_{A'}$ 
(resp. $V_{Y'}$), and suppose the image of $b$ in $F(Y_{A'})$ (resp. in $F(Y_{Y'})$)  has modulus $W_{\infty A'}$ (resp.   $W_{\infty Y'}$). 
Then we have 
\[(j^*_Aa, j_Y^*b)=\left(\begin{smallmatrix}{i'}^*\\ {f'}^*\end{smallmatrix}\right)s_1(a,b) +  s_2(-f^*,i^*)(a,b) .\]
\item[(c)]
Let $a \in F(B)$.
If the image $a$ in $F(B_{B'})$
has modulus $W_{B'}$,
then we have 
\[j_B^*a = (-{f'}^*, {i'}^*) s_2 a.\]
\end{enumerate}
\end{cor}

\begin{proof}
Prop. \ref{prop:key-tech} (1)
implies (a).
(2), (3), (4) and (6) imply (b).
(5) implies (c). (We used all the axioms of a pretheory, plus Proposition \ref{pr.pretheory}.)
\end{proof}

Theorem \ref{thm:key-tech} follows from this corollary, 
because, for instance, 
for any $a \in F(X)$ by weak reciprocity
one can always find
a closed subscheme $V_\infty$ on $\Xb/S$ with $|V_\infty|=X_\infty$
such that $V_{\infty X'}$ is a modulus for $a$.

It remains to prove Proposition \ref{prop:key-tech}.
We divide it into three steps.

\subsubsection{Construction of $\lambda_A$ and $\lambda_B$}
We consider the commutative diagram
\[
\begin{CD}
c(B_{Y'}/Y') @>i_*>>c(Y_{Y'}/Y') \\
@Vf_*VV @Vf_*VV\\
c(A_{Y'}/Y') @>i_*>>c(X_{Y'}/Y') \\
@A{f'}^*AA @A{f'}^*AA\\
 c(A_{X'}/X')@>i_*>> c(X_{X'}/X').
\end{CD}
\]

Observe that  
$\tilde\jmath_X\in c(X_{X'}/X')$
and  $\tilde\jmath_Y\in c(Y_{Y'}/ Y')$
have the same image in $c(X_{Y'}, Y')$.
By passing to the quotient,
we obtain the right half of the following commutative diagram
(see Lemmas \ref{lem:pic-push}, \ref{lem:pic-pull} and \ref{lem:pic-push2}):
\[
\begin{CD}
\sO^*(W_{0 Y'}) @>>>
\Pic(\bar{Y}_{Y'}, W_{Y'}) @>>>
\Pic(\bar{Y}_{Y'}, W_{\infty Y'}) 
\\
@V(*)VV 
@VVV
@VVV
\\
\sO^*(V_{0 Y'}) @>>>
\Pic(\bar{X}_{Y'}, V_{Y'}) @>>>
\Pic(\bar{X}_{Y'}, V_{\infty Y'})
\\
@.
@AAA
@AAA
\\
@.
\Pic(\bar{X}_{X'}, V_{X'}) @>>>
\Pic(\bar{X}_{X'}, V_{\infty X'}).
\end{CD}
\]
The upper two rows are exact
by Lemma \ref{lem:pic-push}.
Note also that $(*)$ is surjective since 
$f|_{Z_Y} : Z_Y = |W_0| \iso Z_X= |V_0|$ 
and $f^* V_0 \simeq W_0$.

Since $T_X$ is split over $X'$, as in the proof of Proposition \ref{prop:factor} we get $\lambda_A \in c(A_{X'}/X')$ such that $i_*\lambda_A$ and $\tilde\jmath_X$ agree 
in $\Pic(\bar{X}_{X'}, V_{\infty X'})$.
Similarly, in view of Lemma \ref{lem:split-cov},
there exists $\lambda_B \in c(B_{Y'}/Y')$
such that  $i_*\lambda_B$ and $\tilde\jmath_Y$
agree in $\Pic(\bar{Y}_{Y'}, W_{\infty Y'})$.
 Moreover, using the surjectivity of $(*)$,
$\lambda_B$ can be chosen so that
its image in $\Pic(\bar{X}_{Y'}, V_{Y'})$
agrees with that of $\lambda_A$.
We have proven (1)-(3).

\subsubsection{Preliminary computation}
Before we construct $\psi$, 
we do some computations on $\lambda_A$ and $\lambda_B$.
Let $\sL$ be the invertible sheaf on $\bar{X}_{X'}$ 
corresponding to (the graph of) $j_X: X' \hookrightarrow X$.
The image of $\lambda_A$ in $\Pic(\bar{X}_{X'}, V_{X'})$
defines a trivialization
$\tau=\tau_{\infty} \sqcup \tau_0$,
where 
$\tau_{\infty} : \sL|_{V_{\infty X'}} \simeq \sO_{V_{\infty X'}}, ~
\tau_{0} : \sL|_{V_{0 X'}} \simeq \sO_{V_{0 X'}}$.
Since $\sL$ is defined by an effective divisor,
it has a canonical global section $\sigma \in \sL(\bar{X}_{X'})$
(given by the image of $1 \in \sO(\bar{X}_{X'})$).
The image of $\tilde\jmath_X $
in $\Pic(\bar{X}_{X'}, V_{\infty X'})$ is
given by the class of $(\sL, \sigma|_{V_{\infty X'}})$,
which equals the image of $i_*\lambda_A$.
This proves $\tau_{\infty} = \sigma|_{V_{\infty X'}}$.
As $\tau_0$ is a trivialization,
there exists
$r \in \sO(V_{0 X'})$
such that
$\sigma|_{V_{0 X'}} = r \tau_0$.
By pulling back along $i': A' \hookrightarrow X'$,
we find that 
\begin{equation}
\label{eq:lam-ja}
\text{
the class of $-\lambda_{A'}+j_A$
is represented by
$(\sO_{\bar{X}_{A'}}, 1 \sqcup (r|_{V_{0 A'}}))$}
\end{equation}
in $\Pic(\bar{X}_{A'}, V_{A'})$.
(Note that $\sigma|_{V_{0 A'}}$ defines a trivialization
$\sL|_{V_{0 A'}} \simeq \sO_{V_{0 A'}}$
because $A' \subset X' \setminus Z_X$.
Thus $r|_{V_{0 A'}} \in \sO(V_{0 A'})$ is invertible.)

Similarly, let
$\sL'$ be the invertible sheaf on $\bar{Y}_{Y'}$ 
corresponding to $j_Y : Y' \to Y$.
The image of $i_*\lambda_B$ in $\Pic(\bar{Y}_{Y'}, W_{Y'})$
defines a trivialization
$\tau'=\tau_{\infty}' \sqcup \tau_0'$,
where 
$\tau_{\infty}' : \sL'|_{W_{\infty Y'}} \simeq \sO_{W_{\infty Y'}}, ~
\tau_{0}' : \sL'|_{W_{0 Y'}} \simeq \sO_{W_{0 Y'}}$.
We also have a canonical global section
$\sigma' \in \sL'(\bar{Y}_{Y'})$.
We have $\tau_{\infty}' = \sigma'|_{W_{\infty Y'}}$.
As $\tau_0'$ is a trivialization,
there exists
$r' \in \sO(W_{0 Y'})$
such that
$\sigma'|_{W_{0 Y'}} = r' \tau_0'$.
We find that 
\begin{equation}
\label{eq:lam-jb}
\text{
the class of $-\lambda_{B'}+j_B$
is represented by
$(\sO_{\bar{Y}_{B'}}, 1 \sqcup (r'|_{W_{0 B'}}))$}
\end{equation}
in $\Pic(\bar{Y}_{B'}, W_{B'})$.
(Note that
$\sigma'|_{W_{0 B'}}$ defines a trivialization
$\sL'|_{W_{0 N'}} \simeq \sO_{W_{0 B'}}$
because $B' \subset Y' \setminus Z_Y$.
Thus we have $r'|_{W_{0 B'}} \in \sO^*(W_{0 B'})$.)

\subsubsection{Construction of $\psi$}
(Compare \cite[21.9]{mvw}.)
We consider the commutative diagram
\begin{equation}
\label{eq:cov-mor2}
\begin{CD}
W_{0 Y'}  @>\subset>> 
\bar{Y}_{Y'} @>f'>>
\bar{Y}_{X'} @<\supset<<
W_{0 X'}
\\
@V\wr VV 
@VfVV 
@VVfV 
@VV\wr V 
\\
V_{0 Y'}  @>\subset>> 
\bar{X}_{Y'} @>f'>>
\bar{X}_{X'} @<\supset<<
V_{0 X'}.
\end{CD}
\end{equation}

Let $\tilde{r} \in \sO(W_{0 X'})$
be the pull-back of $r\in \sO(V_{0 X'})$ along $f$.
By the definition of standard triple,
there is an affine open neighbourhood $U \subset \bar{Y}$
of $Y_{\infty} \sqcup Z_Y$.
Thus $U_{X'}$ is an affine open neighbourhood of $W_{X'}$.
By Chinese reminder theorem,
we find $h \in \sO(U_{X'})$ which is mapped to
$1$ in $\sO(W_{\infty X'})$ 
and to $\tilde{r}$ in $\sO(W_{0 X'})$.
(Thus $h$ is a rational function on $\bar{Y}_{X'}$.)
Now \emph{we define $\psi$ to be $-\div(h)$
considered as an element of $c(B_{A'}/A')$}.
Note that the support of $\psi$ is contained in
$(\bar{Y} \setminus U)_{A'} \subset B_{A'}$,
since $r|_{V_{0 A'}} \in \sO(V_{0 A'})^*$.

By definition,
the image of $\psi$ in 
$\Pic(\bar{Y}_{A'}, W_{A'})$ is represented by
$\alpha := (\sO_{\bar{Y}_{A'}}, 1 \sqcup \tilde{r}^{-1})$.
We consider the images of $\alpha$
by three maps.
First, 
\[
\Pic(\bar{Y}_{A'}, W_{A'}) \to 
\Pic(\bar{Y}_{A'}, W_{\infty A'}) 
\]
(see Lemma \ref{lem:pic-push})
sends $\alpha$ to the class of
$(\sO_{\bar{Y}_{A'}}, 1)$.
This proves (6).
Second, 
\[ {f}_* : \Pic(\bar{Y}_{A'}, W_{A'}) \to
\Pic(\bar{X}_{A'}, V_{A'})
\]
(see Lemma \ref{lem:pic-push2})
sends $\alpha$ to the class of 
$(\sO_{\bar{X}_{A'}}, 1 \sqcup N(h)|_{V_{0 A'}}^{-1})$.
On the other hand,
it is proved in [MVW, 21.10] that
$N(h)|_{V_{0 A'}} = r|_{V_{0 A'}} 
~\text{in} ~\sO(V_{0 A'}).$
(In loc. cit., $V$ is assumed to be reduced,
but this assumption is not used in the proof.)
(4) follows from this and \eqref{eq:lam-ja}.
Finally, 
\[{f'}^* : \Pic(\bar{Y}_{A'}, W_{A'}) \to 
\Pic(\bar{Y}_{B'}, W_{B'}) 
\]
(see Lemma \ref{lem:pic-pull})
sends $\alpha$ to the class of 
$(\sO_{\bar{Y}_{B'}}, 1 \sqcup ({f'}^* \tilde{r}^{-1}))$.
In view of \eqref{eq:lam-jb},
(5) is reduced to 
${f'}^* \tilde{r}|_{W_{0 B'}} = r'|_{W_{0 B'}}$ in 
$\sO(W_{0 B'})$.
By diagram \eqref{eq:cov-mor2} and the definition of $\tilde{r}$,
we have
${f'}^* \tilde{r}|_{W_{0 B'}}=f^* {f'}^* r|_{W_{0 B'}}$.
Since $f$ induces an isomorphism $W_0 \simeq V_0$,
it suffices to show
${f'}^*r|_{V_{0 B'}}=f_*r'|_{V_{0 B'}}$,
but this follows from (3).
This completes the proof of Proposition \ref{prop:key-tech}.

\section{Sheafification preserves reciprocity}\label{sheafification}

We prove
Theorem \ref{thm:ZarNis.intro} in the introduction.
(1), (2) and (3) are respectively shown
in \S \ref{sect:preserve-zar},
\ref{ssect:zar-sheafification-rec} and
\ref{sect:comparison-zar-nis}.
In the last subsection \S \ref{ssect:gen-jac},
we make a brief discussion on generalized Jacobian of a curve.

\subsection{Zariski sheafification preserves transfers}\
\label{sect:preserve-zar}

In \cite[Th. 22.15]{mvw}, Zariski sheafification of a homotopy invariant presheaf
with transfers is shown to have transfers.
The proof actually shows the following:

\begin{thm}[\protect{\cite[Th. 22.15]{mvw}}]
If $F \in \PST$ is MV-effaceable, then $F_{\Zar}$ has a unique structure of presheaf with transfers
such that $F \to F_{\Zar}$ is a morphism in $\PST$.
\end{thm}

Combined with Theorem \ref{thm:key-tech}, we obtain the following theorem:

\begin{thm}
\label{thm:zar-trans}
If $F \in \PST$ has weak reciprocity, $F_{\Zar}$ has a unique structure of presheaf with transfers
such that $F \to F_{\Zar}$ is a morphism in $\PST$.
\end{thm}

\begin{remark}
\'Etale and Nisnevich analogue of the above theorem
hold for any $F \in \PST$
(without assuming reciprocity).
See \cite[6.17, 14.1]{mvw}.
\end{remark}

\subsection{Zariski sheafification preserves reciprocity.}
\label{ssect:zar-sheafification-rec}

In \cite[Th. 22.2]{mvw}, it is proved that Zariski sheafification of a
homotopy invariant presheaf with transfers is homotopy invariant.
The same argument does not work for reciprocity sheaves.
The proofs of following results are based on a different idea.

\begin{lemma}\label{lem:zar-rec}
If $F \in \PST$ has weak reciprocity, so does $F_{\Zar}$ 
(Note that the statement makes sense by Theorem \ref{thm:zar-trans}).
\end{lemma}
\begin{proof}
Let $S\in \Sm$ and $X/S\in \rC S$ with a good compactification $X\hookrightarrow \Xb$,
and take $a \in F_{\Zar}(X)$.
We need to prove that $a$ has a weak modulus $Y\subset \Xb$ such that $X=\Xb \setminus |Y|$.
We may assume $X$ and $S$ connected.
By Corollary \ref{cor:inj},  $F_{\Zar}(S) \to F_{\Zar}(V)$ is injective for any dense open subset $V\subset S$.
Hence we may prove the assertion after the base change to the generic point $\eta$ of $S$. 
Now the theorem follows from Proposition \ref{prop:loc-symbol-over-fields}
since the condition (2) of the proposition is insensible to Zariski sheafification. 
\end{proof}

\begin{thm}\label{thm:zar-rec}
Assume $k$ is perfect.
If $F \in \PST$ has reciprocity, then $F_{\Zar}$ has reciprocity.
\end{thm}
\begin{proof}
Let $\rho : F \to F_\Zar$ be the canonical morphism.
Let $X \subset \Xb$ be an open immersion such that
$\Xb$ is an integral proper variety over $k$
and $X \in \Sm$ is quasi-affine.
Let $a \in F_\Zar(X)$.
We need to show $a$ has a modulus $Y \subset \Xb$
such that $|Y|=\Xb-X$.

There exists a Zariski open covering
$X=\cup_{i=1}^r U_i$ and $a_i \in F(U_i)$
such that $\rho(a_i)=a|_{U_i}$ for each $i$.
By assumption, $a_i$ has a modulus $Y_i \subset \Xb$ such that $|Y_i|=\Xb-U_i$.
Then $Y_i$ is a modulus for $\rho(a_i) \in F_\Zar(U_i)$.
Put $Y := Y_1 \times_{\Xb} \dots \times_{\Xb} Y_r$.
Since $X=\cup_{i=1}^r U_i$, we have $|Y|=\Xb-X$.
We shall prove $Y$ is a modulus for $a$.
By Corollary \ref{cor:inj} and Lemma \ref{recimplyweakrec},
$F_{\Zar}(S) \to F_{\Zar}(V)$ is injective for any dense open subset $V\subset S$.
Hence it suffices to verify the condition in Remark \ref{rem:reduction-to-curve}.

Put $K=k(S)$ and let $\phi:\Cb\to \Xb\times K$ and $\psi:C\to X$ be as in \ref{rem:reduction-to-curve}.
We need show 
$\langle g, \psi^*(a)\rangle_{C/K}=0\in F(K)$ for all $g\in G(\Cb,\phi^*Y))$, where $\phi^*Y$ is the pullback of 
$Y\times K$ by $\phi$.
For each $i$, we put $C_i := \phi^{-1}(U_i \times K) \subset \Cb$
and $I = \{ i ~|~ C_i \not= \emptyset \}$.
Note that $i \in I$ if and only if
$\phi(\Cb) \not\subset (\Xb - U_i) \times K$.
Since $X=\cup_{i=1}^r U_i$, we have $C = \cup_{i \in I} C_i$.
Let $i \in I$.
We write $\psi_i$ for the composition
$C_i \to U_i \times K \overset{pr}{\to}  U_i$.
We have $\psi_i^*( \rho(a_i))={\psi}^*(a)|_{C_i}$.
As $\rho(a_i)$ has modulus $Y_i$, we have
$\langle G(\Cb, \phi^* Y_i), {\psi}^*(a)|_{C_i} \rangle_{C_i/K} = 0\in F_\Zar(K)$.
Hence the claim is a consequence of the following lemma.
\end{proof}

\begin{lemma}
Let $K$ be as above and $\Cb$ be a normal integral proper curve over $K$,
$C$ an open dense subscheme of $\Cb$,
and $C=\cup_{i=1}^r C_i$ a Zariski covering.
Let $F \in \PST$ and $a \in F(C)$.
Suppose that for each $i$ we are given
an effective divisor $Y_i \subset \Cb$
satisfying $|Y_i|=\Cb-C_i$ 
and $\langle G(\Cb, Y_i), a|_{C_i} \rangle_{C_i/K}=0\in F(K)$.
Then we have
$\langle G(\Cb, Y), a \rangle_{C/K}=0$ in $F(K)$,
where $Y := Y_1 \times_{\Cb} \dots \times_{\Cb} Y_r \subset \Cb$.
\end{lemma}
\begin{proof}
This is an easy consequence of Proposition \ref{prop:loc-symbol-over-fields}.
\end{proof}


\subsection{\bf Comparison of Zariski/Nisnevich sheafification.}
\label{sect:comparison-zar-nis}

In \cite[Th. 22.2]{mvw},
it is proved that $F_{\Zar}=F_{\Nis}$ 
if $F$ is a homotopy invariant presheaf with transfers.
The proof actually shows the following:

\begin{thm}[\protect{\cite[Th. 22.2]{mvw}}]
Let $F$ be a Zariski sheaf with transfers.
If $F$ is MV-effaceable,
then one has $F=F_{\Nis}$.
\end{thm}

Theorem \ref{thm:key-tech} and \ref{thm:zar-rec} and Lemma \ref{lem:zar-rec} imply the following:

\begin{thm}\label{thm:zar-nis-rec}
Let $F \in \PST$.
\begin{enumerate}
\item
If $F$ has weak reciprocity, we have $F_{\Zar}=F_{\Nis}$.
\item
Assume $k$ is perfect.
If $F$ has reciprocity, so does $F_{\Nis}$.
\end{enumerate}
\end{thm}

\subsection{Generalized Jacobian}\label{ssect:gen-jac}

Let $\Cb$ be a smooth projective geometrically connected curve over $k$,
$D$ an effective divisor on $\Cb$ and $C:=\Cb - D$.
The map $\Z_{\tr}(C) \to \Z$ induced by
the structure map $\Cb \to \Spec k$
factors through $\deg : h(\Cb, D) \to \Z$
(see Thm. \ref{thm:main-rep}).
We write $h(\Cb, D)^0$ for its kernel.

\begin{proposition}\label{p13.1}
Suppose that $k$ is perfect and
that $C$ has a $k$-rational point.
Then Rosenlicht's generalized Jacobian $J := \Jac(\Cb, D)$
is isomorphic to the Zariski sheafification of $h(\Cb, D)^0$.
\end{proposition}

\begin{proof}
Let $s : \Z_{\tr}(J) \to J$ be the canonical map
constructed in \cite[Proof of Lemma 3.2]{spsz}.
(For $S \in \Sm$ and $Z \in \Z_{\tr}(J)(S)=\Cor(S, J)$ integral,
$s(Z)$ is given by the composition 
$S \overset{(i)}{\to} \Sym^d(Z) \overset{(ii)}{\to} \Sym^d(J) 
\overset{(iii)}{\to} J$,
where $d$ is the degree of $Z$ over $S$,
(i) is given by \cite[p.81]{sus-voe2},
(ii) is induced by $Z \to J$,
and (iii) is given by the addition map of $J$.)
Let $\tilde{a} : \Z_{\tr}(C) \to \Z_{\tr}(J)$
be the map induced by the universal map $a : C \to J$
with respect to a $k$-rational point.
Since $a$ has modulus $D$ in the sense of Rosenlicht-Serre,
Remark \ref{rem:cheating}
shows that $s \tilde{a}$ has modulus $D$ in our sense.
Thus we get an induced map $u : h(\Cb, D)^0 \to J$
(see Remark \ref{rem:universality-hxy}).
For any field $K \in \widetilde{\Sm}$,
$u(K): h(\Cb, D)^0(K) \to J(K)$ is an isomorphism
as both groups are isomorphic
to $\Div(C \times K)/G(\Cb \times K, D \times K)$.
Note that 
$J=J_{\Zar}$ and
$h(\Cb, D)^0_{\Zar}$ have reciprocity by 
Thm. \ref{thm:Algr.intro} and  \ref {thm:ZarNis.intro} (2).
Now the proposition follows from
Thm. \ref{thm:inj.intro} (3)
applied to the kernel and cokernel of $u$.
\end{proof}

%
%
%
%

\theoremstyle{plain}
\newtheorem{introthm}{Theorem}
\newtheorem{introcor}{Corollary}
\newtheorem{lem}[subsection]{Lemma}

\theoremstyle{definition}
\newtheorem{notation}[subsubsection]{Notation}

\theoremstyle{remark}
\newtheorem{rmk}[subsection]{Remark}

\numberwithin{equation}{section}

\newcommand{\eq}[2]{\begin{equation}\label{#1}#2 \end{equation}}
\newcommand{\ml}[2]{\begin{multline}\label{#1}#2 \end{multline}}
\newcommand{\mlnl}[1]{\begin{multline*}#1 \end{multline*}}
\newcommand{\ga}[2]{\begin{gather}\label{#1}#2 \end{gather}}

\newcommand{\arir}{\ar@{^{(}->}}
\newcommand{\aril}{\ar@{_{(}->}}
\newcommand{\are}{\ar@{>>}}

\newcommand{\xr}[1] {\xrightarrow{#1}}
\newcommand{\xl}[1] {\xleftarrow{#1}}
\newcommand{\lra}{\longrightarrow}

\newcommand{\ul}{\underline}
\newcommand{\ol}{\overline}
\newcommand{\bb}{\mathbb}
\newcommand{\mf}[1]{\mathfrak{#1}}
\newcommand{\mc}[1]{\mathcal{#1}}

\newcommand{\rank}{{\rm rank}}
\newcommand{\sgn}{{\rm sgn}}
\newcommand{\sHom}{{\rm \mathcal{H}om}}
\newcommand{\sExt}{{\rm \mathcal{E}xt}}
\newcommand{\Ext}{\mathrm{Ext}}
\newcommand{\im}{{\rm Im}}
\newcommand{\Sing}{{\rm Sing}}
\newcommand{\Char}{{\rm char}}
\newcommand{\Trp}{{\rm Trp}}
\newcommand{\Trf}{{\rm Trf}}
\newcommand{\Nm}{{\rm Nm }}
\newcommand{\Gal}{{\rm Gal}}
\newcommand{\ord}{{\rm ord}}
\newcommand{\va}{{\rm v}}
\newcommand{\0}{\emptyset}
\newcommand{\Ch}{\mathrm{CH}}
\newcommand{\dlog}{\mathrm{\,dlog\,}}
\newcommand{\gr}[2]{\mathrm{gr}^{#1}_{#2}}
\newcommand{\fil}{{\rm Fil }}
\newcommand{\reg}{\mathrm{reg}}
\newcommand{\trdeg}{\mathrm{trdeg}}
\newcommand{\XP}{{(X,\Phi)}}
\newcommand{\YP}{{(Y,\Psi)}}
\newcommand{\ZX}{{(Z, \Xi)}}

\newcommand{\sD}{{\mathcal D}}
\newcommand{\sE}{{\mathcal E}}
\newcommand{\sK}{{\mathcal K}}
\newcommand{\sM}{{\mathcal M}}
\newcommand{\sN}{{\mathcal N}}
\newcommand{\sQ}{{\mathcal Q}}
\newcommand{\sR}{{\mathcal R}}
\newcommand{\sS}{{\mathcal S}}
\newcommand{\sT}{{\mathcal T}}
\newcommand{\sU}{{\mathcal U}}
\newcommand{\sV}{{\mathcal V}}
\newcommand{\sW}{{\mathcal W}}
\newcommand{\sY}{{\mathcal Y}}
\newcommand{\sZ}{{\mathcal Z}}
\renewcommand{\A}{{\mathbb A}}
\newcommand{\B}{{\mathbb B}}
\newcommand{\C}{{\mathbb C}}
\newcommand{\D}{{\mathbb D}}
\newcommand{\E}{{\mathbb E}}
\renewcommand{\G}{{\mathbb G}}
\renewcommand{\H}{{\mathbb H}}
\newcommand{\J}{{\mathbb J}}
\newcommand{\M}{{\mathbb M}}
\renewcommand{\N}{{\mathbb N}}
\renewcommand{\P}{{\mathbb P}}
\renewcommand{\Q}{{\mathbb Q}}
\newcommand{\T}{{\mathbb T}}
\newcommand{\U}{{\mathbb U}}
\newcommand{\Y}{{\mathbb Y}}
\renewcommand{\Z}{\mathbb{Z}}

\newcommand{\fA}{\mathfrak{A}}
\newcommand{\fB}{\mathfrak{B}}
\newcommand{\fC}{\mathfrak{C}}
\newcommand{\fD}{\mathfrak{D}}
\newcommand{\fE}{\mathfrak{E}}
\newcommand{\fF}{\mathfrak{F}}
\newcommand{\fG}{\mathfrak{G}}
\newcommand{\fH}{\mathfrak{H}}
\newcommand{\fI}{\mathfrak{I}}
\newcommand{\fJ}{\mathfrak{J}}
\newcommand{\fK}{\mathfrak{K}}
\newcommand{\fL}{\mathfrak{L}}
\newcommand{\fN}{\mathfrak{N}}
\newcommand{\fO}{\mathfrak{O}}
\newcommand{\fP}{\mathfrak{P}}
\newcommand{\fQ}{\mathfrak{Q}}
\newcommand{\fR}{\mathfrak{R}}
\newcommand{\fS}{\mathfrak{S}}
\newcommand{\fT}{\mathfrak{T}}
\newcommand{\fU}{\mathfrak{U}}
\newcommand{\fV}{\mathfrak{V}}
\newcommand{\fW}{\mathfrak{W}}
\newcommand{\fX}{\mathfrak{X}}
\newcommand{\fY}{\mathfrak{Y}}
\newcommand{\fZ}{\mathfrak{Z}}

\newcommand{\fa}{\mathfrak{a}}
\newcommand{\fb}{\mathfrak{b}}
\newcommand{\fc}{\mathfrak{c}}
\newcommand{\fd}{\mathfrak{d}}
\newcommand{\fe}{\mathfrak{e}}
\newcommand{\ff}{\mathfrak{f}}
\newcommand{\fg}{\mathfrak{g}}
\newcommand{\fh}{\mathfrak{h}}
\newcommand{\fraki}{\mathfrak{i}}
\newcommand{\fj}{\mathfrak{j}}
\newcommand{\fk}{\mathfrak{k}}
\newcommand{\fl}{\mathfrak{l}}
\newcommand{\fo}{\mathfrak{o}}
\newcommand{\fs}{\mathfrak{s}}
\newcommand{\ft}{\mathfrak{t}}
\newcommand{\fu}{\mathfrak{u}}
\newcommand{\fv}{\mathfrak{v}}
\newcommand{\fw}{\mathfrak{w}}
\newcommand{\fx}{\mathfrak{x}}
\newcommand{\fy}{\mathfrak{y}}
\newcommand{\fz}{\mathfrak{z}}

\newpage

\appendix

\begin{center}\bf  APPENDIX: K\"AHLER DIFFERENTIALS AND DE RHAM-WITT DIFFERENTIALS HAVE RECIPROCITY\\
\end{center}
\bigskip

\begin{center} \footnotesize KAY R\"ULLING\footnote{The author is supported by the ERC Advanced Grant 226257.}
\end{center}
\bigskip

In this appendix we prove that the absolute K\"ahler differentials of any degree are reciprocity presheaves.
If the ground field $k$ is perfect ditto for the K\"ahler differentials relative to $k$. If $k$ has positive characteristic we show that the de Rham-Witt complex on a finite level of Bloch-Deligne-Illusie is a complex of reciprocity presheaves. 

\section{K\"ahler differentials }
We use the notation from Section \ref{s1}.
For a $k$-scheme $X$ we denote by $\Omega^j_{X/k}$, $j\ge 0$, the sheaf of K\"ahler differentials of degree $j$ 
relative to $k$ and by $\Omega^j_{X}=\Omega^j_{X/\Z}$ the sheaf of absolute K\"ahler differentials of degree $j$.

We recall some constructions from \cite{CR1}. 

\subsection{Pullback}\label{pb}
Let $f:X\to Y$ be a morphism of $k$-schemes then there is a natural pullback map
\[f^*: H^i(Y,\Omega^j_{Y/k}) \to H^i(X,\Omega^j_{X/k}),\quad \text{for all } i,j\ge 0, \]
which is functorial in the obvious sense.

\subsection{Pushforward}\label{pf}
Let $f:X\to Y$ be a morphism in {\bf Sm} of pure relative dimension $r$, 
$Z\subset X$ a closed subset such that $f_{|Z}:Z\to Y$ is proper and $Z'\subset Y$ a closed subset
with $Z\subset f^{-1}(Z')$. Then there is a pushforward morphism (see \cite[2.3]{CR1})
\[f_* : H^{i+r}_Z(X,\Omega^{j+r}_{X/k})\to H^i_{Z'}(Y,\Omega^j_{X/k}),\quad \text{for all } i,j\ge 0,\]
which is functorial in the obvious sense. 

\subsection{Cycle class}\label{cc}
Assume $k$ is a perfect field. 
Let $X\in {\bf Sm}$ be integral and $V\subset X$ a closed integral subscheme of codimension $c$.
Then (see e.g. \cite[Prop. 3.1.1]{CR1}) there is an element 
\[cl(V)\in H^c_V(X, \Omega^c_{X/k}),\]
which is unique with the following property: For any (or some) open subset $U\subset X$ such that the
closed immersion $i:V\cap U\inj U$ is a regular embedding of codimension $c$ the restriction of $cl(V)$ to $U$
equals the image of $1$ under the pushforward 
$i_*: H^0(V\cap U, \sO_V)\to H^c_{V\cap U}(U, \Omega^c_{U/k})$.

\subsection{Correspondence action}\label{corr}
Assume $k$ is a perfect field. Let $X, Y\in {\bf Sm}$ be equidimensional and $Z\subset X\times Y$ a 
closed integral subscheme which is {\em proper} over $X$. Set $r:=\dim Z-\dim X$.
We define the action of $Z$ on Hodge cohomology 
\[Z^*: H^{i+r}(Y,\Omega^{j+r}_{Y/k})\to H^i(X,\Omega^j_{X/k}),\quad \text{for all }i,j\ge 0,\]
as the composition
\mlnl{
H^{i+r}(Y,\Omega^{j+r}_{Y/k})  \xr{p_Y^*}H^{i+r}(X\times Y, \Omega^{j+r}_{X\times Y/k})\\
 \xr{\cup cl(Z)} H^{i+\dim Y}_{Z}(X\times Y, \Omega^{j+\dim Y}_{X\times Y/k})
\xr{p_{X*}} H^i(X,\Omega^j_{X/k}),}
where we denote by $p_X, p_Y: X\times Y \to X, Y$ the projection maps.
If $f: X\to Y$ is a morphism and $Z:=\Gamma_f\subset X\times Y$ is its graph, then 
$Z^*= f^*$, if $f$ is also proper and $Z^t\subset Y\times X$ denotes the transpose of $Z$, then
$(Z^t)^*=f_*$, see \cite[3.2.1]{CR1}.\\
\\
The following theorem is a particular case of \cite[Th. 3.1.8]{CR1} together with \cite[1.3.18, Lem. 1.3.19]{CR1}.

\begin{thm}\label{thm-PST}
Assume $k$ is a perfect field. Then
the correspondence action from \ref{corr} above induces the structure of a presheaf with transfers on the 
presheaf
\[{\bf Sm}\ni X\mapsto H^i(X,\Omega^j_{X/k})\in (k\text{-\bf vector spaces}),\]
for all $i,j\ge 0$. 
\end{thm}

\begin{rmk}
It follows from \cite[Th. 1.2.3]{CR1} and \cite[1.3.18, Lem. 1.3.19]{CR1} that any weak cohomology theory 
in the sense of \cite[1.1.9]{CR1} which satisfies the conditions of \cite[Th. 1.2.3]{CR1} 
defines a graded presheaf with transfers.
In \cite[Th. 3.1.8]{CR1} it is proven that Hodge cohomology defines such a weak cohomology theory.
\end{rmk}

\subsection{The absolute case}\label{absolute} Let $k$ be a field and $k_0\subset k$ its prime field. Let $\sI$ be the set of
smooth $k_0$-subalgebras of $k$. The ordering by inclusion makes $\sI$ a filtered partially ordered set.   
For each $X\in {\bf Sm}$  we find a $k_0$-algebra $A\in \sI$ and a smooth separated $A$-scheme $X_A$,
with $X_A\otimes_A k=X$. Fix such an $A$ for each $X$ and set  $X_B:=X_A\otimes_A B\in{\bf Sm}_{k_0}$, 
for each $B\in \sI$ containing $A$.
For $X\in {\bf Sm}$ we have
\[H^i(X,\Omega^j_X)= \varinjlim_{B\in\sI}H^i(X_B, \Omega^j_{X_B/k_0}).\]

For $Z\subset X\times Y$ as in \ref{corr} we find a $k_0$-algebra $A\in \sI$, such that there exists a closed integral
subscheme $Z_A\subset X_A\times_A Y_A$ which is proper over $X_A$, flat over $A$
and satisfies $Z_A\otimes_A k=Z$. For $B\in \sI$ with $B\supset A$ we set $Z_B:=Z_A\otimes_A B$.
Let $f: \Spec B'\to \Spec B$ be the map induced by an inclusion $B\subset B'$ in $\sI$, then
\[Z_{B'}^*\circ (\id_Y\times f)^*= (\id_X\times f)^*\circ Z_B^*:
H^{i+r}(Y_B,\Omega^{j+r}_{Y_B/k_0})\to H^i(X_{B'},\Omega^j_{X_{B'}/k_0}),\]
for all $i,j\ge 0$. (Indeed by \cite[Th. 3.1.8]{CR1} the two compositions are given by
the correspondences $(\Gamma_{\id_Y\times f}\circ Z_{B'})$ and $(Z_B\circ \Gamma_{\id_X\times f})$,
which are both equal to $Z_{B'}$ viewed as a correspondences from $X_{B'}$ to $Y_B$
via the closed immersion $X_{B'}\times_{B'} Y_{B'}=X_{B'}\times_B Y_{B}\subset X_{B'}\times_{k_0} Y_{B}$.
Here we view correspondences as elements in the Chow groups with supports, see \cite[1]{CR1}.)
Therefore we can define the action of $Z$ on absolute Hodge cohomology
\[Z^*: H^{i+r}(Y,\Omega^{j+r}_Y)\to H^i(X,\Omega^j_X)\]
by the formula
\[\varinjlim_{B\in\sI} Z_B^*: \varinjlim_{B\in \sI} H^{i+r}(Y_B,\Omega^{j+r}_{Y_B/k_0}) 
                        \to \varinjlim_{B\in \sI} H^i(X_B,\Omega^j_{X_B/k_0}).\]

We obtain:

\begin{cor}\label{cor-absolute-PST}
Let $k$ be an arbitrary field.  Then the correspondence action from \ref{absolute} above 
induces the structure of a presheaf with transfers on the 
presheaf
\[{\bf Sm}\ni X\mapsto H^i(X,\Omega^j_X)\in (k\text{-\bf vector spaces}),\]
for all $i,j\ge 0$.
\end{cor}

\begin{thm}\label{thm-Omega-RF}
The presheaf with transfers
\[{\bf Sm}\ni X\mapsto H^0(X,\Omega^i_X),\quad i\ge 0,\]
has reciprocity in the sense of Definition \ref{def.reciprocity}.
If $k$ is perfect the same is true with $\Omega^i_X$ replaced by $\Omega^i_{X/k}$.
\end{thm}

\begin{proof}
First assume that $k$ is perfect and we show that $\Omega^i_{-/k}$ has reciprocity.
Let $X\in {\bf Sm}$ be quasi-affine and take $a\in H^0(X,\Omega^i_{X/k})$.
Choose an open  immersion $X\inj \bar{X}$ of $X$ into an integral and proper $k$-scheme $\bar{X}$
such that $\bar{X}\setminus X$ is the support of a Cartier divisor $Y_0$. 
Then for some large enough integer $n$ the form $a$ is the restriction of a section  
in $H^0(\bar{X}, \Omega^i_{\bar{X}/k}(nY_0))$,
where we write $\Omega^i_{\bar{X}/k}(nY_0):=\Omega^i_{\bar{X}/k}\otimes_{\sO_{\bar{X}}}\sO_{\bar{X}}(nY_0)$.
By Corollary \ref{coro:reduction-to-cartier} it suffices to prove
\eq{thm-Omega-RF1}{ Y:= (n+1)Y_0 \text{ is a modulus for }a.}
To this end take $S\in {\bf Sm}$,  consider a diagram  $(\varphi:\bar{C}\to \bar{X}\times S)$ as in \eqref{eq:diag-modulus-cond}
(we will use the notation from \eqref{eq:diag-modulus-cond} freely) and 
a function $f\in G(\bar{C}, \gamma_{\varphi}^*Y)$. We have to show 
\eq{thm-Omega-RF2}{(\varphi_*{\rm div}_{\bar{C}}(f))^*(a)=0 \text{ in } H^0(S,\Omega^i_{S/k}).}
Clearly we can assume that $S$ is connected. Further $\Omega^i_{S/k}$ is locally free and hence restriction to 
open subsets is injective. Therefore (cf. Remark \ref{rem:reduction-to-curve}) 
we can replace $S$ by a non-empty open subset. Using the perfectness of $k$ we can thus assume that
$\bar{C}$ is smooth and connected over $k$, the map $p_\varphi: \bar{C}\to S$ is proper and 
flat of pure relative dimension 1 and the support of ${\rm div}_{\bar{C}}(f)$  
is a disjoint union of smooth prime divisors $|{\rm div}_{\bar{C}}(f)|=\sqcup_i Z_i$.  
 Set $C:=\varphi^{-1}(X\times S)$ and denote by $b\in H^0(C, \Omega^i_{C/k})$ the pullback of $a$ to
$C$. By assumption $b$ extends to a section $H^0(\bar{C}, \Omega^i_{\bar{C}/k}(n\gamma_\varphi^*Y_0))$.
We have to show
\eq{thm-Omega-RF3}{ {\rm div}_{\bar{C}}(f)^*(b)=0 \text{ in } H^0(S,\Omega^i_{S/k}),}
where we consider ${\rm div}_{\bar{C}}(f)$ as an element in ${\bf \rm Cor}(S,C)$ via the transpose of the graph map 
$C\to S\times C$ of $p_\varphi$. For $Z$ a prime divisor in the support of ${\rm div}_{\bar{C}}(f)$
denote by $i:Z\inj S\times C$ the induced closed immersion and by $p_S, p_C: S\times C\to S, C$ the projections. 
We get
\eq{thm-Omega-RF4}{Z^*(b)=p_{S*}(p_C^*(b)\cup cl(Z))= p_{S*}(p_C^*(b)\cup i_*(1))= \Tr_{Z/S}(b_{|Z}), }
where we denote by $\Tr_{Z/S}$ the pushforward along the finite morphism $Z\to S$
and the first equality holds by definition, the second by the characterization of the cycle class in \ref{cc} and 
the last equality follows from the projection formula \cite[Prop. 1.1.16]{CR1}.
Writing ${\rm div}_{\bar{C}}(f)=\sum_j n_j Z_j$ we therefore have to show
\eq{thm-Omega-RF5}{ \sum_j n_j \Tr_{Z_j/S}( b_{|Z_j})=0 \text{ in } H^0(S, \Omega^i_{S/k}).}
By the functoriality of the pushforward the map $\Tr_{Z_j/S}$ equals the composition 
 \[ H^0(Z_j,\Omega^i_{Z_j/k})\xr{i_{j*}} H^1_{D}(C, \Omega^{i+1}_{C/k})
\to H^1(\bar{C}, \Omega^{i+1}_{\bar{C}/k})\xr{p_{\varphi*}} H^0(S, \Omega^i_{S/k}),\]
here  $D:=|{\rm div}_{\bar{C}}(f)|$ and $i_j:Z_j\inj C $  is the closed immersion. We claim
\[\sum_j n_j\,i_{j*}(b_{|Z_j}) = -\delta(\tfrac{df}{f}\wedge b) \text{ in } H^1_D(C, \Omega^{i+1}_{C/k}),\]
where $\delta: H^0(C\setminus D, \Omega^{i+1}_{C/k})\to H^1_D(C, \Omega^{i+1}_{C/k})$
is the connecting homomorphism.
Indeed it suffices to check this equality after restricting to an open subset of $C$ which contains all
the generic points of $D$, in particular we can assume that the $Z_j$'s are the zero loci of sections in 
$H^0(C,\sO_C)$ and then the claim follows from \cite[Prop. 2.2.19]{CR1}.
 Since $b\in H^0(\bar{C}, \Omega^i_{\bar{C}/k}(n\gamma_\varphi^*Y_0))$ and
$f\in G(\bar{C}, (n+1)\gamma_{\varphi}^*Y_0)$ the section $\frac{df}{f}\wedge b$
extends to a section in $H^0(\bar{C}\setminus D, \Omega^{i+1}_{\bar{C}/k})$.
Therefore the image of $\delta(\tfrac{df}{f}\wedge b)$ in $H^1(\bar{C}, \Omega^{i+1}_{\bar{C}/k})$
is zero, which implies the vanishing \eqref{thm-Omega-RF5}. 

Now let $k$ be an arbitrary field with prime field $k_0$. We want to show that $\Omega^i_{-/\Z}$ has reciprocity. 
Let $X\in {\bf Sm}$ be quasi-affine and take $a\in H^0(X,\Omega^i_X)$. With the notation from
\ref{absolute} we find a smooth $k_0$-algebra $A\in \sI$ such that $a$ comes via pullback from an element
$a_A\in H^0(X_A, \Omega^i_{X_A/k_0})$. If $B\in\sI$ contains $A$ we denote by $a_B$ the pullback of $a_A$ to $X_B$.
Choose an open immersion $X_A\inj \bar{X}_A$ with $\bar{X}_A$ an integral and proper $A$-scheme such that
$\bar{X}_A\setminus X_A$ is the support of a Cartier divisor $Y_{0,A}$. There exists an integer $n\ge 0$ such that
$a_A$ is the restriction of a section in  $H^0(\bar{X}_A,\Omega^i_{\bar{X}_A/k_0}(nY_{0,A}))$.
Set $\bar{X}:= \bar{X}_A\otimes_A k$ and $Y:= (n+1) Y_{0,A}\otimes_A k$ and we claim
that $Y$ is a modulus for $a$. Take $S\in {\bf Sm}$, $(\varphi: \bar{C}\to \bar{X}\times S)$  and 
$f\in G(\bar{C},\gamma_\varphi^*Y)$ as above. Then there exists a $B\in \sI$ containing $A$ 
and $S_B\in {\bf Sm}_B$, $(\varphi_B: \bar{C}_B\to \bar{X}_B\times_B S_B)$ as in \eqref{eq:diag-modulus-cond} (only that the cartesian product is over $B$) and $f_B\in G(\bar{C}_B,\gamma_{\varphi_B}^*Y_B)$,
which give $S, \varphi, f$ when pulled back over $k$.  It suffices to show 
\[(\varphi_{B*}{\rm div}_{\bar{C}_B}(f_B))^* (a_B)=0 \text{ in } H^0(S_B,\Omega^{i}_{S_B/k_0}).\]
Notice  that this is not exactly the same situation as in the first case since $\bar{X}_B$ is not proper over $k_0$.
Nevertheless the same argument as above reduces us to prove the vanishing $\eqref{thm-Omega-RF3}$
with $S, \bar{C}, f, k$ replaced by $S_B,\bar{C}_B, f_B, k_0$, 
which follows from the first case.
\end{proof}

\section{De Rham-Witt differentials}

In this section $k$ is a perfect field of characteristic $p>0$ and we denote by ${\bf qpSm}$ the category of smooth
and quasi-projective $k$-schemes.

\subsection{De Rham-Witt complex}\label{DRW}
For a $k$-scheme $X$ we denote by $W_n\Omega^\cdot_X$ the de Rham-Witt complex of Bloch-Deligne-Illusie of length $n$,
see \cite{Il}. We denote by $W_n\Omega^j_X$ the degree $j$ part. Recall that $W_n\Omega^0_X=W_n\sO_X$ is the sheaf 
of Witt vectors of length $n$ on $X$ and $W_1\Omega^j_X=\Omega^j_X$. Also recall that the de Rham-Witt complex
comes with morphisms of sheaves of abelian groups
$R: W_{n+1}\Omega^j_X\to W_n\Omega^j_X$ (the restriction), $F: W_{n+1}\Omega^j_X\to W_n\Omega^j_X$ 
(the Frobenius), $V:W_n\Omega^j_X\to W_{n+1}\Omega^j_X$ (the Verschiebung) and 
$d: W_n\Omega^j_X\to W_n\Omega^{j+1}_X$ (the differential) satisfying various relations
(see \cite[I, Prop. 2.18]{Il}).

\subsection{Correspondence action}\label{Witt-corr}
If we restrict to the category ${\bf qpSm}$, then we have analogs of the pullback map $\ref{pb}$, the pushforward
\ref{pf} and the cycle map \ref{cc} for Hodge-Witt cohomology (i.e. in \ref{pb}-\ref{cc} replace $\Omega_{-/k}^\cdot$
by $W_n\Omega^\cdot_{-}$ and in \ref{pf} and \ref{cc} restrict to ${\bf qpSm}$). 
The construction of the cycle map and the pushforward uses essentially the results from \cite{Ek}.
For the cycle map and the pushforward for proper maps this is carried out in \cite[II]{Gros}, for the pushforward with  projective supports and the compatibilities see \cite[\S 2, 3]{CR2}.

Let $X, Y\in {\bf qpSm}$  be connected and $Z\subset X\times Y$ a closed integral subscheme which is 
{\em projective} over $X$. Set $r:=\dim Z-\dim X$. 
We define the action of $Z$ on Hodge-Witt cohomology 
\[Z^*: H^{i+r}(Y, W_n\Omega^{j+r}_Y)\to H^i(X,W_n\Omega^j_X),\quad \text{for all }i,j\ge 0,\]
as in \ref{corr} to be the composition
\mlnl{
H^{i+r}(Y,W_n\Omega^{j+r}_Y)  \xr{p_Y^*}H^{i+r}(X\times Y, W_n\Omega^{j+r}_{X\times Y})\\
 \xr{\cup cl(Z)} H^{i+d_Y}_{Z}(X\times Y, W_n\Omega^{j+d_Y}_{X\times Y})
\xr{p_{X*}} H^i(X,W_n\Omega^j_X).}
If $f: X\to Y$ is a morphism and $Z:=\Gamma_f\subset X\times Y$ is its graph, then 
$Z^*= f^*$, if $f$ is also projective and $Z^t\subset Y\times X$ denotes the transpose of $Z$, then
$(Z^t)^*=f_*$, see \cite[Prop. 3.4.7]{CR2}. 

\begin{thm}\label{thm-Witt-corr}
The correspondence action from \ref{Witt-corr} above induces the structure of a presheaf with transfers on the 
presheaf
\[{\bf qpSm}\ni X\mapsto H^i(X,W_n\Omega^j_X)\in (W_n(k)\text{-\bf modules}),\]
for all $i,j\ge 0$. Furthermore the correspondence action is compatible with the maps $R,F,V,d$ from \ref{DRW}
in the obvious sense.
\end{thm}

\begin{proof}
This is actually a special case of \cite[Th. 3.4.6]{CR2} only that there the statement is proved in the limit, i.e.
for the presheaf $X\mapsto H^i(X, W\Omega^i_X)$. The proof on the finite level goes through except for one place
in the proof of \cite[Th. 3.4.3]{CR2} (which is used in the proof of \cite[Th. 3.4.6]{CR2}), where a {\it pro-argument} 
is used. There the situation is the following: We are given closed immersions between smooth schemes 
$D'\inj D \xr{i}  Y$ with $c:=\codim(D',D)$ and $1=\codim(D,Y)$.
 Then it is shown that the pushforward $i_* :H^c_{D'}(D, W\Omega^c_D)\to 
H^{c+1}_{D'}(Y, W\Omega^{c+1}_Y)$ is injective on its Frobenius invariant submodule $H^c_{D'}(D, W\Omega^c_D)^F$.
Replace this with the following argument: By \cite[1.4, Lem. 2]{CTSS} we have a short exact sequence on $D_{\text{\'{e}t}}$
\[0\to W_n\Omega^c_{D,\log}\to W_n\Omega^c_D\xr{1-F} W_n\Omega^c_D/dV^{n-1}\Omega^{c-1}_D\to 0.\]
We get a short exact sequence
\ml{thm-Witt-corr1}{H^{c-1}_{D'}(D, W_n\Omega^c_D/dV^{n-1}\Omega^{c-1}_D)
           \to H^c_{D'}(D, W_n\Omega^c_{D,\log})\\
\to H^c_{D'}(D, W_n\Omega^c_D)^F\to 0,}
where the group on the right is defined as the kernel of $1-F$ on $H^c_{D'}$. By
\cite[I, Cor. 3.9]{Il} (and with the notations from there) we have a short exact sequence of sheaves 
of abelian groups
\[0\to \frac{\Omega^c_D}{B_n\Omega^c_D}\xr{V^{n-1}} 
\frac{W_n\Omega^c_D}{dV^{n-1}\Omega^{c-1}_D}\xr{R} W_{n-1}\Omega^c_D\to 0. \]
The two outer sheaves are Cohen-Macaulay by \cite[I, Cor. 3.9]{Il} hence so is the sheaf in the middle.
In particular the cohomology group on the left of \eqref{thm-Witt-corr1} vanishes. Thus 
\[H^c_{D'}(D, W_n\Omega^c_D)^F = H^c_{D'}(D, W_n\Omega^c_{D,\log})\cong \Z/p^n\Z,\]
where the second isomorphism is the composition of \cite[II, Th. 3.5.8]{Gros} and \cite[(3.5.19)]{Gros}, and similarly
\[H^{c+1}_{D'}(Y, W_n\Omega^{c+1}_Y)^F= H^{c+1}_{D'}(Y, W_n\Omega^{c+1}_{Y,\log})\cong \Z/p^n\Z.\]
Via this identifications $i_*$ sends $1\in \Z/p^n\Z$ to itself and hence 
\[i_*: H^c_{D'}(D, W_n\Omega^j_D)^F\inj H^{c+1}_{D'}(Y, W_n\Omega^{c+1}_D)\]
 is injective.
Now the rest of the proof of \cite[Thm 3.4.3]{CR2} and of \cite[Thm 3.4.6]{CR2} goes through.
Notice that the compatibility of the correspondence action with $R,F, V,d$ requires Ekedahl's 
notion of a Witt-dualizing system, see \cite{Ek}, \cite{CR2}.
\end{proof}

\begin{thm}\label{thm-Witt-RF}
The presheaf 
\[{\bf Sm}\ni X\mapsto H^0(X,W_n \Omega^i_X)\in (W_n(k)-\bf modules )\]
has the structure of a presheaf with transfers and has reciprocity, for all $i\ge 0$ and $n\ge 1$.
\end{thm}

\begin{proof}
Since $W_n\Omega^i_{(-)}$ is a Zariski sheaf and has transfers on smooth and quasi-affine schemes by 
Theorem \eqref{thm-Witt-corr} we can glue the transfers to obtain $W_n\Omega^i_{(-)}\in {\bf PST}$.
We prove that $W_n\Omega^i_{(-)}$ has reciprocity. Let $X\in {\bf Sm}$ be quasi-affine and 
$a\in H^0(X, W_n\Omega^i_X)$ a section. Choose a compactification $X\inj \bar{X}$ as in the proof of 
Theorem \ref{thm-Omega-RF} with $\bar{X}\setminus X$ the support of the Cartier divisor $Y_0$; we can assume that
$\bar{X}$ is projective.
Denote by $W_n\sO_{\bar{X}}(Y_0)$ the invertible $W_n\sO_{\bar{X}}$-module whose isomorphism class 
in $H^1(\bar{X}, W_n\sO_{\bar{X}}^\times)$ is the image of the class of $\sO_{\bar{X}}(Y_0)$ in 
$H^1(\bar{X}, \sO_{\bar{X}}^\times)$ under the map
induced by the Teichm\"uller lift $[-]:H^1(\bar{X}, \sO_{\bar{X}}^\times)\to H^1(\bar{X}, W_n\sO_{\bar{X}}^\times)$.
More precisely, if $t$ is a local coordinate of $Y_0$ on some open $U\subset \bar{X}$ then
$W_n\sO_{\bar{X}}(Y_0)_{|U}= W_n(\sO_U)\cdot \frac{1}{[t]}$. In particular 
\[W_n\Omega^i_X=\varinjlim_r W_n\Omega^i_{\bar{X}}(rY_0),\]
where we set 
$W_n\Omega^i_{\bar{X}}(rY_0):=W_n\Omega^i_{\bar{X}}\otimes_{W_n\sO_{\bar{X}}}W_n\sO_{\bar{X}}(rY_0).$
Hence we find an integer $r\ge 1$ such that $a$ is the restriction of a section in 
$H^0(\bar{X}, W_n\Omega^i_{\bar{X}}(rY_0))$. By Corollary \ref{coro:reduction-to-cartier} it suffices to prove
\eq{thm-Witt-RF1}{Y:=sY_0 \text{ is a modulus for }a, \text{ for } s\ge p^{n-1}r+1.}
To this end take $S\in {\bf Sm}$ and $(\varphi: \bar{C}\to \bar{X}\times S)$ as in \eqref{eq:diag-modulus-cond} and a
function $f\in G(\bar{C}, \gamma_\varphi^*Y)$. By \cite[I, Cor. 3.9]{Il} restriction to dense open subsets is injective
on $W_n\Omega^i_S$. Thus as in the proof of Theorem \ref{thm-Omega-RF} we can assume that
$S, \varphi, f$ have the same properties as in the proof of Theorem \ref{thm-Omega-RF}  
between \eqref{thm-Omega-RF2} and \eqref{thm-Omega-RF3}; further we can achieve that all schemes in question are 
quasi-projective. Denote by $b$ the pullback of $a$ to $C$.
By assumption $b$ extends to a section of $H^0(\bar{C}, W_n\Omega^i_{\bar{C}}(r\gamma_\varphi^*Y_0))$.
As in {\it ibid.} we are reduced to show that
\eq{thm-Witt-RF2}{\sum_j n_j\Tr_{Z_j/S}(b_{|Z_j})=0 \text{ in } H^0(S,W_n\Omega^i_S),}
where we write ${\rm div}_{\bar{C}}(f)=\sum_j n_j Z_j$ and 
$\Tr_{Z_j/S}: H^0(Z_j, W_n\Omega^i_{Z_j})\to H^0(S, W_n\Omega^i_S)$ is the pushforward along the finite map 
$Z_j\to S$. By the functoriality of the pushforward the map $\Tr_{Z_j/S}$ equals the composition 
 (with $D=|{\rm div}_{\bar{C}}(f)|$ and $i_j:Z_j\inj C $  the closed immersion)
\mlnl{ H^0(Z_j,W_n\Omega^i_{Z_j})\xr{i_{j*}} H^1_{D}(C, W_n\Omega^{i+1}_C)
\to H^1(\bar{C}, W_n\Omega^{i+1}_{\bar{C}})\\ 
\xr{p_{\varphi*}} H^0(S, W_n\Omega^i_S).}
As in the proof of Theorem \ref{thm-Omega-RF} the following equality follows from
 \cite[Prop. 2.4.1]{CR2} (see also \cite[II, 3.4]{Gros}) 
\[\sum_j n_j\,i_{j*}(b_{|Z_j}) = -\delta(\tfrac{d[f]}{[f]} b) \text{ in } H^1_D(C, W_n\Omega^{i+1}_C),\]
where $\delta: H^0(C\setminus D, W_n\Omega^{i+1}_C)\to H^1_D(C, W_n\Omega^{i+1}_C)$
is the connecting homomorphism and $[f]\in W_n(k(\bar{C}))$ denotes the Teichm\"uller lift of $f$.
Thus it suffices to show that  $\frac{d[f]}{[f]}b$ extends to a
section of $H^0(\bar{C}\setminus D, W_n\Omega^{i+1}_{\bar{C}})$ to conclude 
the vanishing \eqref{thm-Witt-RF2}.
To this end it suffices to show that  $d([f]) b$ is regular around any point of $Y_0$. 
Let $A$ be the local ring of $\bar{C}$ at a point of $Y_0$
and $t\in A$ an equation for $Y_0$. Then we can write $b= b_0/[t]^r$, for some 
$b_0\in W_n\Omega^i_A$, and $f=1+t^{s}g$, for some 
$g\in A$.  By \cite[Lem. 3.4]{R} 
\[ [f]= [1]+\sum_{j=0}^{n-1}V^j([t]^{s}g_j)\quad \text{in } W_n(A), \text{ for some }g_j\in W_{n-j}(A).\]
For $j\in \{0, \ldots, n-1\}$ and $h=g_j$ we have (using the standard identities for $R, F, V, d $)
\begin{align}
d(V^j([t]^{s}h)) b & = d(V^j([t]^s h) b) -  V^j([t]^s h) db\notag\\
           & = dV^j([t]^s h F^j(b))- 
     V^j([t]^s h) \left(\frac{d b_0}{[t]^r}- r b_0 \frac{1}{[t]^r}\frac{d[t]}{[t]} \right)\notag\\
        & = dV^j([t]^{s-p^j r}h F^j(b_0)) - V^j([t]^{s-p^jr}h F^j(db_0))\notag\\
         &  \phantom{=}       +      V^j([t]^{s-(p^j r+1)}h r F^j(b_0) d[t]).\notag
\end{align}
This expression is regular by the choice of $s$ and hence so is $d([f]) b$. This finishes the proof.
\end{proof}

\begin{rmk}
One can remove the perfectness assumption on $k$ in this section using the same method as in \ref{absolute}.
\end{rmk}

\end{document}